%% file: paper.tex
\documentclass{article}

\usepackage[accepted]{icml2023}

\usepackage[hidelinks]{hyperref}
\usepackage{url}

\include{macros}

\newcommand{\squeeze}{}

\usepackage{tikz}
\usetikzlibrary{tikzmark,decorations.pathreplacing}

\hypersetup{
  colorlinks   = true, 
  urlcolor     = blue, 
  linkcolor    = {red!75!black}, 
  citecolor   = {blue!50!black} 
}

\usepackage{microtype}
\usepackage{graphicx}
\usepackage{booktabs} 

\usepackage{hyperref}

\usepackage{amsmath}
\usepackage{amssymb}
\usepackage{mathtools}
\usepackage{amsthm}

\usepackage[capitalize,noabbrev]{cleveref}

\usepackage[textsize=tiny]{todonotes}
\usepackage{tcolorbox}
\usepackage{pifont}
\definecolor{mydarkgreen}{RGB}{39,130,67}
\definecolor{mydarkorange}{RGB}{39,130,67}
\definecolor{mydarkred}{RGB}{192,47,25}

\newcommand{\orange}{\color{mydarkorange}}

\newcommand{\algnamebig}[1]{{\orange \sf #1}}
\newcommand{\algname}[1]{{\orange\small \sf #1}}
\newcommand{\algnamesmall}[1]{{\orange\scriptsize \sf #1}}

\newcommand{\wtos}{\textnormal{w}}
\newcommand{\stow}{\textnormal{s}}

\allowdisplaybreaks

\icmltitlerunning{EF21-P and Friends}

\begin{document}

\twocolumn[
\icmltitle{EF21-P and Friends: Improved Theoretical Communication Complexity \\ for Distributed Optimization with Bidirectional Compression}



\icmlsetsymbol{zzz}{*}

\begin{icmlauthorlist}
\icmlauthor{Kaja Gruntkowska}{zzz,yyy}
\icmlauthor{Alexander Tyurin}{yyy}
\icmlauthor{Peter Richt\'{a}rik}{yyy}
\end{icmlauthorlist}

\icmlaffiliation{yyy}{King Abdullah University of Science and Technology, Thuwal, Saudi Arabia}

\icmlcorrespondingauthor{Alexander Tyurin}{alexandertiurin@gmail.com}


\vskip 0.3in
]


\printAffiliationsAndNotice{\textsuperscript{*}The work of Kaja Gruntkowska was performed during a Summer research internship in the Optimization and Machine Learning Lab at KAUST led by Peter Richt\'{a}rik. Kaja Gruntkowska is an undergraduate student at the University of Warwick, United Kingdom.} 

\begin{abstract}
    
    In this work we focus our attention on distributed optimization problems in the context where the communication time between the server and the workers is non-negligible. We obtain novel methods supporting bidirectional compression (both from the server to the workers and vice versa) that enjoy new state-of-the-art theoretical communication complexity for convex and nonconvex problems. Our bounds are the first that manage to decouple the  variance/error coming from the workers-to-server and server-to-workers compression, transforming a multiplicative dependence to an additive one. Moreover, in the convex regime, we obtain the first bounds that match the theoretical communication complexity of gradient descent. Even in this convex regime, our algorithms work with biased gradient estimators, which is non-standard and requires new proof techniques that may be of  independent interest. Finally, our theoretical results are corroborated through suitable experiments.
\end{abstract}

\section{Distributed Optimization and Bidirectional Compression}
In this paper, we consider distributed optimization problems in strongly convex, convex and nonconvex settings. Such problems arise in federated learning \citep{konevcny2016federated, mcmahan2017communication} and in deep learning \citep{ramesh2021zero}. In federated learning, a large number of workers/devices/nodes contain local data and communicate with a parameter-server that performs optimization of a function in a distributed fashion \citep{ramaswamy2019federated}. Due to privacy concerns and the potentially large number of workers, the communication between the workers and the server is a bottleneck and requires specialized algorithms capable of reducing the communication overhead. Popular algorithms dealing with these kinds of problems are based on communication compression \citep{DIANA,EF21,tang2019doublesqueeze}. 

We consider the distributed optimization problem 
\begin{align}
   \label{eq:main_task}
 \squeeze  \min \limits_{x \in \R^d} \left\{f(x) \eqdef \frac{1}{n} \sum\limits_{i=1}^n f_i(x)\right\},
\end{align}
where $n$ is the number of workers, and $f_i \,:\, \R^d \rightarrow \R$ are smooth (possibly nonconvex) functions for all $i \in [n] \eqdef \{1, \dots, n\}.$ We assume that the functions $f_i$ are stored on $n$ workers. Each of them is directly connected to a server that orchestrates the work of the devices \citep{kairouz2021advances}, i.e., the workers perform some calculations and send the results to the server, after which the server does calculations and sends the results back to the workers and the whole process repeats. 

\begin{figure}
\centering
        \includegraphics[width=0.6\linewidth]{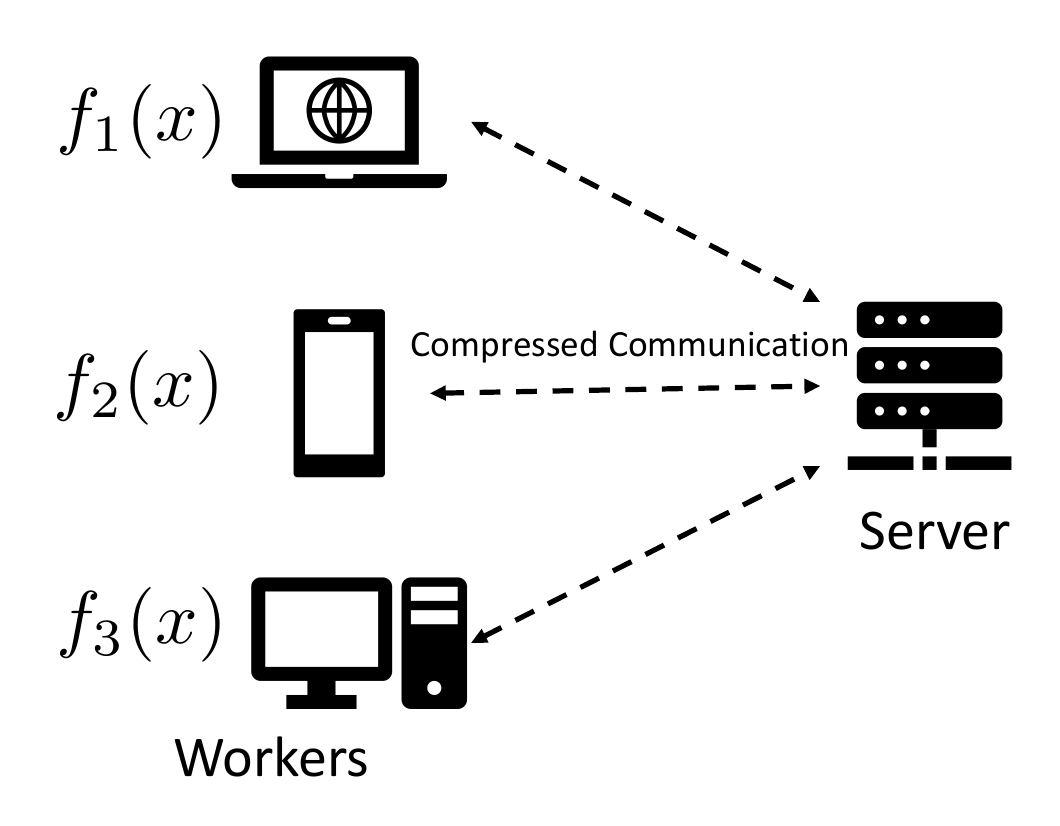}
        \caption{Distributed optimization with bidirectionally compressed communication.}
\end{figure}

\subsection{Assumptions}
Throughput the work we will refer to a subset of these assumptions:
\begin{assumption}
    \label{ass:lipschitz_constant}
    The function $f$ is $L$--smooth, i.e., $$\norm{\nabla f(x) - \nabla f(y)} \leq L \norm{x - y}, \quad \forall x, y \in \R^d.$$
 \end{assumption}
 \begin{assumption}
    \label{ass:workers_lipschitz_constant}
    The functions $f_i$ are $L_i$--smooth for all $i \in [n]$, i.e.,
    $$\norm{\nabla f_i(x) - \nabla f_i(y)} \leq L_i \norm{x - y}, \quad \forall x, y \in \R^d.$$
Let     $L_{\max} \eqdef \max_{i \in [n]} L_i.$    Further,  let $\widehat{L}^2$ be a constant such that
    $$\frac{1}{n} \sum_{i=1}^n \norm{\nabla f_i(x) - \nabla f_i(y)}^2 \leq \widehat{L}^2 \norm{x - y}^2, \forall x, y \in \R^d.$$
    
 \end{assumption}
 Note that if the functions $f_i$ are $L_i$--smooth for all $i \in [n],$ then there exists $\widehat{L}$ such that $\widehat{L} \leq L_{\max}.$
 \begin{assumption}
    \label{ass:convex}
    The functions $f_i$ are convex for all $i \in [n]$. Further, the function $f$ is $\mu$-strongly convex with $\mu \geq 0$, and attains a minimum at some point $x^* \in \R^d.$  
 \end{assumption}
 
     To avoid ambiguity, the constants $L$, $\widehat{L},$ and $L_i$ are the smallest such numbers.

 \begin{restatable}{lemma}{LEMMALIPTCONSTANTS}
    \label{lemma:lipt_constants}
If Assumptions~\ref{ass:lipschitz_constant}, \ref{ass:workers_lipschitz_constant} and \ref{ass:convex} hold, then $\widehat{L} \leq L_{\max} \leq n L$ and $L \leq \widehat{L} \leq \sqrt{L_{\max} L}.$
\end{restatable}

\subsection{Communication complexity of vanilla gradient descent}

Solving the aforementioned optimization problem involves two key steps: i) the workers send results to the server (server-to-workers communication), ii) the server sends results to the workers (workers-to-server communication). Let us first consider how this procedure works in the case of  \algname{GD}:
\begin{align*}
  \squeeze  x^{t+1} = x^{t} - \gamma \nabla f(x^{t}) = x^{t} -  \frac{\gamma}{n}\sum\limits_{i=1}^n\nabla f_i(x^{t}).
\end{align*}
It is well known that if the function $f$ is $L$-smooth and $\mu$-strongly convex (see Assumptions~\ref{ass:lipschitz_constant} and \ref{ass:convex}), then \algname{GD} with stepsize $\gamma = \nicefrac{1}{L}$ returns an $\varepsilon$-solution after $\cO\left(\nicefrac{L}{\mu}\log\nicefrac{1}{\varepsilon}\right)$ steps. In distributed setting, \algname{GD} would require i) the workers to send $\nabla f_i(x^{t})$ to the server ii) the server to send $x^{t+1}$ to the workers or, alternatively, ii) the server to send $\frac{1}{n}\sum_{i=1}^n\nabla f_i(x^{t})$ to the workers, depending on whether the iterates $x^t$ are updated on the server or on the workers. Assuming that the communication complexity is proportional to the number of coordinates, the server-to-workers and workers-to-server communication complexities are equal $\cO\left(\nicefrac{dL}{\mu}\log\nicefrac{1}{\varepsilon}\right).$

\subsection{Workers-to-server (=uplink) compression}
We now move on to more advanced algorithms that aim to improve the workers-to-server communication complexity. These algorithms assume that the server-to-workers communication complexity is negligible and focus exclusively on sending the message from devices to the server. Such an approach can be justified by the fact that broadcast operation may in some systems be much faster than gather operation \citep{DIANA, kairouz2021advances}. Moreover, the server can be considered to be just an abstraction representing ``all other nodes'', in which case server-to-worker communication does not exist at all.

The primary tools that help reduce communication cost are compression operators, such as vector sparsification and quantization~\citep{alistarh2017qsgd,beznosikov2020biased}. The literature distinguishes two main classes of such operators: biased and unbiased compressors. In particular, we say that:
\begin{definition}
    \label{def:biased_compression}
    A (possibly) stochastic mapping $\cC\,:\,\R^d \rightarrow \R^d$ is a \textit{biased compressor} if
    there exists $\alpha \in (0,1]$ such that
    \begin{align}
        \label{eq:biased_compressor}
        \qquad \Exp{\norm{\cC(x) - x}^2} \leq (1 - \alpha) \norm{x}^2, \,\, \forall x \in \R^d.
    \end{align}
 \end{definition}
 
 \begin{definition}
    \label{def:unbiased_compression}
    A stochastic mapping $\cC\,:\,\R^d \rightarrow \R^d$ is an \textit{unbiased compressor} if
    there exists $\omega \geq 0$ such that
    \begin{align}
        \label{eq:compressor}
        \Exp{\cC(x)} = x, \,\, \Exp{\norm{\cC(x) - x}^2} \leq \omega \norm{x}^2, \,\,\forall x \in \R^d.
    \end{align}
 \end{definition}

 We denote the collections of mappings satisfying Definition~\ref{def:biased_compression} and \ref{def:unbiased_compression} by $\mathbb{B}(\alpha)$ and $\mathbb{U}(\omega)$ respectively. One can easily show that if $\cC \in \mathbb{U}(\omega),$ then $(\omega + 1)^{-1} \cC \in \mathbb{B}\left((\omega + 1)^{-1}\right)$, meaning that the family of biased compressors is wider. Two canonical examples of compressors belonging to these two classes are  the Top$K \in \mathbb{B}(\nicefrac{K}{d})$ and Rand$K \in \mathbb{U}(\nicefrac{d}{K} - 1)$ sparsifiers. The former retains the $K$ largest values of the input vector, while the latter takes $K$ random values of this vector scaled by $\nicefrac{d}{K}$ \citep{beznosikov2020biased}. Further examples of compressors belonging to $\mathbb{B}(\alpha)$ and $\mathbb{U}(\omega)$ can be found in \citep{beznosikov2020biased}.

The theory of methods supporting workers-to-server compression is reasonably well developed. In the convex and strongly convex setting, the current state-of-the-art methods are \algname{DIANA} \citep{DIANA}, \algname{ADIANA} \citep{ADIANA}, and \algname{CANITA} \citep{li2021canita}.
In the nonconvex setting, the current state-of-the-art methods are \algname{DCGD} \citep{khaled2020better} (in the low accuracy regime) and \algname{MARINA}, \algname{DASHA}, \algname{FRECON}, and \algname{EF21} \citep{MARINA, tyurin2022dasha, tyurin2022computation, zhao2021faster, EF21} (in the high accuracy regime).

To see that these types of algorithms can achieve workers-to-server communication complexity that is no worse than that of \algname{GD}, let us consider the \algname{DIANA} method.
In the strongly convex case, \algname{DIANA} \citep{khaled2020unified} has the convergence rate $$\cO\left(\left(\left(1 + \frac{\omega}{n}\right)\frac{L_{\max}}{\mu} + \omega\right) \log \frac{1}{\varepsilon}\right).$$ Using the Rand$K$ compression operator with $K = \nicefrac{d}{n}$,  the workers-to-server complexity is not greater than 
\begin{align*} \squeeze &\cO\left(\frac{d}{n} \times \left(\left(1 + \frac{\omega}{n}\right)\frac{L_{\max}}{\mu} + \omega\right) \log \frac{1}{\varepsilon}\right) \\
    &= \cO\left( \left(\frac{d L_{\max}}{n \mu} + d\right) \log \frac{1}{\varepsilon}\right), 
\end{align*}
meaning that \algname{DIANA}'s complexity is better than \algname{GD}'s complexity $\cO\left(\nicefrac{d L}{\mu}\log\nicefrac{1}{\varepsilon}\right)$ (recall that $L_{\max} \leq n L$). The same reasoning applies to other algorithms in the convex and nonconvex worlds.

\subsection{Bidirectional compression}
\label{sec:bidirection_compression}
In the previous section, we showed that it is possible to improve workers-to-server communication complexity of \algname{GD}. But what about the server-to-workers compression? Does there exist a method that would also compress the information sent from the server to the workers and obtain the workers-to-server and server-to-workers communication complexities at least as good as with the vanilla \algname{GD} method? As far as we know, the current answer to the question is NO!

Bidirectional compression has been considered in many papers, including \citep{horvath2019natural, tang2019doublesqueeze,  liu2020double, philippenko2020artemis, philippenko2021preserved,fatkhullin2021ef21}. In Table~\ref{table:strongly_convex_case}, we provide a comparison of methods applying this type of compression in the strongly convex setting. 
Let us now take a closer look at the \algname{MCM} method of \citet{philippenko2021preserved}. For simplicity, we assume that the server and the workers use Rand$K$ compressors with parameters $K_{{\stow}}$ and $K_{{\wtos}}$, respectively. The server-to-workers communication complexity of \algname{MCM} is not less than
\begin{align*} 
 \squeeze &\Omega\left(K_{{\stow}} \times \left(1 + \omega_{{\stow}}^{3/2} + \frac{\omega_{\textnormal{\stow}} \omega_{\textnormal{\wtos}}^{1/2}}{\sqrt{n}} + \frac{\omega_{\textnormal{\wtos}}}{n}\right) \frac{L_{\max}}{\mu} \log \frac{1}{\varepsilon}\right)  \\
 &= \Omega\left(\frac{d^{3/2}}{K_{\textnormal{\stow}}^{1/2}} \frac{L_{\max}}{\mu} \log \frac{1}{\varepsilon}\right).
\end{align*}
Thus, for any $K_{\textnormal{\stow}} \in [1, d]$, the server-to-workers communication complexity is worse than the \algname{GD}'s complexity $\cO\left(\frac{d L}{\mu}\log\nicefrac{1}{\varepsilon}\right)$ by a factor of $\nicefrac{d^{1/2}}{K_{\textnormal{\stow}}^{1/2}}$. The same reasoning applies to \algname{Dore} \citep{liu2020double} and \algname{Artemis} \citep{philippenko2020artemis}:
\begin{align*} 
 \squeeze  \Omega\left(K_{{\stow}} \left(\frac{\omega_{\textnormal{\stow}}\omega_{\textnormal{\wtos}}}{n}\right) \frac{L_{\max}}{\mu} \log \frac{1}{\varepsilon}\right)  = \Omega\left(\frac{d^2}{K_{{\wtos}} n} \frac{L_{\max}}{\mu} \log \frac{1}{\varepsilon}\right).
\end{align*}
It turns out that one can find an example of problem \eqref{eq:main_task} with $L_{\max} = n L.$ Therefore, in the worst case scenario, the server-to-workers communication complexity can be up to $\nicefrac{d}{K_{{\wtos}}}$ times worse than the \algname{GD}'s complexity for any $K_{{\wtos}} \in [1, d].$

\begin{table*}
    \caption{\textbf{Strongly Convex Case.} The number of communication rounds to get an $\varepsilon$-solution (${\rm E}[\norm{\widehat{x} - x^*}^2] \leq \varepsilon$) up to logarithmic factors. To make comparison easier, if a method works with a biased compressor, we assume that the biased compressor is formed from the unbiased compressors and the following relations hold: $\omega_{\wtos} + 1 = \nicefrac{1}{\alpha_{\wtos}}$ and $\omega_{\stow} + 1 = \nicefrac{1}{\alpha_{\stow}},$ where $\omega_{\wtos}$ and $\omega_{\stow}$ are parameters of workers-to-server and server-to-workers compressors, accordingly.}
    \label{table:strongly_convex_case}
    \centering 
    \begin{threeparttable}
      \begin{tabular}{cccccc}
\toprule
     \bf  Method & \bf \# Communication Rounds& 
      \bf Limitations \\
        \midrule
        \makecell{\algnamesmall{EF} \\
        \citep{Seide2014} \\ \citep{EC-SGD}} &   $\Omega\left(\left(1 + \omega_{\wtos}\right)\frac{L_{\max}}{\mu}\right)$ & \makecell{No server-to-worker  compression.} \\
       \midrule
       \makecell{\algnamesmall{DIANA} \\
       \citep{DIANA}} &   $\left(1 + \frac{\omega_{\wtos}}{n}\right)\frac{L_{\max}}{\mu} + \omega_{\wtos}$ & \makecell{No server-to-worker  compression.} \\
       \midrule
      \makecell{\algnamesmall{Dore}, \algnamesmall{Artemis}, \algnamesmall{MURANA} \\
      \citep{liu2020double} \\
      \citep{philippenko2020artemis} \\ \citep{condat2022murana}} & $\Omega\left(\frac{\omega_{\textnormal{\stow}}\omega_{\textnormal{\wtos}}}{n}\frac{L_{\max}}{\mu}\right)$ & --- \\
       \midrule
       \makecell{\algnamesmall{MCM} \\ \citep{philippenko2021preserved}} & $\Omega\left(\left(\omega_{{\stow}}^{3/2} + \frac{\omega_{\textnormal{\stow}} \omega_{\textnormal{\wtos}}^{1/2}}{\sqrt{n}} + \frac{\omega_{\textnormal{\wtos}}}{n}\right)\frac{L_{\max}}{\mu}\right)$ & --- \\
      \midrule
    \cellcolor{bgcolor1}   \begin{tabular}{c} \algnamesmall{EF21-P + DIANA} (new) \\ (Theorem~\ref{theorem:diana_strong}) \end{tabular} & \cellcolor{bgcolor1} $(1 + \omega_{\textnormal{\stow}})\frac{L}{\mu} + \frac{\omega_{\wtos}}{n}\frac{L_{\max}}{\mu} + \omega_{\wtos}$ & \cellcolor{bgcolor1}  --- \\
      \midrule
 \cellcolor{bgcolor1}       \begin{tabular}{c}\algnamesmall{EF21-P + DCGD} (new) \\ (Theorem~\ref{theorem:dcgd_strong})\end{tabular} & \cellcolor{bgcolor1} $(1 + \omega_{\textnormal{\stow}})\frac{L}{\mu} + \frac{\omega_{\wtos}}{n}\frac{L_{\max}}{\mu}$ & 
 \cellcolor{bgcolor1}       \begin{tabular}{c}Interpolation regime: \\
      $\nabla f_i(x^*) = 0$\end{tabular} \\
      \bottomrule
      \end{tabular}
  \end{threeparttable}
 \end{table*}

\section{\algnamebig{EF21-P}: A Useful Reparameterization of the Classical \algnamebig{EF} Mechanism}
\label{sec:ef21_p}
Before we continue discussing bidirectional methods and our contributions, let us remark on the key moment which ultimately enabled the main results of this paper.
Consider solving the optimization problem
\begin{equation}\label{eq:min_f}\squeeze \min \limits_{x\in \R^d} f(x),\end{equation}
where $f:\R^d\to \R$ is a smooth but not necessarily convex function.
We now introduce a technique which we call \algname{EF21-P} that performs error-feedback updates in the primal space of the iterates/models \footnote{\algnamesmall{EF21-P} is initially inspired by the recently proposed error-feedback mechanism, \algnamesmall{EF21}, of \citet{EF21}, which compresses the dual vectors, i.e., the gradients. \algnamesmall{EF21} is currently the state-of-the-art error feedback mechanism in terms of its theoretical properties and practical performance~\citep{fatkhullin2021ef21}. If we wish to explicitly highlight its dual nature, we could instead meaningfully call their method \algnamesmall{EF21-D}.}. 
 Given a contractive compression operator $\cC \in \mathbb{B}(\alpha)$ from Definition~\ref{def:biased_compression}, \algname{EF21-P} method aims to solve \eqref{eq:min_f} via the iterative process
\begin{equation}
\begin{aligned}
   \label{eq:ef21_primal}
   x^{t+1} &= x^t - \gamma \nabla f(w^t), \\
  w^{t+1} &= w^t + \cC^t(x^{t+1} - w^t), 
\end{aligned}
\end{equation}
where $\gamma>0$ is a stepsize,  $x^0 \in \R^d$ is the initial iterate, $w^0=x^0\in \R^d$ is the initial iterate {\em shift},  and  $\cC^t$ is an instantiation of a randomized contractive compressor $\cC$ sampled at time $t$. If $f$ is $L$-smooth and $\mu$-strongly convex, we prove that both $x^t$ and $w^t$ converge to $x^* = \arg\min f$ at a linear rate, in $\cO\left(\nicefrac{L}{\alpha \mu} \log \nicefrac{1}{\varepsilon}\right)$ iterations in expectation (see Section~\ref{sec:EF21-P_convex}).

Surprisingly, it turns out that \algname{EF21-P} is equivalent to the classical \algname{EF} mechanism of \citet{Seide2014} under an appropriate reparameterization of the iterates $x^t$ and $w^t.$ In particular, by taking $e^t \eqdef x^t - w^t$ in \eqref{eq:ef21_primal}, we have
\begin{equation}
\begin{aligned}
    \label{eq:ef_classical}
    w^{t+1} &= w^t + \cC^t(e^{t} - \gamma \nabla f(w^t)), \\
    e^{t+1} &= e^t - \cC^t(e^{t} - \gamma \nabla f(w^t)) - \gamma \nabla f(w^t).
\end{aligned}
\end{equation}
The procedure \eqref{eq:ef_classical} is the classical \algname{EF} mechanism (up to $\pm$ signs) studied in many papers \cite{Seide2014, Koloskova2019-DecentralizedEC-2019, gorbunov2020linearly}.

Looking ahead, \algname{EF21-P} \eqref{eq:ef21_primal} turns out to be an extremely useful ``form'' of \algname{EF} \eqref{eq:ef_classical}. Indeed, as we shall see, when combined with suitable methods performing worker-to-server compression, \algname{EF21-P} leads to new state-of-the-art theoretical communication complexities! This reparameterization is an essential component of our proofs since they explicitly use the iterates $x^t$ from \eqref{eq:ef21_primal} which are not defined in \eqref{eq:ef_classical}.

In its vanilla form, \algname{EF21-P} is not the main focus of this paper. Instead, we use it as an important ingredient in the design of more elaborate algorithms. Namely, we exploit \algname{EF21-P} as the mechanism for compressing and subsequently error-correcting the model broadcast by the server to the workers (=downlink compression).

\section{Contributions}
By combining \algname{EF21-P} with suitable methods (``friends'' in the title of the paper) performing worker-to-server compression, in particular, \algname{DIANA} \citep{DIANA,horvath2019stochastic} or \algname{DCGD} \citep{alistarh2017qsgd, DCGD}, we obtain methods, suggestively named  \algname{EF21-P + DIANA} (Algorithm~\ref{algorithm:diana_ef21_p}) and \algname{EF21-P + DCGD} (Algorithm~\ref{algorithm:dcgd_ef21_p}),  both supporting bidirectional compression, and both  enjoying  new state-of-the-art theoretical communication complexity for convex and nonconvex problems. 
\begin{algorithm*}[t]
    \caption{\algname{EF21-P + DIANA}}
    \footnotesize
    \begin{algorithmic}[1]
    \label{algorithm:diana_ef21_p}
    \STATE \textbf{Parameters:} learning rates $\gamma > 0$ {\color{gray}(for learning the model)} and $\beta > 0$ {\color{gray}(for learning the gradient shifts);} initial model $x^0 \in \R^d$ {\color{gray}(stored on the server and the workers);}  initial gradient shifts $h^0_1, \dots , h^0_n\in \R^d$ {\color{gray}(stored on the workers);} average of the initial gradient shifts  $h^0 = \frac{1}{n}\sum_{i = 1}^n h_i^0$  {\color{gray}(stored on the server);} initial model shift $w^0 = x^0\in \R^d$ {\color{gray}(stored on the server and the  workers)} 
    \FOR{$t = 0, 1, \dots, T - 1$} 
    \FOR{$i = 1, \dots, n$ in parallel} 
    \STATE $m_i^t = \cC_i^{D}(\nabla f_i(w^t) - h_i^t)$  \hfill{\scriptsize \color{gray} Worker $i$ compresses the shifted gradient via the dual compressor $\cC_i^{D} \in \mathbb{U}(\omega)$}
    \STATE {\color{blue}Send compressed message $m_i^t$ to the server}
    \STATE $h^{t+1}_i = h^t_i + \beta m^t_i$ \hfill{\scriptsize \color{gray} Worker $i$ updates its local gradient shift with stepsize $\beta$}
    \ENDFOR
    \STATE $m^t = \frac{1}{n}\sum_{i = 1}^n m_i^t$ \hfill{\scriptsize \color{gray}Server averages the $n$ messages received from the workers}
    \STATE $h^{t+1} = h^t + \beta m^t$ \hfill{\scriptsize \color{gray} Server updates the average gradient shift so that $h^t = \frac{1}{n}\sum_{i = 1}^n h_i^t$}
    \STATE $g^{t} = h^t + m^t$ \hfill{\scriptsize \color{gray} Server computes the gradient estimator}
    \STATE $x^{t+1} = x^t - \gamma g^t$ \hfill{\scriptsize \color{gray} Server takes a gradient-type step with stepsize $\gamma$}
    \STATE $p^{t+1} = \cC^{P}\left(x^{t+1} - w^t\right)$  \hfill {\scriptsize \color{gray} Server compresses the shifted model   via the primal compressor $\cC^{P} \in \mathbb{B}\left(\alpha\right)$}
    \STATE $w^{t+1} = w^t + p^{t+1}$\hfill{\scriptsize \color{gray}Server updates the model shift}
    \STATE {\color{blue}Broadcast compressed message $p^{t+1}$ to all $n$ workers}
    \FOR{$i = 1, \dots, n$ in parallel}
    \STATE $w^{t+1} = w^{t} + p^{t+1}$ \hfill{\scriptsize \color{gray} Worker $i$ updates its local copy of the model shift}
    \ENDFOR
    \ENDFOR
    \end{algorithmic}
\end{algorithm*}
    
 $\diamond$ {\bf Convex setting.} \algname{EF21-P + DIANA} provides new state-of-the-art convergence rate for distributed optimization  in the strongly convex (see Table~\ref{table:strongly_convex_case}) and general convex regimes.
 This is the first method supporting bidirectional compression whose server-to-workers and workers-to-server communication complexity is no worse than that of vanilla \algname{GD}. When the workers calculate stochastic gradients (see Section~\ref{sec:StochGrad}), we prove that \algname{EF21-P + DIANA} improves the rates of prior methods. Further,  we prove that \algname{EF21-P + DCGD} has an even better convergence rate than \algname{EF21-P + DIANA} in the interpolation regime (see Section~\ref{sec:interpol}). 
 
 $\diamond$ {\bf Nonconvex setting.} In the nonconvex setting (see Section~\ref{sec:nonconvex}),  \algname{EF21-P + DCGD} is the first method supporting bidirectional compression whose convergence rate  {\em decouples the noise} coming from the workers-to-server and server-to-workers compression, respectively, from a multiplicative to an {\em additive} dependence (see Table~\ref{table:general_case}). Moreover, \algname{EF21-P + DCGD} provides the new state-of-the-art convergence rate in the low accuracy regimes ($\varepsilon$ is small or the \# of workers $n$ is large). Further, we provide examples of optimization problems where \algname{EF21-P + DCGD} outperforms previous state-of-the-art methods even in the high accuracy regime.

 $\diamond$  {\bf Unified \algname{SGD} analysis framework with the \algname{EF21-P} mechanism.} 
 \citet{khaled2020better} provide a unified framework for the analysis of \algname{SGD}-type methods for smooth nonconvex problems. Their framework allows to analyze \algname{SGD} and \algname{DCGD} under various assumptions, including strong and weak growth, and various sampling strategies, including uniform and importance sampling. Unfortunately, the theory relies heavily on the unbiasedness of the stochastic gradients and, as a result, it is not applicable to our methods (in \algname{EF21-P + DCGD}, $\Exp{g^t} = \nabla f(w^t) \neq \nabla f(x^t)$). Therefore, we decided to rebuild the theory from scratch. Our results inherit all previous achievements of \citep{khaled2020better}, and further generalize the unified framework to make it suitable for optimization methods where the iterates are perturbed using the \algname{EF21-P} mechanism.
 We believe that this is a contribution with potential applications beyond the focus of this work (distributed optimization with bidirectional compression). This development is presented in Section~\ref{sec:abc}. Our main results from Section~\ref{sec:nonconvex_dcgd_general}--\ref{sec:nonconvex_dcgd_homogeneous} which cater to the nonconvex setting are simple corollaries of our general theory. 

\section{\algnamebig{EF21-P + DIANA} and \algnamebig{EF21-P + DCGD} Methods}

We are now ready to  present our main method \algname{EF21-P + DIANA} (see Algorithm~\ref{algorithm:diana_ef21_p}), which is a combination of \algname{EF21-P} mechanism described in Section~\ref{sec:ef21_p} and the \algname{DIANA} method of \citet{DIANA, horvath2019stochastic,gorbunov2020unified}. The pseudocode of Algorithm~\ref{algorithm:diana_ef21_p} should be self-explanatory. If the gradient shifts $\{h_i^t\}$ employed by \algname{DIANA} are initialized to zeros, and we choose $\beta=0$, then \algname{DIANA} reduces to \algname{DCGD}, and  \algname{EF21-P + DIANA}  thus reduces to  \algname{EF21-P + DCGD} (see Algorithm~\ref{algorithm:dcgd_ef21_p}). If we further choose the  dual/gradient compressors $\cC_i^D$ to be identity mappings, then \algname{EF21-P + DCGD} further reduces to \algname{EF21-P}.

\section{Analysis in the Convex Setting} \label{sec:convex}

Let us state our first convergence theorem.
\begin{restatable}{theorem}{THEOREMDIANASTRONGLY}
        \label{theorem:diana_strong}
        Suppose that Assumptions~\ref{ass:lipschitz_constant}, \ref{ass:workers_lipschitz_constant} and \ref{ass:convex} hold, $\beta = \frac{1}{\omega + 1},$ set $x^0 = w^0$ and let
$\gamma \leq \min\left\{\frac{n}{160 \omega L_{\max}}, \frac{\alpha}{100 L}, \frac{1}{(\omega + 1)\mu}\right\}.
$
        Then Algorithm~\ref{algorithm:diana_ef21_p} returns $x^{T}$ such that
        \begin{align*}
 \squeeze  &\frac{1}{2\gamma}\Exp{\norm{x^{T} - x^*}^2} + \Exp{f(x^{T}) - f(x^*)} \\&\leq \left(1 - \frac{\gamma \mu}{2}\right)^T V^0,
        \end{align*}
        where $V^0 \eqdef \frac{1}{2\gamma}\Exp{\norm{x^0 - x^*}^2} + \left(f(x^{0}) - f(x^{*})\right) + \frac{8 \gamma \omega (\omega + 1)}{n^2}\sum_{i=1}^n \norm{h^0_i - \nabla f_i(x^*)}^2.$
\end{restatable}
 
The above result says that \algname{EF21-P + DIANA} guarantees to find an $\varepsilon$-solution after 
$$ \squeeze T_{\textnormal{NEW}} \eqdef \cO\left(\left(\frac{L}{\alpha \mu} + \frac{\omega}{n}\frac{L_{\max}}{\mu} + \omega\right)\log \frac{1}{\varepsilon}\right)$$ steps.
Comparing this rate with rates achieved by prior algorithms (see Table~\ref{table:strongly_convex_case}), our method is the first one to guarantee the decoupling of noises $\alpha$ and $\omega$ coming from the server-to-workers and the workers-to-server compressors. Moreover, it is more general, as the server-to-workers compression can use biased compressors, including Top$K$ and Rank$K$ \citep{safaryan2021fednl}. These can in practice perform better than their unbiassed counterparts \citep{beznosikov2020biased, PowerSGD}.

As promised, let us now show that the communication complexity  of \algname{EF21-P + DIANA} is no worse than that of \algname{GD}. For simplicity, we assume that the server and the workers use the Top$K$ and Rand$K$ compressors, respectively. Since under this assumption, $\omega = \nicefrac{d}{K} - 1$ and $\alpha = \nicefrac{K}{d}$, the server-to-workers and the workers-to-server communication complexities equal
\begin{align*}
 \squeeze &\cO\left(K \times\left(\frac{L}{\alpha \mu} + \frac{\omega}{n}\frac{L_{\max}}{\mu} + \omega\right)\log \frac{1}{\varepsilon}\right)\\
 &=\cO\left(\left(d\frac{L}{\mu} + \frac{d}{n}\frac{L_{\max}}{\mu}\right)\log \frac{1}{\varepsilon}\right).
\end{align*}
Note that $L_{\max} \leq n L,$ so this complexity is no worse than  \algname{GD}'s complexity for any $K \in [1, d].$ In Section~\ref{sec:why_better}, we discuss regimes in which the new complexity can be strictly better. The general convex case is discussed in Section~\ref{sec:comm_diana_general}.

\subsection{Stochastic gradients}\label{sec:StochGrad}
In this section, we assume that the workers in \algname{EF21-P + DIANA} calculate stochastic gradients instead of exact gradients.

\begin{assumption}[Stochastic gradients]
    \label{ass:stochastic_unbiased_and_variance_bounded}
    For all $x \in \R^d,$ stochastic gradients $\widetilde{\nabla} f_i(x)$ are unbiased and have bounded variance, i.e.,
    ${\rm E}[\widetilde{\nabla} f_i(x)] = \nabla f_i(x),$ and ${\rm E}[ \|\widetilde{\nabla} f_i(x) - \nabla f_i(x)\|^2] \leq \sigma^2$ for all $i\in[n]$, where $\sigma^2 \geq 0.$
\end{assumption} 

We now provide a generalization of Theorem~\ref{theorem:diana_strong}:
\begin{restatable}{theorem}{THEOREMDIANASTRONGLYSTOCHASTIC}
    \label{theorem:diana_strong_stochastic}
    Let us consider Algorithm~\ref{algorithm:diana_ef21_p} using stochastic gradients $\widetilde{\nabla} f_i$ instead of exact gradients $\nabla f_i$ for all $i \in [n]$.
Let Assumptions~\ref{ass:lipschitz_constant}, \ref{ass:workers_lipschitz_constant}, \ref{ass:convex} and \ref{ass:stochastic_unbiased_and_variance_bounded} hold, $\beta = \frac{1}{\omega + 1},$ $x^0 = w^0,$ and 
    $\gamma \leq \min\left\{\frac{n}{160 \omega L_{\max}}, \frac{\alpha}{100 L}, \frac{1}{(\omega + 1)\mu}\right\}.$
    Then Algorithm~\ref{algorithm:diana_ef21_p} returns $x^{T}$ such that
    \begin{align*}
 \squeeze &\frac{1}{2\gamma}\Exp{\norm{x^{T} - x^*}^2} + \Exp{f(x^{T}) - f(x^*)} \\ &\leq \left(1 - \frac{\gamma \mu}{2}\right)^T V^0 + \frac{24 (\omega + 1) \sigma^2}{\mu n},
    \end{align*}
    where $V^0 \eqdef \frac{1}{2\gamma}\Exp{\norm{x^0 - x^*}^2} + \left(f(x^{0}) - f(x^{*})\right) + \frac{8 \gamma \omega (\omega + 1)}{n^2}\sum_{i=1}^n \norm{h^0_i - \nabla f_i(x^*)}^2.$
\end{restatable}

For general convex case, we refer to Theorem~\ref{theorem:diana_general_convex_stochastic}. Note that Theorem~\ref{theorem:diana_strong_stochastic} has the same convergence rate as Theorem~\ref{theorem:diana_strong}, except for the \textit{statistical} term $\cO\left(\nicefrac{(\omega + 1) \sigma^2}{\mu n}\right)$ that is the same as in \algname{DIANA} \citep{gorbunov2020unified, khaled2020unified} and does not depend on $\alpha$! 
In addition to the conditions on $\gamma$ stated in Theorem~\ref{theorem:diana_strong_stochastic}, achieving an $\varepsilon$-solution requires setting $\gamma \leq \Theta\left(\nicefrac{\varepsilon \mu n}{(\omega + 1) \sigma^2}\right)$.
Under this assumption, the convergence rate equals $T_{\textnormal{NEW}} + \cO\left(\frac{(\omega + 1) \sigma^2}{\mu^2 \varepsilon n}\log \frac{1}{\varepsilon}\right)$ and hence is no longer linear. However, this $\nicefrac{1}{\varepsilon}$ dependence is 
natural for stochastic methods given our assumptions \citep{gower2019sgd}.

\subsection{\algname{EF21-P + DCGD} and interpolation regime} \label{sec:interpol}

We also analyze a second method, \algname{EF21-P + DCGD}, which is based on \algname{DCGD} \citep{khaled2020better, alistarh2017qsgd}. One can think of \algname{DCGD} as \algname{DIANA} with parameter $\beta = 0.$ 
On one hand, the convergence of \algname{EF21-P + DCGD} is faster (see Theorem~\ref{theorem:dcgd_strong}) comparing to \algname{EF21-P + DIANA} (see Theorem~\ref{theorem:diana_strong}). On the other hand, we can guarantee the convergence only to a $\cO(\nicefrac{1}{n}\sum_{i=1}^n \norm{\nabla f_i(x^*)}^2)$ ``neighborhood'' of the solution. 
However, this ``neighborhood'' disappears in the interpolation regime, i.e., when $\nabla f_i(x^*) = 0$ for all $i \in [n]$. The interpolation regime is very common in modern  deep learning tasks \citep{brown2020language, bubeck2021universal}.

\subsection{Why do bidirectional methods work much better than \algname{GD}?}\label{sec:bidir-vs-GD}
\label{sec:why_better}

Our analysis of \algname{EF21-P + DIANA} covers the worst case scenario for the values of $L_{\max}$ and $\alpha$. Although $L_{\max}$ can be equal to $n L,$ in practice it tends to be much smaller. Similarly, the assumed bound on the parameter $\alpha$ equal to $\nicefrac{k}{d}$ for the Top$K$ compressor is also very conservative and the ``effective'' $\alpha$ is much larger \citep{beznosikov2020biased,PowerSGD,xu2021deepreduce}. Our claims are also supported by experiments from Section~\ref{sec:experiments}.

\begin{table*}[!ht]
    \caption{\textbf{General nonconvex Case.} The \# of communication rounds to get an $\varepsilon$-stationary point (${\rm E}[\norm{\nabla f(\widehat{x})}^2] \leq \varepsilon$). For simplicity, we assume that $f^*_i = f^*$ for all $i \in [n]$ and only the terms with respect to $\omega_{\textnormal{\wtos}}$ and $\omega_{\textnormal{\stow}}$ are shown. The parameters $\omega_{\textnormal{\wtos}}$ and $\omega_{\textnormal{\stow}}$ have the same meaning as in Table~\ref{table:strongly_convex_case}.}
    \label{table:general_case}
    \centering 
    \begin{threeparttable}
      \begin{tabular}{cccccc}
        \toprule
\bf      Method & \bf \# Communication Rounds& 
\bf      Limitations \\
       \toprule
       \algnamesmall{DCGD} \citep{khaled2020better} & $\frac{\Delta_0^2 \omega_{\textnormal{\wtos}} L L_{\max}}{n \varepsilon^2}$ & \makecell{No server-to-worker compression.} \\
    \midrule
       \makecell{\algnamesmall{MARINA}, \algnamesmall{DASHA} \\
       \citep{MARINA} \\ \citep{tyurin2022dasha}} &   $\frac{\Delta_0 \omega_{\textnormal{\wtos}} \widehat{L}}{\sqrt{n} \varepsilon}$ & \makecell{No server-to-worker compression.} \\
       \midrule
       \algnamesmall{dist-EF-SGD} \citep{zheng2019communication} & $\Omega\left(\frac{\max\{\omega_{\textnormal{\stow}}, \omega_{\textnormal{\wtos}}\}^2 L_{\max}}{\varepsilon^{3/2}}\right)$ & \makecell{Homogeneous regime only, \\ i.e., $f_i = f$ for all $i \in [n]$. \\ Bounded gradient assumption.} \\
       \midrule
       \algnamesmall{MCM} \citep{philippenko2021preserved} & 
       ${\scriptstyle \Delta_0} \left(\frac{\omega_{{\stow}}^{3/2}}{\varepsilon} {\scriptstyle +} \frac{\omega_{\textnormal{\stow}} \omega_{\textnormal{\wtos}}^{1/2}}{\sqrt{n} \varepsilon} {\scriptstyle +} \frac{\omega_{\textnormal{\wtos}}}{n \varepsilon}\right) {\scriptstyle L_{\max}}$ 
       & \begin{tabular}[x]{@{}c@{}}Homogeneous regime only.\end{tabular} \\
       \midrule
       \algnamesmall{NEOLITHIC} \citep{huang2022lower} & $\Omega\left(\frac{\Delta_0 L_{\max}}{\varepsilon}\right)$ & \makecell{Sends $\Omega(d)$ coordinates in each round\tnote{\color{blue}(a)}. \\Bounded gradient similarity assumption.} \\
       \midrule
       \algnamesmall{CD-Adam} \citep{wang2022communication} & $\Omega\left(\frac{\sqrt{d} \max\{\omega_{\textnormal{\stow}}, \omega_{\textnormal{\wtos}}\}^4}{\varepsilon^2}\right)$ & Bounded gradient assumption. \\
       \midrule
       \algnamesmall{EF21-BC} \citep{fatkhullin2021ef21} & $\frac{\Delta_0 \omega_{\wtos} \omega_{\stow} \widehat{L}}{\varepsilon}$ & --- \\
      \midrule
 \cellcolor{bgcolor1}     \algnamesmall{EF21-P + DCGD} (new) &  \cellcolor{bgcolor1} $\frac{\Delta_0^2 \omega_{\textnormal{\wtos}} L L_{\max}}{n \varepsilon^2} + \frac{\Delta_0 \omega_{\stow}L}{\varepsilon}$ & \cellcolor{bgcolor1}  --- \\
      \midrule
 \cellcolor{bgcolor1}     \algnamesmall{EF21-P + DCGD} (new) &  \cellcolor{bgcolor1} $\frac{\Delta_0 D \omega_{\wtos} L}{n \varepsilon} +  \frac{\Delta_0 \omega_{\stow} L}{\varepsilon}$ & \cellcolor{bgcolor1}  \makecell{Strong-growth assumption with parameter $D.$} \\
      \bottomrule
      \end{tabular}
  \begin{tablenotes}
  \scriptsize
  \item [{\color{blue}(a)}] In each communication round (outer loop for $k = 0, \dots, K - 1$ of Algorithm 2 in \citep{huang2022lower}), \algnamesmall{NEOLITHIC} sends a number of compressed vectors ($R$ in \citep{huang2022lower}) that is proportional to $\Omega\left(\nicefrac{1}{\alpha}\right)$ (this quantity is even worse since we ignore the logarithmic factors), where $\alpha$ is the parameter of a contractive compressor. For Top$K$ or Rand$K$, it means that \algnamesmall{NEOLITHIC} sends $\Omega(\nicefrac{d}{K})$ sparsified vectors with $K$ nonzero elements. This means that, in total, $\Omega(d)$ values are sent in each communication round. The total number of communication rounds $K$ is at least $\Omega\left(\nicefrac{\Delta_0 L_{\max}}{\varepsilon}\right).$ Note that the vanilla \algnamesmall{GD} method requires $\Omega\left(\nicefrac{\Delta_0 L}{\varepsilon}\right)$ rounds and sends $\cO(d)$ coordinates in each round.
  \end{tablenotes}          
  \end{threeparttable}
 \end{table*} 

\section{Analysis in the Nonconvex Setting} 
\label{sec:nonconvex}

In the nonconvex case, existing bidirectional methods suffer from the same problem as those used in the convex case (see Section~\ref{sec:bidirection_compression}): they either do not provide  server-to-workers compression at all, or the compressor errors/noises are coupled in a multiplicative fashion (see $\omega_{\textnormal{\wtos}}$ and $\omega_{\textnormal{\stow}}$ in Table~\ref{table:general_case}).

Instead of the convexity (see Assumption~\ref{ass:convex}), we will need the following assumption:
\begin{assumption}[Lower boundedness]
   \label{ass:lower_bound}
   There exist $f^* \in \R$ and $f^*_1, \dots, f^*_n \in \R$ such that $f(x) \geq f^*$ and $f_i(x) \geq f^*_i$ for all $x \in \R^d$ and for all $i \in [n].$
\end{assumption}

As in the convex setting, the theory of methods that only use workers-to-server compression is well examined. In the high accuracy regimes, the current state-of-the-art methods are \algname{MARINA} and \algname{DASHA} \citep{MARINA, tyurin2022dasha}; both return an $\varepsilon$-stationary point after $$\cO\left(\frac{\Delta_0 L}{\varepsilon} + \frac{\Delta_0\omega\widehat{L}}{\sqrt{n}\varepsilon}\right)$$ iterations, where $\Delta_0 \eqdef f(x^0) - f^*.$ In the low accuracy regimes, the current state-of-the-art method is \algname{DCGD} \citep{khaled2020better}, with iteration complexity $$\cO\left(\frac{\Delta_0 L}{\varepsilon} + \frac{\Delta_0 (\Delta_0 + \Delta^*) (1 + \omega) L L_{\max}}{n\varepsilon^2}\right),$$ where $\Delta^* \eqdef f^* - \frac{1}{n} \sum_{i=1}^n f^*_i.$ Note that \algname{DCGD} has worse dependence on $\varepsilon,$ but it scales much better with the number of workers $n.$

We now investigate how \algname{EF21-P} can help us in the general nonconvex case.
Let us recall that in the convex case, decoupling of the noises coming from two compression schemes can be achieved by combining \algname{EF21-P} with \algname{DIANA}.
In the nonconvex setting, we successfully combine \algname{EF21-P} and \algname{DCGD}. Moreover, we provide analysis of some particular cases where \algname{EF21-P + DCGD} can be the method of choice in the high accuracy regimes.

Whether or not it is possible to achieve the decoupling by combining our method with \algname{MARINA} or \algname{DASHA} is not yet known and we leave it to future work\footnote{We did not try to get the convergence rate of \algname{EF21-P + DIANA} in the nonconvex regime because it is well known that \algname{DIANA} is a suboptimal method in the nonconvex case \citep{MARINA}.}.

\subsection{\algname{EF21-P + DCGD} in the general nonconvex case}
\label{sec:nonconvex_dcgd_general}
Without any restrictive assumptions, we can prove the following convergence result:
\begin{restatable}{theorem}{THEOREMDCGDnonconvex}
        \label{theorem:dcgd_nonconvex}
Consider Algorithm~\ref{algorithm:dcgd_ef21_p} and let Assumptions~\ref{ass:lipschitz_constant}, \ref{ass:workers_lipschitz_constant} and \ref{ass:lower_bound} hold, $x^0 = w^0,$ and
 $\gamma = \min\left\{
                \frac{\alpha}{8 L},
                \frac{\sqrt{n}}{\sqrt{2\omega L L_{\max} T}},
                \frac{n \varepsilon}{32\Delta^* \omega L L_{\max}}
                \right\}.$
        If the number of iterations $$ \squeeze T \geq 
            \frac{48\Delta_{0} L}{\varepsilon}
        \max\left\{
        \frac{8}{\alpha},
        \frac{96\Delta_{0}\omega L_{\max}}{n \varepsilon},
        \frac{32\Delta^* \omega L_{\max}}{n \varepsilon}
        \right\},$$ then $\min\limits_{0\leq t\leq T-1} \Exp{\norm{\nabla f(x^{t})}^2} \leq \varepsilon$
 (The proof follows from Theorem~\ref{theorem:abcrate} and Proposition~\ref{prop:abccases} (Part 1)).
\end{restatable}
We get the rate of \algname{DCGD} \citep{khaled2020better} plus an additional $\cO\left(\frac{\Delta_{0} L}{\alpha \varepsilon}\right)$ factor, thus obtaining the first method with bidirectional compression where the noises from the compressors are decoupled. Moreover, as noted before, this method provides the state-of-the-art rates when $\varepsilon$ is small or the number of workers $n$ is large.

\subsection{Strong growth condition}
\label{sec:nonconvex_dcgd_strong_growth}
Here we analyze \algname{EF21-P + DCGD} under the strong-growth condition \citep{schmidt2013fast}.
\begin{assumption}
    \label{ass:strong_growth}
    There exists $D>0$ such that $\frac{1}{n} \sum_{i=1}^n \norm{\nabla f_i(x)}^2 \leq D\norm{\nabla f(x)}^2$
    for all $x\in\R^d$.
\end{assumption}
While this assumption is restrictive and does not even hold for quadratic optimization problems, there exist numerous practical applications when it is reasonable. These include, for example, deep learning, where the number of parameters $d$ is so huge that the model can interpolate the training dataset \citep{schmidt2013fast, vaswani2019fast, meng2020fast}. To train such models, engineers use distributed environments, in which case communication becomes the main a bottleneck \citep{ramesh2021zero}. 
For these problems, our method is suitable and can be successfully applied.

\begin{restatable}{theorem}{THEOREMDCGDnonconvexSTRONGGROWTH}
        \label{theorem:dcgd_nonconvex_strong_growth}
Consider Algorithm~\ref{algorithm:dcgd_ef21_p}, let Assumptions~\ref{ass:lipschitz_constant}, \ref{ass:workers_lipschitz_constant}, \ref{ass:lower_bound} and \ref{ass:strong_growth} hold, and choose $x^0 = w^0$ and $
            \gamma = \min\left\{
                \frac{\alpha}{8 L},
                \frac{n}{4 D \omega L}
                \right\}.$
        If the number of iterations
        $$\squeeze T \geq 
            \frac{48\Delta_{0} L}{\varepsilon}
        \max\left\{
            \frac{8}{\alpha},\frac{4 D \omega}{n}
        \right\},$$ then $\min \limits_{0\leq t\leq T-1} \Exp{\norm{\nabla f(x^{t})}^2} \leq \varepsilon.$
        (The proof follows from Theorem~\ref{theorem:abcrate} and Proposition~\ref{prop:abccases} (Part 2)).
\end{restatable}
Comparing to Section~\ref{sec:nonconvex_dcgd_general}, the above result shows an improved dependence on $\varepsilon$ under the strong growth assumption.

\subsection{Homogeneous regime}
\label{sec:nonconvex_dcgd_homogeneous}
Another important problem where our method can be useful is distributed optimization in the data homogeneous regime. In particular, we consider the case when $f_i = f$ for all $i \in [n]$ and when instead of exact gradients, stochastic gradients are used. This assumption holds, for instance, for distributed machine learning problems where every worker samples mini-batches from a large shared dataset \citep{recht2011hogwild, goyal2017accurate}.

\begin{restatable}{theorem}{THEOREMDCGDnonconvexHOMOG}
        \label{theorem:dcgd_nonconvex_homog}
        Let us consider Algorithm~\ref{algorithm:dcgd_ef21_p} with the stochastic gradients $\widetilde{\nabla} f$ instead of the exact gradients $\nabla f$. Suppose that Assumptions~\ref{ass:lipschitz_constant}, \ref{ass:workers_lipschitz_constant}, \ref{ass:stochastic_unbiased_and_variance_bounded} and \ref{ass:lower_bound} hold and $f_i =f $ for all $i \in [n]$. Set $x^0 = w^0$ and let
        $
            \gamma = \min\left\{
                \frac{\alpha}{8 L},
                \frac{1}{4 \left(\frac{\omega}{n} + 1\right)L},
                \frac{n \varepsilon}{16 \left(\omega + 1\right) \sigma^2 L}
                \right\}.
        $
        If the number of iterations $$\squeeze 
            T \geq 
            \frac{48\Delta_{0} L}{\varepsilon}
        \max\left\{
            \frac{8}{\alpha},4\left(\frac{\omega}{n} + 1\right),
        \frac{16 (\omega + 1)\sigma^2}{n \varepsilon}
        \right\},$$ then $\min \limits_{0\leq t\leq T-1} \Exp{\norm{\nabla f(x^{t})}^2} \leq \varepsilon.$
 (The proof follows from Theorem~\ref{theorem:abcrate} and Proposition~\ref{prop:abccases} (Part 4)).
\end{restatable}

Under the same assumptions, \algname{MCM} method by \citep{philippenko2020artemis} with bidirectional compression guarantees the convergence rate (up to constant factors)
$$
\frac{\Delta_{0} L}{\varepsilon}
\max\left\{
\left(\frac{\omega_{\wtos}}{n} + 1\right),
\omega_{\stow}^{3/2},
\frac{\omega_{\stow}\omega_{\wtos}^{1/2}}{\sqrt{n}},
\frac{(\omega_{\wtos} + 1)\sigma^2}{n \varepsilon}
\right\}.
$$
Comparing this with our result, the last \textit{statistical} term $\nicefrac{(\omega + 1)\sigma^2}{n \varepsilon}$ is the same in both cases, but we significantly improve the other \textit{communication} terms (take $\omega = \omega_{\wtos}$ and $\alpha = (\omega_{\stow} + 1)^{-1}$ in Theorem~\ref{theorem:dcgd_nonconvex_homog}).

\begin{figure}[t]
    \centering
    \begin{subfigure}{.49\textwidth}
      \centering
      \includegraphics[width=\textwidth]{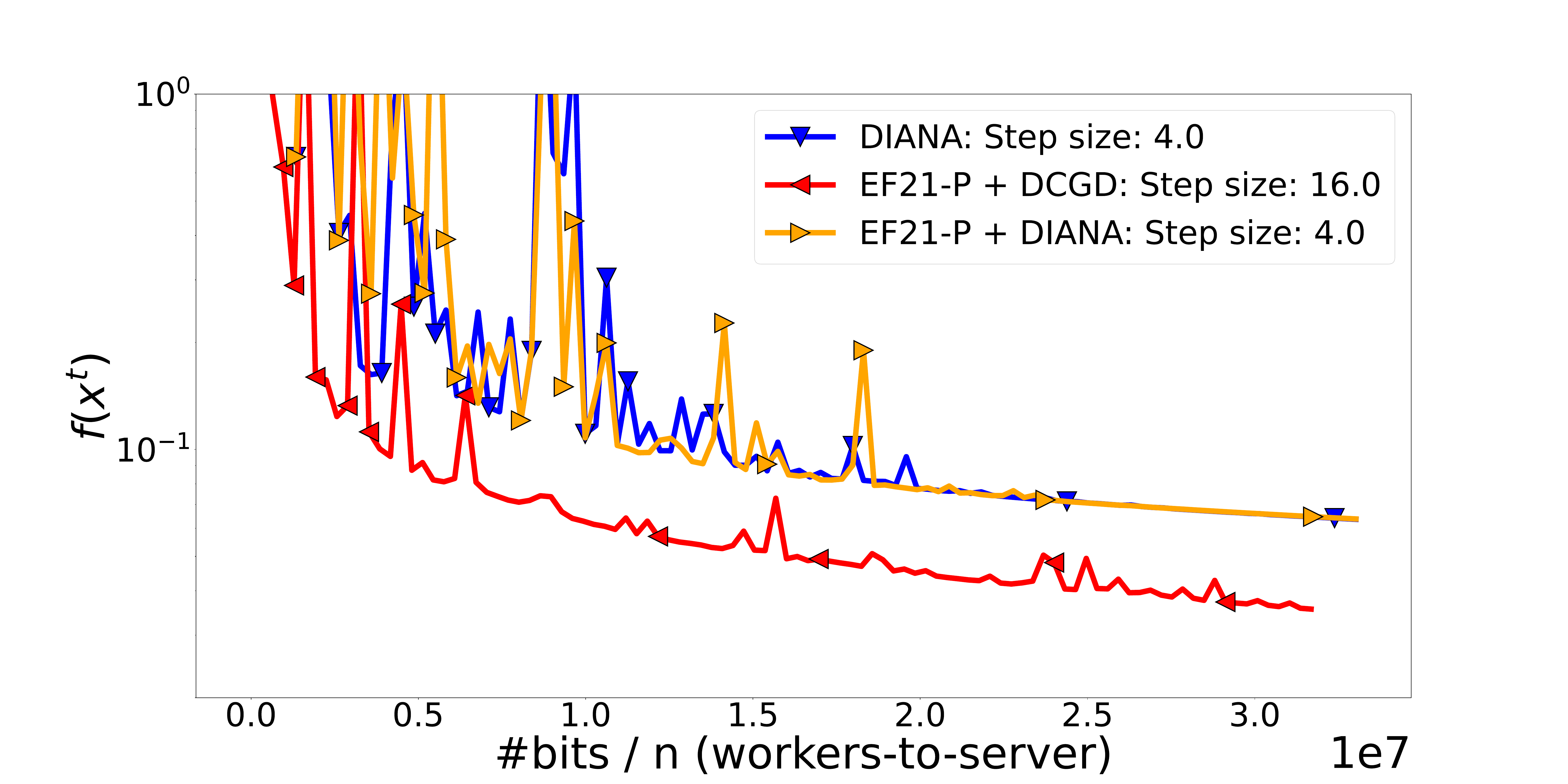}
    \end{subfigure}
    \begin{subfigure}{.49\textwidth}
        \centering
        \includegraphics[width=\textwidth]{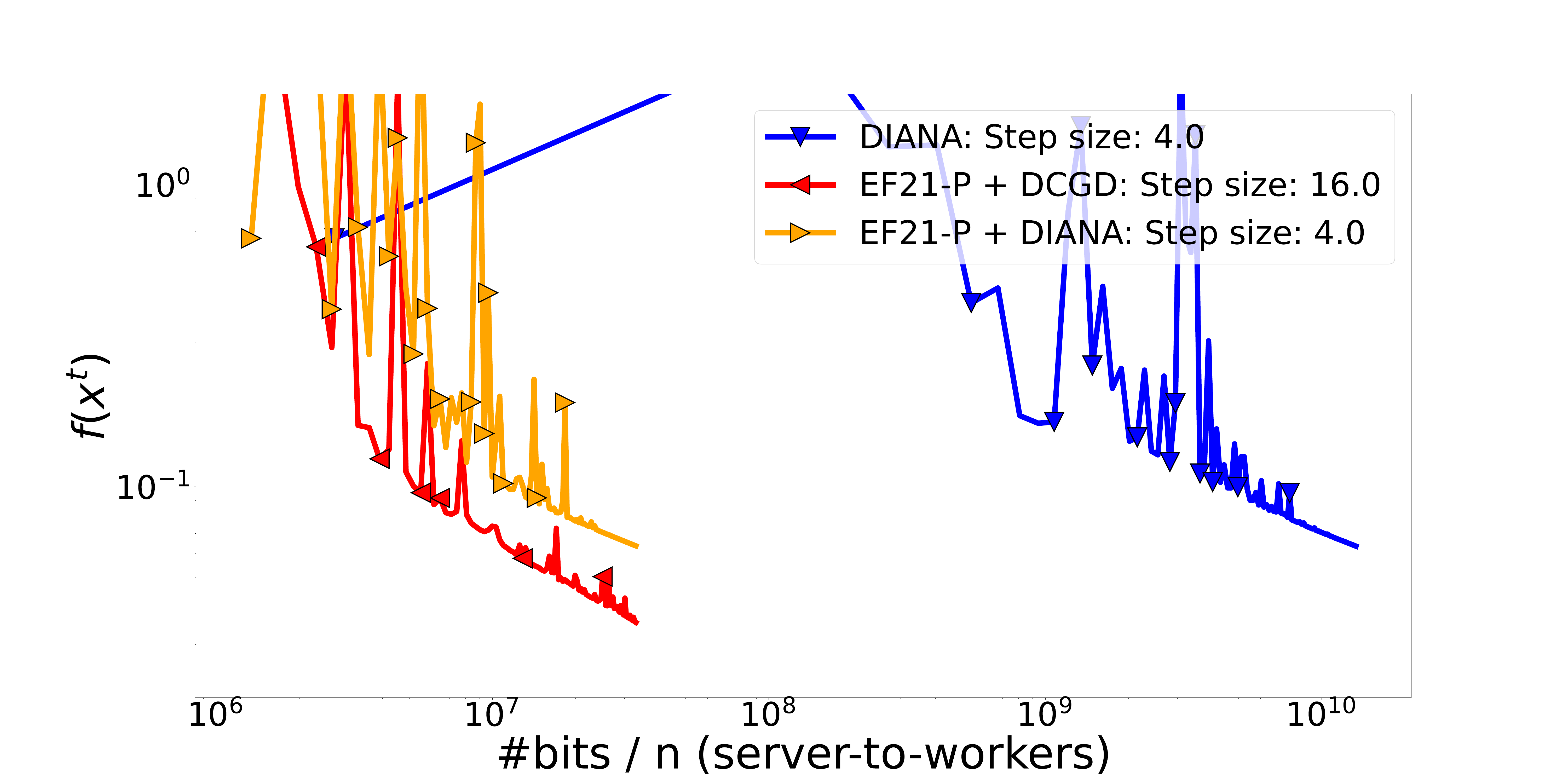}
      \end{subfigure}\hfill
    \caption{Logistic Regression with \textit{real-sim} dataset. Number of workers: $n = 100.$ Sparsification level was set to $K = 100$ for all compressors.}
    \label{fig:diana_real_sim}
\end{figure}

\section{Experimental Highlights}\label{sec:experiments}

We first provide a few highlights from our experiments.  For more details and experiments, we refer to Section~\ref{sec:experiments_appendix}, where we compare our algorithms with the previous state-of-the-art method \algname{MCM} and solve a nonconvex task.

In particular, we consider the logistic regression task with \textit{real-sim} (\# of features $ = $ 20,958, \# of samples equals 72,309) from LIBSVM dataset \citep{chang2011libsvm}. Each plot represents the relations between function values and the total number of coordinates transmitted from and to the server. In all algorithms, the Rand$K$ compressor is used to compress information from the workers to the server. In the case of \algname{EF21-P + DIANA} and \algname{EF21-P + DCGD}, we take Top$K$ compressor to compress from the server to the workers. 

The results are presented in Figure~\ref{fig:diana_real_sim}. The main conclusion from these experiments is that \algname{EF21-P + DIANA} and \algname{EF21-P + DCGD} converge to a solution not slower than \algname{DIANA}, even though \algname{DIANA} does not compress vectors sent from the server to the workers! This means that \algname{EF21-P + DIANA} and \algname{EF21-P + DCGD} can send $\times400$ less values from the server to the workers for free! Moreover, we see that \algname{EF21-P + DCGD} converges faster than its competitors. Similar experimental results were observed in \citep{philippenko2021preserved}.

\section{Future Work and Possible Extensions}
In this paper, many important features of distributed and federated learning were not investigated in detail. These include variance reduction of stochastic gradients \citep{horvath2019stochastic,tyurin2022dasha}, acceleration \citep{li2021canita, ADIANA}, local steps \citep{murata2021bias}, partial participation \citep{mcmahan2017communication, tyurin2022computation} and asynchronous communication \citep{koloskova2022sharper}. While some are simple exercises and can be easily added to our methods, many of them deserve further investigation and separate work. 

Further, note that several authors, including \citet{szlendak2021permutation,richtarik20223pc, condat2022ef}, considered somewhat different families of compressors than those we consider here. We believe that the results and discussion from our paper can be adapted to these families.

\subsection*{Acknowledgements}
The work of P. Richt\'{a}rik was partially supported by the KAUST Baseline Research Fund Scheme and by the SDAIA-KAUST Center of Excellence in Data Science and Artificial Intelligence. The work of A. Tyurin was supported by the Extreme Computing Research Center (ECRC) at KAUST. The work of K. Gruntkowska was supported by the Visiting Student Research Program (VSRP) at KAUST. We thank Laurent Condat (KAUST) for the suggestion to improve the presentation of the theorems.

\bibliography{paper}
\bibliographystyle{icml2023}

\newpage
\appendix
\onecolumn

\tableofcontents
\newpage

\section{Further Experiments} \label{sec:experiments_appendix}

We now provide the results of our experiments on practical machine learning tasks with LIBSVM datasets \citep{chang2011libsvm} (under the 3-clause BSD license).  Each plot represents the relations between function values and the total number of coordinates transmitted from and to the server. The parameters of the algorithms are as suggested by the theory, except for the stepsizes $\gamma$ that we fine-tune from a set $\{2^i\,|\,i \in [-10, 10]\}$.

We solve the logistic regression problem:
    $$f_i(x_1, \dots, x_{c}) \eqdef -\frac{1}{m} \sum_{j=1}^m\log\left(\frac{\exp\left(a_{ij}^\top x_{y_{ij}}\right)}{\sum_{y=1}^c \exp\left(a_{ij}^\top x_{y}\right)}\right), $$
where $x_1, \dots, x_{c} \in \R^{d}$, $c$ is the number of unique labels,
$a_{ij} \in \R^{d}$ is a feature of a sample on the $i$\textsuperscript{th} worker, $y_{ij}$ is a corresponding label and $m$ is the number of samples located on the $i$\textsuperscript{th} worker. In all algorithms, the Rand$K$ compressor is used to compress information from the workers to the server. In the case of \algname{EF21-P + DIANA} and \algname{EF21-P + DCGD}, we take Top$K$ compressor to compress from the server to the workers. The performance of algorithms is compared on \textit{w8a} (\# of features $ = 300$, \# of samples equals $\num[group-separator={,}]{49749}$), CIFAR10 \citep{krizhevsky2009learning} (\# of features $ = 3072$, \# of samples equals $\num[group-separator={,}]{50000}$), and \textit{real-sim} (\# of features $ = 20958$, \# of samples equals $\num[group-separator={,}]{72309}$) datasets.

The results are presented in Figures~\ref{fig:diana_real_sim}, \ref{fig:diana_w8a} and \ref{fig:diana_cifar10}. The conclusions are the same as in Section~\ref{sec:experiments}. One can see that \algname{EF21-P + DIANA} and \algname{EF21-P + DCGD} converge to a solution not slower than \algname{DIANA}, even though \algname{DIANA} does not compress vectors sent from the server to the workers! \algname{EF21-P + DIANA} and \algname{EF21-P + DCGD} send $\times100-\times1000$ less values from the server to the workers!

\begin{figure}[h]
    \centering
    \begin{subfigure}{.49\textwidth}
      \centering
      \includegraphics[width=\textwidth]{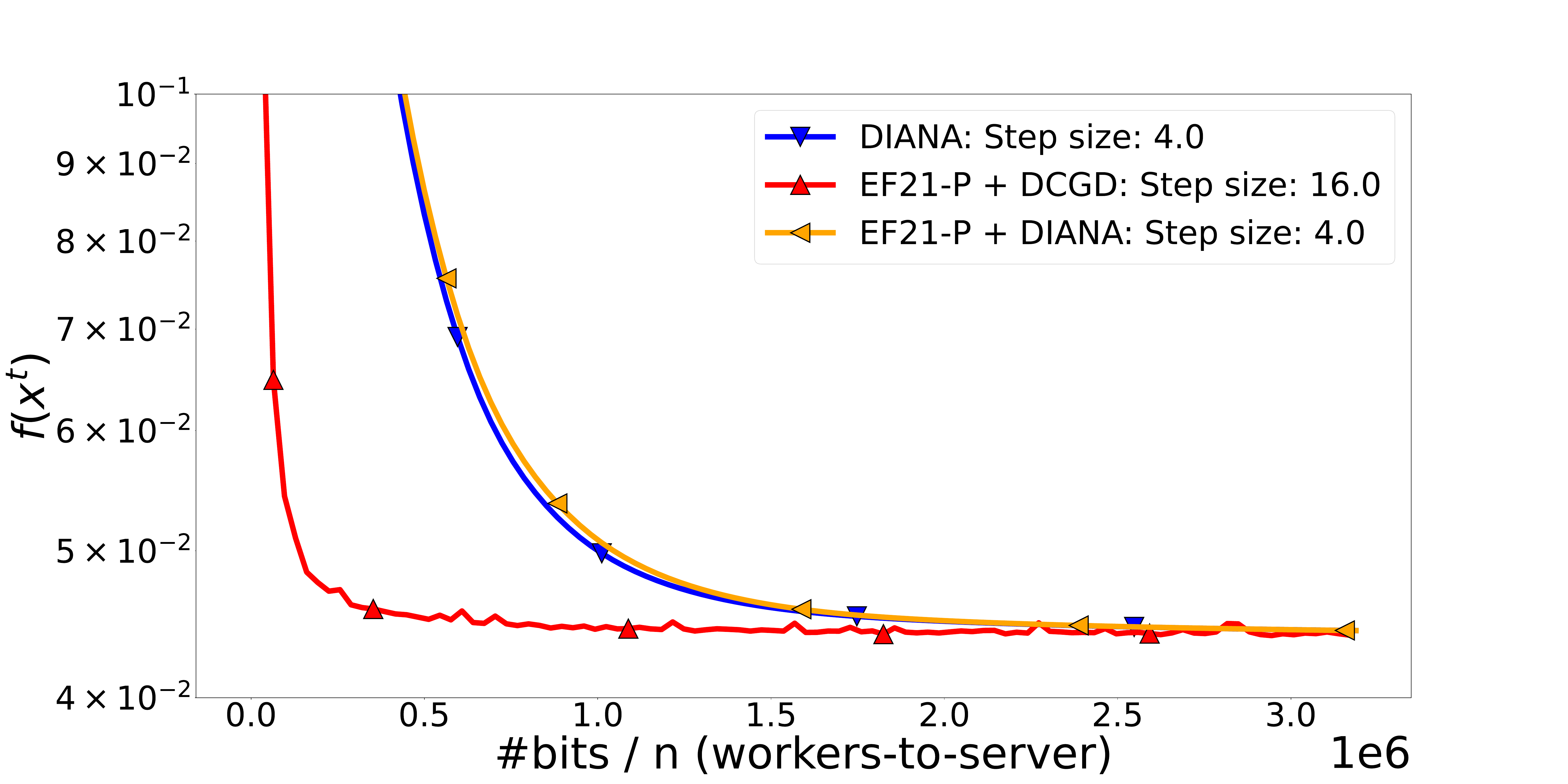}
    \end{subfigure}
    \begin{subfigure}{.49\textwidth}
        \centering
        \includegraphics[width=\textwidth]{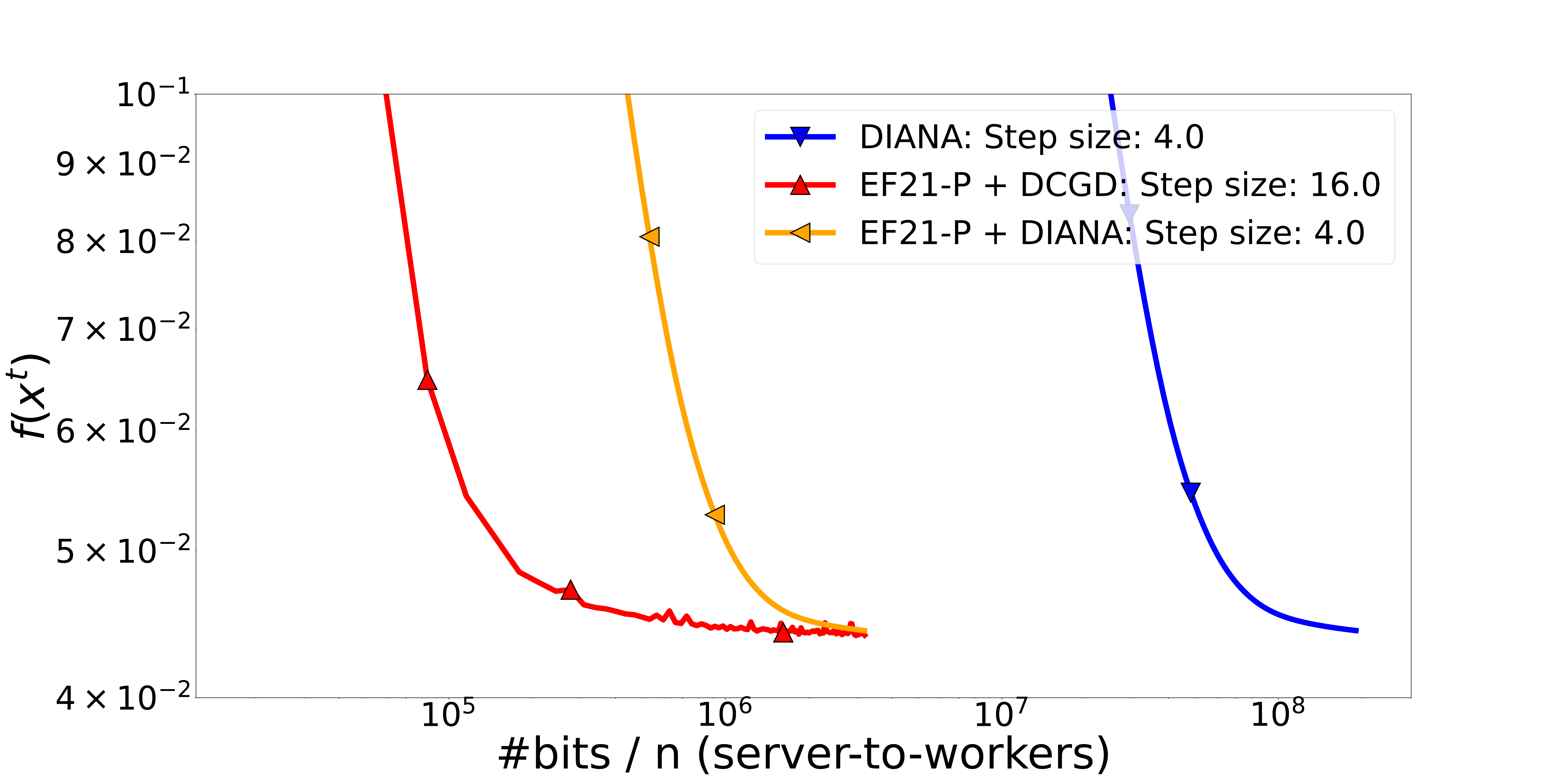}
      \end{subfigure}\hfill
    \caption{Logistic Regression with \textit{w8a} dataset. \# of workers $n = 10.$ $K = 10$ in all compressors.}
    \label{fig:diana_w8a}
\end{figure}

\begin{figure}[h]
    \centering
    \begin{subfigure}{.49\textwidth}
      \centering
      \includegraphics[width=\textwidth]{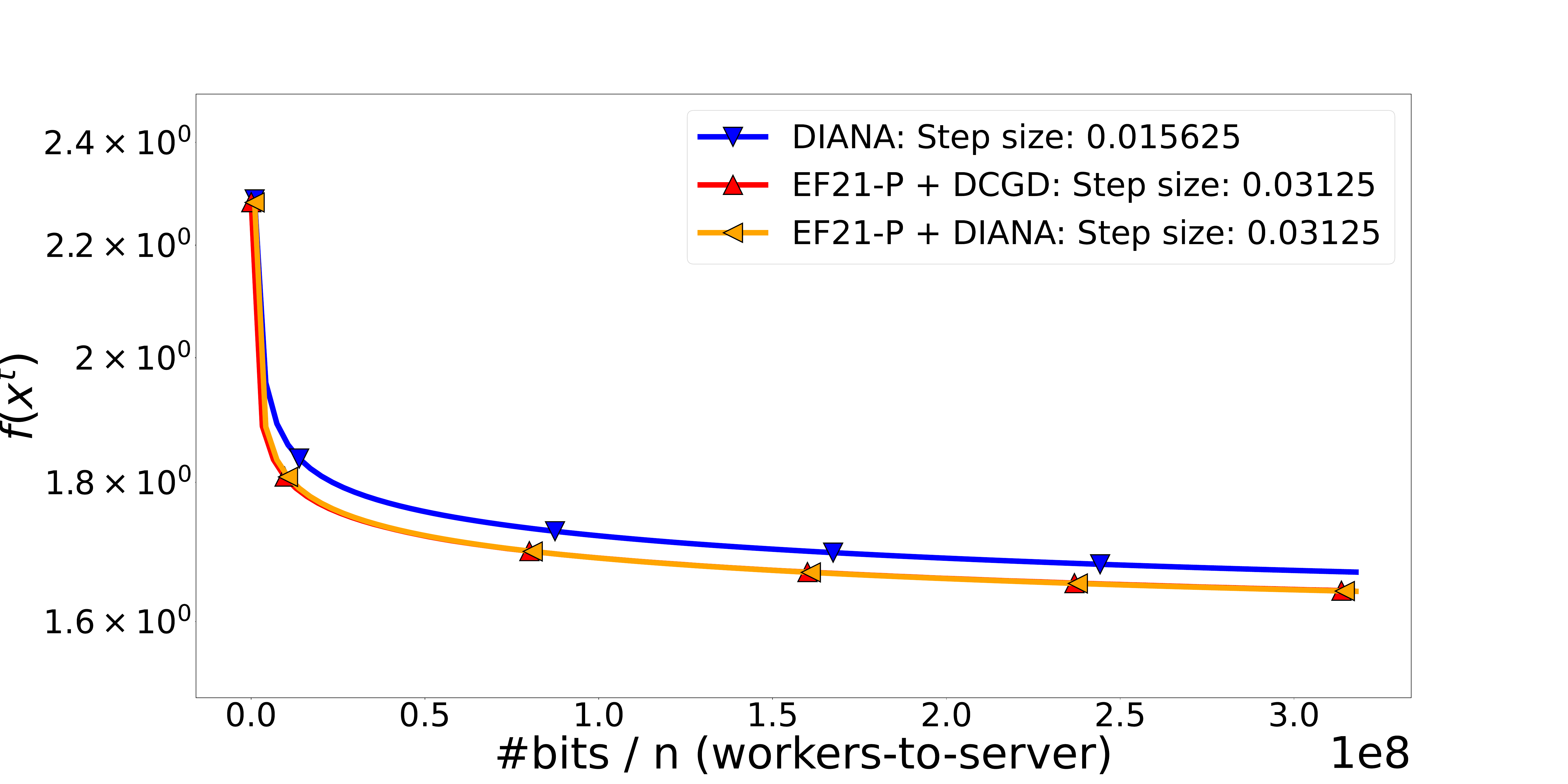}
    \end{subfigure}
    \begin{subfigure}{.49\textwidth}
        \centering
        \includegraphics[width=\textwidth]{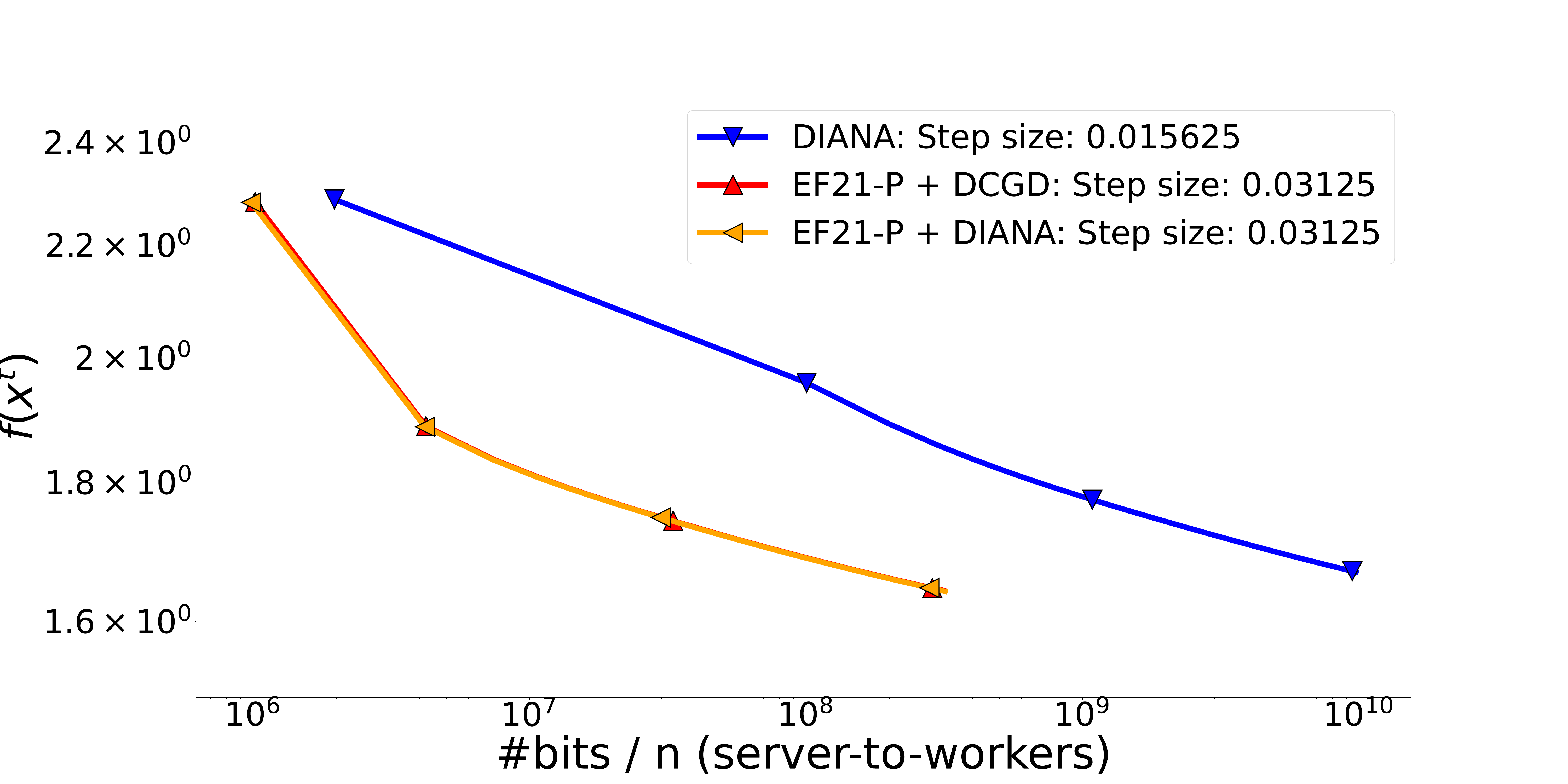}
      \end{subfigure}\hfill
    \caption{Logistic Regression with \textit{CIFAR10} dataset. \# of workers $n = 10.$ $K = 1000$ in all compressors.}
    \label{fig:diana_cifar10}
\end{figure}

We also compare our algorithm to \algname{MCM}. Since \algname{MCM} does not support contractive compressors defined in~\eqref{eq:biased_compressor}, we use Rand$K$ instead of the Top$K$ compressor in the server-to-workers compression. Figure~\ref{fig:mcm_real_sim} shows that our new algorithms converge faster.

\begin{figure}[h]
    \centering
    \begin{subfigure}{.49\textwidth}
      \centering
      \includegraphics[width=\textwidth]{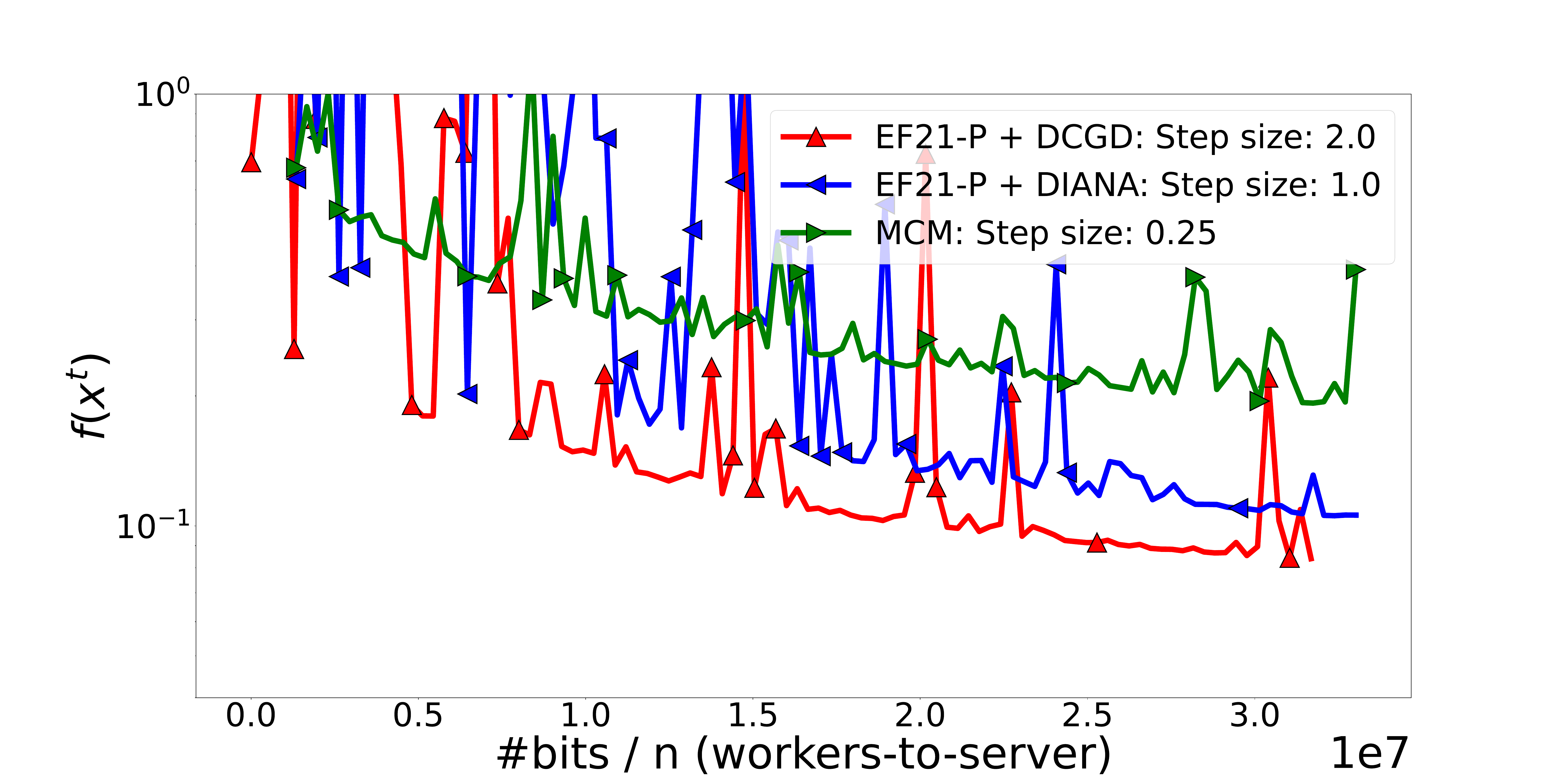}
    \end{subfigure}
    \begin{subfigure}{.49\textwidth}
        \centering
        \includegraphics[width=\textwidth]{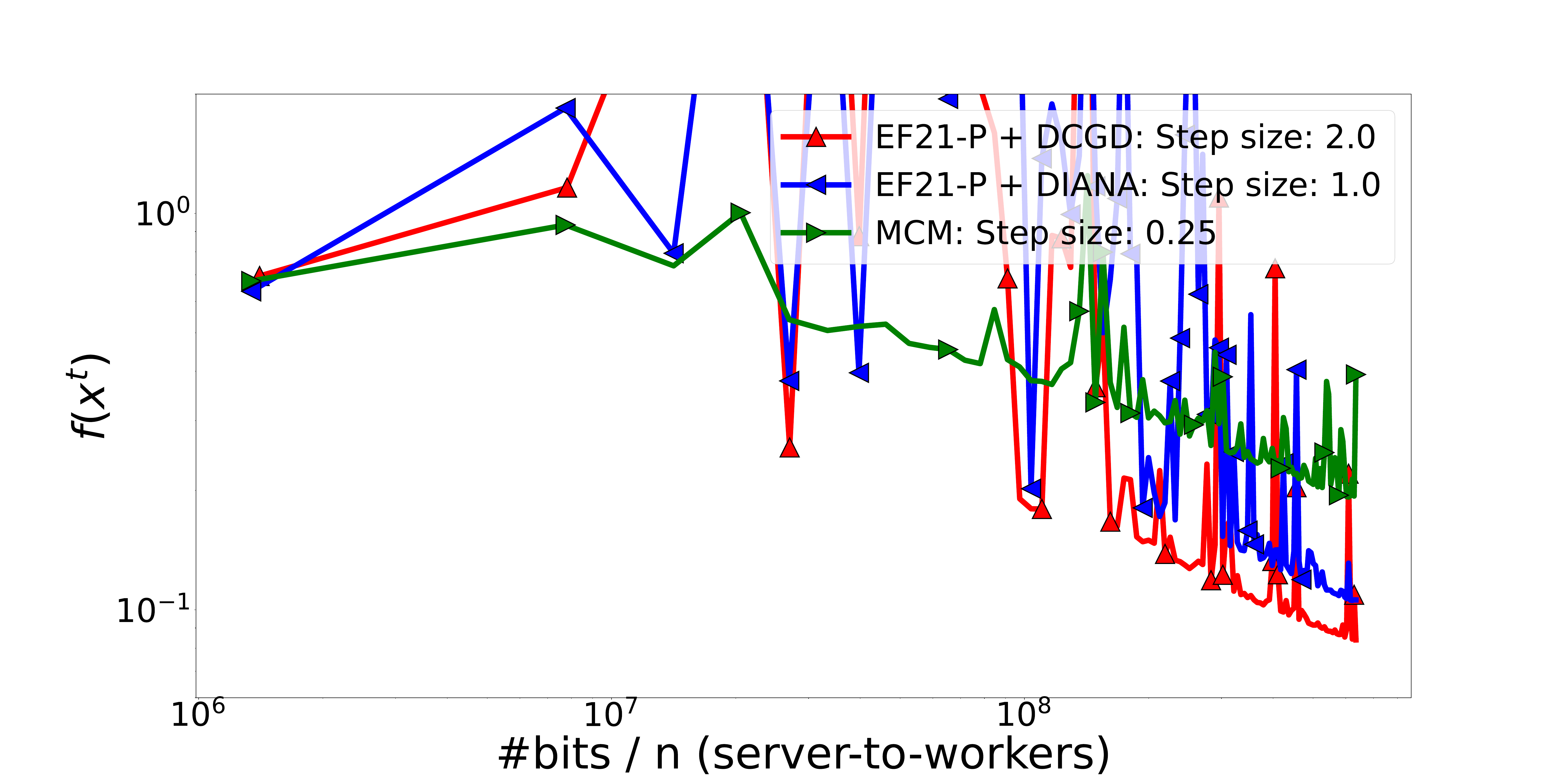}
      \end{subfigure}\hfill
    \caption{Logistic Regression with \textit{real-sim} dataset. \# of workers $n = 100.$ The parameters of workers-to-server and server-to-workers compressors are $K_{\wtos} = 100$ and $K_{\stow} = 2000.$}
    \label{fig:mcm_real_sim}
\end{figure}

Finally, we provide experiments for the nonconvex setting and compare \algname{EF21-P + DCGD} against \algname{EF21-BC} \citep{fatkhullin2021ef21} and \algname{DASHA} \citep{tyurin2022dasha}. We consider the logistic regression with a nonconvex regularizer $$r(x_1, \dots, x_c) \eqdef \lambda \sum_{y = 1}^c \sum_{k = 1}^d \frac{[x_{y}]_k^2}{1 + [x_{y}]_k^2},$$
where $[\cdot]_k$ is an indexing operation of a vector and $\lambda = 0.001.$ We use Rand$K$ and Top$K$ compressors for the workers-to-server and server-to-workers compressions, respectively. Note that in these experiments, the server-to-workers compression is only supported by \algname{EF21-P + DCGD} and \algname{EF21-BC}. In Figure~\ref{fig:nonconve_exper}, one can see that \algname{EF21-P + DCGD} converges faster than other algorithms and outperforms \algname{DASHA}, which does not compress vectors when transmitting them from the server to the workers.

\begin{figure}[h]
    \centering
    \begin{subfigure}{.49\textwidth}
      \centering
      \includegraphics[width=\textwidth]{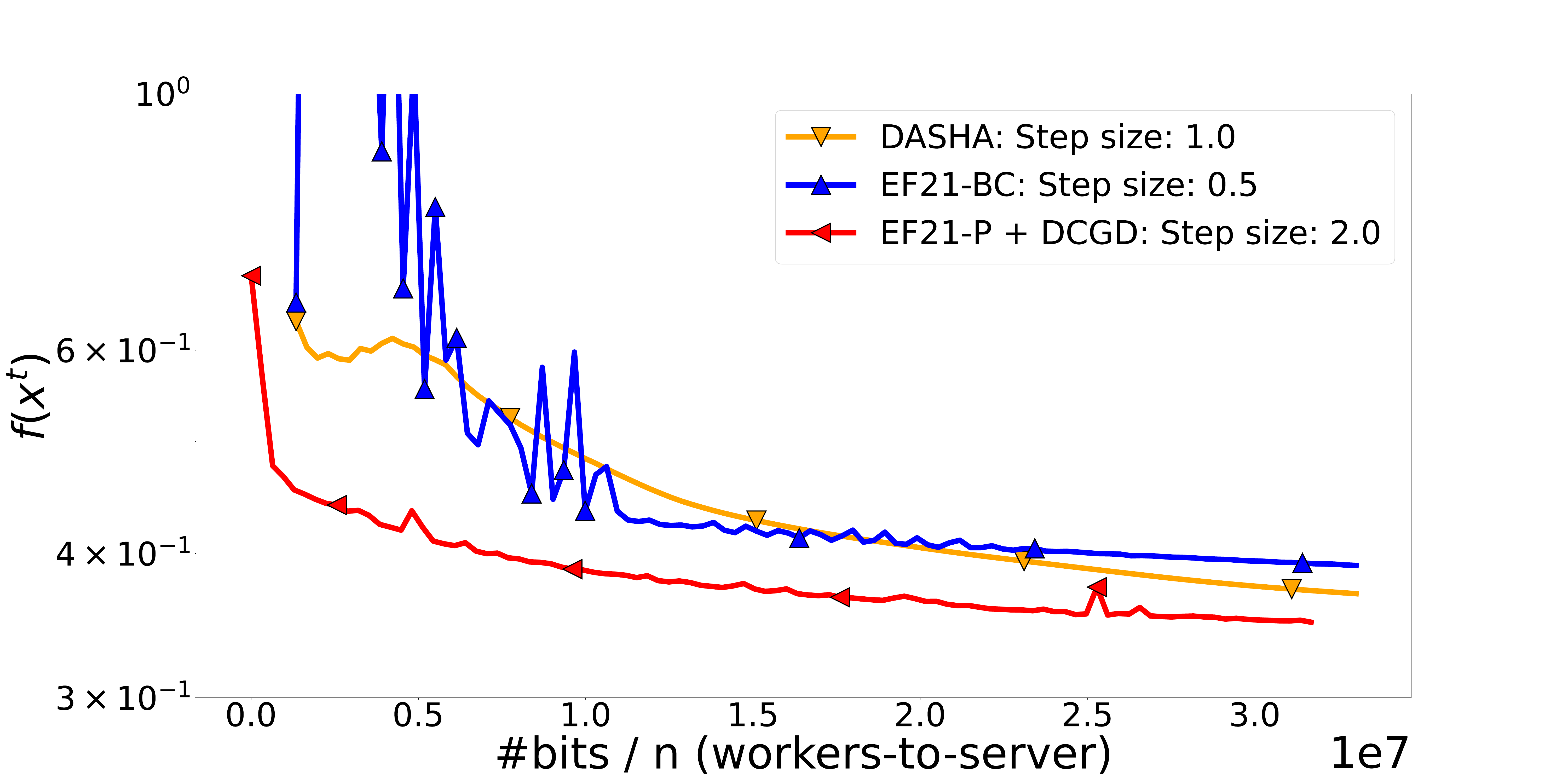}
    \end{subfigure}
    \begin{subfigure}{.49\textwidth}
        \centering
        \includegraphics[width=\textwidth]{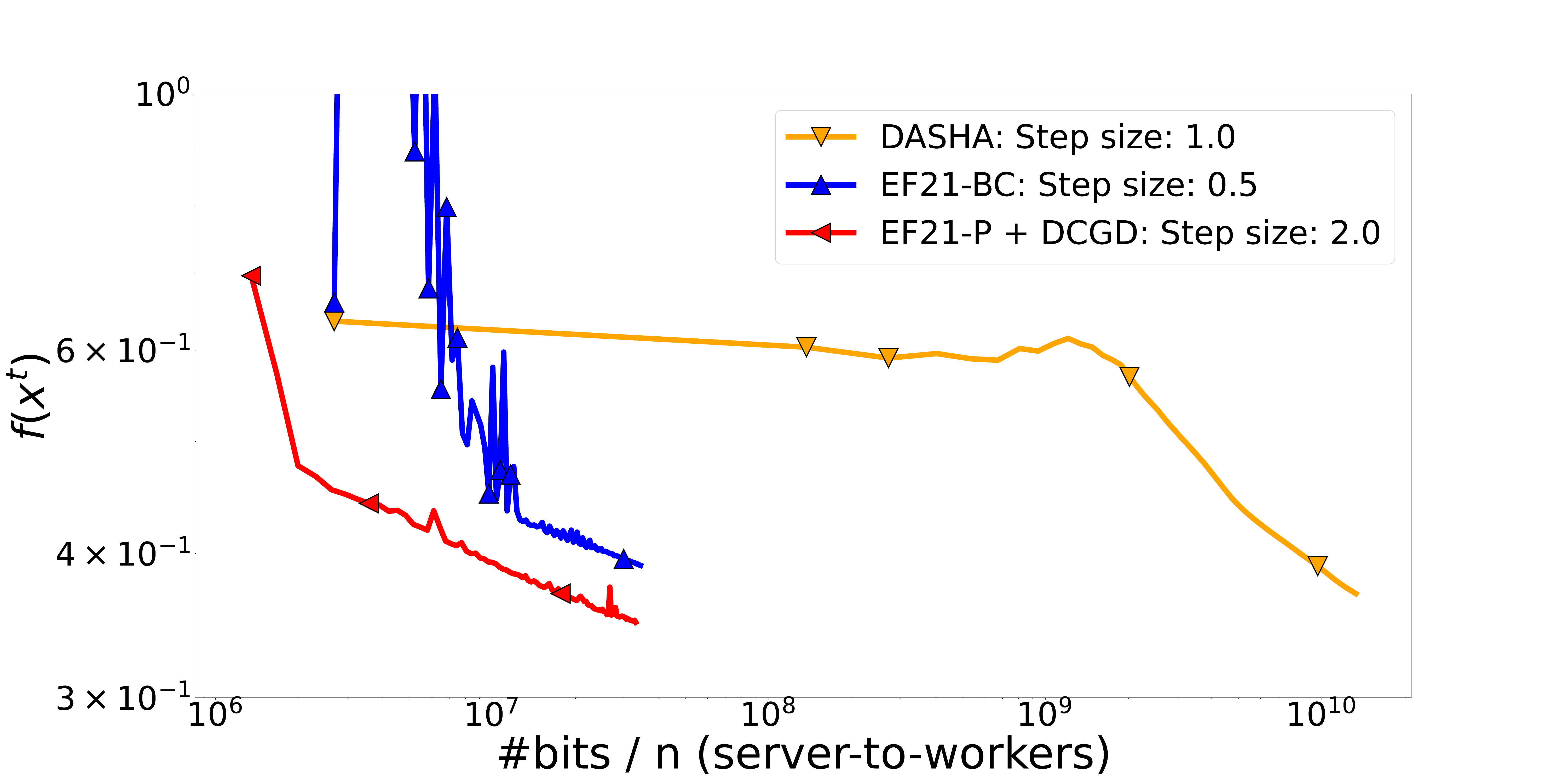}
      \end{subfigure}\hfill
    \caption{Logistic Regression with the nonconvex regularizer and \textit{real-sim} dataset. \# of workers $n = 100.$ $K = 100$ in all compressors.}
    \label{fig:nonconve_exper}
\end{figure}

\clearpage
\section{Useful Identities and Inequalities}

For all $x,y,x_1,\ldots,x_n \in \R^d$, $s>0$ and $\alpha\in(0,1]$, we have:
\begin{align}
    \label{eq:young}
    \norm{x+y}^2 &\leq (1+s) \norm{x}^2 + (1+s^{-1}) \norm{y}^2, \\
    \label{eq:young_2}
    \norm{x+y}^2 &\leq 2 \norm{x}^2 + 2 \norm{y}^2, \\
    \label{eq:fenchel}
    \inp{x}{y} &\leq \frac{\norm{x}^2}{2s} + \frac{s\norm{y}^2}{2}, \\
    \label{eq:ineq1}
    \left( 1 - \alpha \right)\left( 1 + \frac{\alpha}{2} \right) &\leq 1 - \frac{\alpha}{2}, \\
    \label{eq:ineq2}
    \left( 1 - \alpha \right)\left( 1 + \frac{2}{\alpha} \right) &\leq \frac{2}{\alpha}, \\
    \label{eq:inp}
    \inp{a}{b} &= \frac{1}{2} \left( \norm{a}^2 + \norm{b}^2 - \norm{a-b}^2 \right).
\end{align}
\textbf{Tower property:} For any random variables $X$ and $Y$, we have
\begin{align}
\label{eq:tower}
    \Exp{\Exp{X\,|\,Y}} = \Exp{X}.
\end{align}
\textbf{Variance decomposition:} For any random vector $X\in\R^d$ and any non-random $c\in\R^d$, we have
\begin{align}
\label{eq:vardecomp}
   \Exp{\norm{X-c}^2} = \Exp{\norm{X - \Exp{X}}^2} + \norm{\Exp{X}-c}^2.
\end{align}

\begin{lemma}[\cite{nesterov2018lectures}]
   \label{lemma:lipt_func}
   Let $f:\R^d\to \R$ be a function for which Assumptions \ref{ass:lipschitz_constant} and \ref{ass:convex} are satisfied. Then for all $x, y\in\R^d$ we have:
   \begin{align}
       \norm{\nabla f(x) - \nabla f(y)}^2 \leq 2L (f(x) - f(y) - \inp{\nabla f(y)}{x - y}).
   \end{align}
\end{lemma}

\begin{lemma}[\cite{khaled2020better}]
    \label{lemma:lipt_func_nonconvex}
    Let $f$ be a function for which Assumptions \ref{ass:lipschitz_constant} and \ref{ass:lower_bound} are satisfied. Then for all $x, y\in\R^d$ we have:
    \begin{align}
        \norm{\nabla f(x)}^2 \leq 2L (f(x) - f^*).
    \end{align}
\end{lemma}

\clearpage
\section{Proof of Lemma~\ref{lemma:lipt_constants}}
\LEMMALIPTCONSTANTS*
\begin{proof}
    One can show (see \citep{nesterov2003introductory}) that a convex function $f$ is $L$-smooth if and only if either of the two conditions below holds:
    \begin{align*}
        &0 \leq \inp{\nabla f(x) - \nabla f(x)}{x - y} \leq L \norm{x - y}^2, \quad \forall x,y\in\R^d,\\
        &\norm{\nabla f(x) - \nabla f(x)}^2 \leq L \inp{\nabla f(x) - \nabla f(x)}{x - y}, \quad \forall x,y\in\R^d.
    \end{align*}
    For any fixed $i \in [n]$, we have
    \begin{eqnarray*}
        \inp{\nabla f_i(x) - \nabla f_i(y)}{x - y} 
        &\leq &\sum_{i=1}^n\inp{\nabla f_i(x) - \nabla f_i(y)}{x - y} \\
        &= &n \frac{1}{n}\sum_{i=1}^n\inp{\nabla f_i(x) - \nabla f_i(y)}{x - y} \\
        &=&n\inp{\nabla f(x) - \nabla f(y)}{x - y} \\
        &\leq &n\norm{\nabla f(x) - \nabla f(y)}\norm{x - y} \\
        &\overset{\eqref{ass:lipschitz_constant}}{\leq} &n L \norm{x-y}^2.
    \end{eqnarray*}
    Thus $L_i \leq n L$ and $L_{\max} \leq n L.$ Next,
    \begin{eqnarray*}
        \frac{1}{n} \sum_{i=1}^n \norm{\nabla f_i(x) - \nabla f_i(y)}^2
        &\leq& \frac{1}{n} \sum_{i=1}^n L_i \inp{\nabla f_i(x) - \nabla f_i(y)}{x - y} \\
        &\leq &L_{\max} \frac{1}{n} \sum_{i=1}^n \inp{\nabla f_i(x) - \nabla f_i(y)}{x - y} \\
        &= &L_{\max} \inp{\nabla f(x) - \nabla f(y)}{x - y} \\
        &\leq &L_{\max} \norm{\nabla f(x) - \nabla f(y)} \norm{x - y} \\
        &\overset{\eqref{ass:lipschitz_constant}}{\leq}& L_{\max} L \norm{x - y}^2,
    \end{eqnarray*}
    and hence $\widehat{L} \leq \sqrt{L_{\max} L}.$ Using Jensen's inequality, we have
    \begin{align*}
        \norm{\nabla f(x) - \nabla f(y)}^2
        \leq \frac{1}{n} \sum_{i=1}^n \norm{\nabla f_i(x) - \nabla f_i(y)}^2 \leq \widehat{L}^2 \norm{x - y}^2.
    \end{align*}
    Thus $L \leq \widehat{L}.$ Finally, $\widehat{L} \leq L_{\max}$ follows from
    \begin{align*}
        \frac{1}{n} \sum_{i=1}^n \norm{\nabla f_i(x) - \nabla f_i(y)}^2
        \leq \frac{1}{n} \sum_{i=1}^n L_i^2 \norm{x - y}^2 
       \leq L_{\max}^2 \norm{x - y}^2.
    \end{align*}
\end{proof}

\clearpage
\section{Convergence of \algname{EF21-P} in the Strongly Convex Regime} \label{sec:EF21-P_convex}

 We now provide the convergence rate of \algnamebig{EF21-P} 
 in the strongly convex case.
 \begin{theorem}
    \label{theorem:ef21_strong}
    Let Assumptions~\ref{ass:lipschitz_constant} and \ref{ass:convex} hold, set $w^0 = x^0$ and choose $\gamma \leq \frac{\alpha}{16 L}.$ Then \algnamebig{EF21-P} returns $x^{T}$ such that
    \begin{align*}
        &\frac{1}{2\gamma}\Exp{\norm{x^{T} - x^*}^2} + \Exp{f(x^{T}) - f(x^*)} \leq \left(1 - \frac{\gamma \mu}{2}\right)^T\left(\frac{1}{2\gamma}\Exp{\norm{x^0 - x^*}^2} + \left(f(x^{0}) - f(x^{*})\right)\right).
     \end{align*}
     Moreover, $\Exp{\norm{w^t-x^*}^2} \to 0$ as $t\to \infty$.
 \end{theorem}

 Theorem~\ref{theorem:ef21_strong} states that \algnamebig{EF21-P} will return an $\varepsilon$-solution after $\cO\left(\frac{L}{\alpha\mu}\log\nicefrac{1}{\varepsilon}\right)$ steps. Comparing to \algname{GD}'s rate $\cO\left(\frac{L}{\mu}\log\nicefrac{1}{\varepsilon}\right),$ one can see that \algnamebig{EF21-P} converges $\nicefrac{1}{\alpha}$ times slower.

 \begin{proof}
    First, let us note that
    \begin{align}
    \label{eq:equalconv1}
       \norm{x^t - x^*}^2 - &\norm{x^{t+1} - x^*}^2 - \norm{x^{t+1} - x^{t}}^2 \nonumber \\
       &=\inp{x^t - x^{t+1}}{x^t - 2x^* + x^{t+1}} - \inp{x^{t+1} - x^{t}}{x^{t+1} - x^{t}} \nonumber \\
       &=2\inp{x^t - x^{t+1}}{x^{t+1} - x^*} \nonumber \\
       &=2\gamma\inp{\nabla f(w^t)}{x^{t+1} - x^*}.
    \end{align}
    Using $L$-smoothness of $f$ (Assumption~\ref{ass:lipschitz_constant}), we obtain
    \begin{eqnarray*}
        f(x^{t+1}) &\leq & f(w^t) + \inp{\nabla f(w^t)}{x^{t+1} - w^t} + \frac{L}{2}\norm{x^{t+1} - w^t}^2 \nonumber \\
        &\overset{\textnormal{convexity}}{\leq}& f(x^*) + \inp{\nabla f(w^t)}{x^{t+1} - x^*} - \frac{\mu}{2} \norm{w^t - x^*}^2 + \frac{L}{2}\norm{x^{t+1} - w^t}^2\\
        &\overset{\eqref{eq:equalconv1}}{\leq} & f(x^*) 
        + \frac{1}{2\gamma}\norm{x^t - x^*}^2 - \frac{1}{2\gamma}\norm{x^{t+1} - x^*}^2 - \frac{1}{2\gamma}\norm{x^{t+1} - x^{t}}^2 \\
        &&\qquad- \frac{\mu}{2} \norm{w^t - x^*}^2 + \frac{L}{2}\norm{x^{t+1} - w^t}^2.
    \end{eqnarray*}
    Using \eqref{eq:young_2}, we have 
    \begin{align*}
        \frac{L}{2}\norm{x^{t+1} - w^t}^2 \leq L\norm{x^{t+1} - x^t}^2 + L\norm{w^t - x^t}^2
    \end{align*}
    and
    \begin{align*}
        \frac{\mu}{4} \norm{x^t - x^*}^2 \leq \frac{\mu}{2} \norm{w^t - x^*}^2 + \frac{\mu}{2} \norm{w^t - x^t}^2 \leq \frac{\mu}{2} \norm{w^t - x^*}^2 + L \norm{w^t - x^t}^2,
    \end{align*}
    where we used the fact that $\mu \leq L$.
    Hence
    \begin{align*}
        f(x^{t+1})
        &\leq f(x^*) 
        + \frac{1}{2\gamma}\norm{x^t - x^*}^2 - \frac{1}{2\gamma}\norm{x^{t+1} - x^*}^2 - \frac{1}{2\gamma}\norm{x^{t+1} - x^{t}}^2 \\
        &\qquad- \frac{\mu}{2} \norm{w^t - x^*}^2 + \frac{L}{2}\norm{x^{t+1} - w^t}^2 \\
        &\leq f(x^*) 
        + \frac{1}{2\gamma}\norm{x^t - x^*}^2 - \frac{1}{2\gamma}\norm{x^{t+1} - x^*}^2 - \frac{1}{2\gamma}\norm{x^{t+1} - x^{t}}^2 \\
        &\qquad+ L \norm{w^t - x^t}^2 - \frac{\mu}{4} \norm{x^t - x^*}^2
        + L\norm{x^{t+1} - x^t}^2 + L\norm{w^t - x^t}^2 \\
        &= f(x^*) 
        + \frac{1}{2\gamma}\left(1 - \frac{\gamma \mu}{2}\right) \norm{x^t - x^*}^2
        - \frac{1}{2\gamma} \norm{x^{t+1} - x^*}^2 \\
        &\qquad- \left(\frac{1}{2\gamma} - L\right) \norm{x^{t+1} - x^{t}}^2
        + 2L \norm{w^t - x^t}^2 \\
        &\leq f(x^*) 
        + \frac{1}{2\gamma}\left(1 - \frac{\gamma \mu}{2}\right) \norm{x^t - x^*}^2
        - \frac{1}{2\gamma} \norm{x^{t+1} - x^*}^2
        + 2L \norm{w^t - x^t}^2,
    \end{align*}
    where the last inequality follows from the fact that $\gamma \leq \frac{1}{2L}$.
    Let us denote by $\ExpSub{t+1}{\cdot}$ the expectation conditioned on previous iterations $\{0, \dots, t\}.$ Then
    \begin{align}
    \label{eq:lsmboundef21p}
        \ExpSub{t+1}{f(x^{t+1})} &\leq f(x^*)
        + \frac{1}{2\gamma}\left(1 - \frac{\gamma \mu}{2}\right)\norm{x^t - x^*}^2 \nonumber\\
        &\qquad - \frac{1}{2\gamma}\ExpSub{t+1}{\norm{x^{t+1} - x^*}^2} + 2 L \norm{w^t - x^t}^2.
    \end{align}
    It remains to bound $\ExpSub{t+1}{\norm{w^{t+1} - x^{t+1}}^2}$:
    \begin{eqnarray*}
        \ExpSub{t+1}{\norm{w^{t+1}-x^{t+1}}^2} &=& \ExpSub{t+1}{\norm{w^t + \mathcal{C}^p(x^{t+1} - w^t) - x^{t+1}}^2} \\
        & \overset{\eqref{eq:biased_compressor}}{\leq} & (1-\alpha) \ExpSub{t+1}{\norm{x^{t+1} - w^t}^2} \\
        & = & (1-\alpha) \norm{x^{t} - \gamma \nabla f(w^t) - w^t}^2 \\
        & \overset{\eqref{eq:young}}{\leq} & \left(1-\frac{\alpha}{2}\right) \norm{w^t - x^t}^2
        + \frac{2 \gamma^2}{\alpha} \norm{\nabla f(w^{t})}^2 \\
        & \overset{\eqref{eq:young_2}}{\leq} & \left(1-\frac{\alpha}{2}\right) \norm{w^t - x^t}^2
        + \frac{4 \gamma^2}{\alpha} \norm{\nabla f(w^{t}) - \nabla f(x^{t})}^2 \\
        && \qquad + \frac{4 \gamma^2}{\alpha} \norm{\nabla f(x^{t}) - \nabla f(x^{*})}^2 \\
        &\overset{\eqref{ass:lipschitz_constant}, \eqref{lemma:lipt_func}}   {\leq} & \left(1-\frac{\alpha}{2} + \frac{4 \gamma^2 L^2}{\alpha}\right) \norm{w^t - x^t}^2 + \frac{8 \gamma^2 L}{\alpha} \left(f(x^{t}) - f(x^{*})\right) \\
        &\leq & \left(1-\frac{\alpha}{4}\right) \norm{w^t - x^t}^2
        + \frac{8 \gamma^2 L}{\alpha} \left(f(x^{t}) - f(x^{*})\right),
    \end{eqnarray*}
    where in the last step we assume that $\gamma \leq \frac{\alpha}{4 L}$. Adding a $\frac{16L}{\alpha}$ multiple of the above inequality to \eqref{eq:lsmboundef21p}, we obtain
    \begin{align*}
        &\ExpSub{t+1}{f(x^{t+1})}
        + \frac{16L}{\alpha}\ExpSub{t+1}{\norm{w^{t+1}-x^{t+1}}^2}
        \leq f(x^*)
        + \frac{1}{2\gamma}\left(1 - \frac{\gamma \mu}{2}\right)\norm{x^t - x^*}^2 \\
        &\qquad - \frac{1}{2\gamma}\ExpSub{t+1}{\norm{x^{t+1} - x^*}^2}
        + \frac{16L}{\alpha}\left(1-\frac{\alpha}{8}\right) \norm{w^t - x^t}^2
        + \frac{128 \gamma^2 L^2}{\alpha^2} \left(f(x^{t}) - f(x^{*})\right).
    \end{align*}
    Thus, taking full expectation over both sides of the inequality and considering $\gamma \leq \frac{\alpha}{16L} \leq \frac{\alpha}{4\mu}$ gives
    \begin{align*}
        &\Exp{f(x^{t+1}) - f(x^*)}
        + \frac{1}{2\gamma}\Exp{\norm{x^{t+1} - x^*}^2}
        + \frac{16L}{\alpha}\Exp{\norm{w^{t+1}-x^{t+1}}^2} \\
        &\qquad \leq \left(1 - \frac{\gamma \mu}{2}\right) \left(\Exp{f(x^{t}) - f(x^{*})}
        + \frac{1}{2\gamma} \Exp{\norm{x^t - x^*}^2}
        + \frac{16L}{\alpha} \Exp{\norm{w^t - x^t}^2}\right).
    \end{align*}
    Applying this inequality iteratively and using the assumption $w^0=x^0$ proves the result.
\end{proof}

\clearpage
\section{Convergence of \algnamebig{EF21-P} in the Smooth Nonconvex Regime}
 \label{sec:abc}
 
 \subsection{General convergence theory}
 
 We now move on to study how the \algname{EF21-P} method can be used in the nonconvex regime.
 The analysis relies on the expected smoothness assumption introduced by \cite{khaled2020better}. In their work, they study \algname{SGD} methods, performing iterations of the form
 \begin{align*}
     x^{t+1} &= x^t - \gamma g^t,
 \end{align*}
 where $g^t$ is an unbiased estimator of the true gradient $\nabla f(x^t)$.
 Following \citet{khaled2020better}, we shall assume that $\Exp{g(x)}=\nabla f(x)$. However, in our case, gradients will be evaluated at perturbed points, thus resulting in biased stochastic gradient estimators.
 In particular, we consider the following general update rule, where the stochastic gradients are calculated at points evolving according to the \algname{EF21-P} mechanism, rather than at the current iterate:
 \begin{align}
 \label{eq:nonconvexupdate}
     x^{t+1} &= x^t - \gamma g(w^t), \nonumber \\
     w^{t+1} &= w^t + \cC^{P}(x^{t+1} - w^t).
 \end{align}
 Our result covers a wide range of sources of stochasticity that may be present in $g$. For a detailed discussion of the topic, we refer the reader to the original paper \citep{khaled2020better}.
 
 Throughout this section, we will rely on the following assumptions:
 \begin{assumption}
     \label{ass:unbiased}
     The stochastic gradient $g(x)$ is an unbiased estimator of the true gradient $\nabla f(x)$, i.e., $$\Exp{g(x)} = \nabla f(x)$$ for all $x\in\R^d$.
 \end{assumption}
 \begin{assumption}[From \cite{khaled2020better}]
     \label{ass:ABC}
     There exist constants $A,B,C \geq 0$ such that:
     \begin{align*}
         \Exp{\norm{g(x)}^2} \leq 2A(f(x)-f^*) + B \norm{\nabla f(x)}^2 + C
     \end{align*}
     for all $x\in\R^d$.
 \end{assumption}

 We are ready to state the main theorem:

 \begin{restatable}{theorem}{THEOREMABC}
    \label{theorem:abcrate}
    Let Assumptions \ref{ass:lipschitz_constant}, \ref{ass:lower_bound}, \ref{ass:unbiased} and \ref{ass:ABC} hold and set $w^0 = x^0.$
    Fix $\varepsilon>0$ and choose the stepsize
    \begin{align*}
        \gamma &= \min\left\{
            \frac{\alpha}{8 L},
            \frac{1}{4 BL},
            \frac{1}{\sqrt{2ALT}},
            \frac{\varepsilon}{16CL}
            \right\}.
    \end{align*}
    Then
    \begin{align}
        T \geq 
        \frac{48\Delta_{0} L}{\varepsilon}
    \max\left\{
    \frac{8}{\alpha},4B,
    \frac{96\Delta_{0}A}{\varepsilon},
    \frac{16C}{\varepsilon}
    \right\} \quad \Rightarrow \quad 
        \min \limits_{0\leq t\leq T-1} \Exp{\norm{\nabla f(x^{t})}^2}
        \leq \varepsilon. \label{eq:norm_eps}
    \end{align}
\end{restatable}

Note that by taking $A=C=0$ and $B=1$, one gets the $\cO(\frac{L}{\alpha\varepsilon })$ rate for \algname{EF21-P} in the nonconvex setting. Namely, under Assumptions~\ref{ass:lipschitz_constant} and \ref{ass:lower_bound}, for $x^0 = w^0$ and $0<\gamma\leq \frac{\alpha}{8 L}$, we have $\min_{0\leq t\leq T-1} \Exp{\norm{\nabla f(x^{t})}^2} \leq \varepsilon$ as soon as $T \geq \frac{384\Delta_{0} L}{\alpha\varepsilon}$.
 
We now apply the above result to the combination of \algname{EF21-P} perturbation of the model and \algname{DCGD} \citep{khaled2020better} (\algname{EF21-P + DCGD}). Suppose that the iterates follow the update \eqref{eq:nonconvexupdate} (see also Algorithm~\ref{algorithm:dcgd_ef21_p}), where
 \begin{align}
 \label{eq:abcg2}
     g(x) &= \frac{1}{n} \sum_{i=1}^n \cC_i \left(g_i(x) \right)
 \end{align}
 and each stochastic gradient $g_i(x)$ is an unbiased estimator of the true gradient $\nabla f_i(x)$ (i.e., $\Exp{g_i(x)} = \nabla f_i(x)$).

 \begin{proposition}
 \label{prop:abccases}
 Suppose that the gradient estimator $g(x)$ is constructed via \eqref{eq:abcg2} and that Assumption \ref{ass:workers_lipschitz_constant} holds. Let $\Delta^* \eqdef \frac{1}{n} \sum_{i=1}^n (f^* - f^*_i)$. Then:
 \begin{enumerate}
     \item \label{example_full_grad}
     For $g_i(x) = \nabla f_i(x)$, Assumption \ref{ass:ABC} is satisfied with
     $A = \frac{1}{n} \omega L_{max}$, $B = 1$ and $C = 2A\Delta^*.$
     \item In the same setting as in part \ref{example_full_grad}, assuming additionally that Assumption \ref{ass:strong_growth} holds,  Assumption \ref{ass:ABC} is satisfied with $A=C=0$ and $B=\frac{D \omega}{n} + 1$.
     \item \label{example_bdd_var}
     Assume that each stochastic gradient $g_i$ has bounded variance, (i.e., $\Exp{\norm{g_i(x) - \nabla f_i(x)}^2} \leq \sigma^2$). Then Assumption \ref{ass:ABC} is satisfied with $A = \frac{1}{n}\omega L_{max}$, $B = 1$ and $C = 2A\Delta^* + \frac{\omega+1}{n} \sigma^2$.
     \item Suppose that $\Exp{\norm{g_i(x) - \nabla f_i(x)}^2} \leq \sigma^2$ and $f_i=f$ for all $i\in[n]$. Then Assumption \ref{ass:ABC} is satisfied with $A = 0$, $B = \frac{\omega}{n}+1$ and $C = \frac{\omega+1}{n}\sigma^2$.
 \end{enumerate}
 \end{proposition}
 
 In Section~\ref{sec:nonconvex}, we apply Proposition~\ref{prop:abccases} and state the corresponding theorems.

\subsection{Proof of the convergence result}

We will need the following two lemmas:

\begin{lemma}
    \label{abcrecursion}
    Consider sequences $(\delta^t)_t$, $(r^t)_t$ and $(s^t)_t$ such that $\delta^t, r^t, s^t \geq 0$ for all $t\geq 0$ and $s^0=0$. Suppose that
    \begin{align}
    \label{eq:abcrecursion}
        \delta^{t+1} + & a s^{t+1}
        \leq b \delta^{t} + a s^t - c r^t + d,
    \end{align}
    where $a,b,c,d$ are non-negative constants and $b \geq 1$. Then for any $T \geq 1$
    \begin{align*}
        \min_{0\leq t\leq T-1} r^t
        \leq \frac{b^T}{cT} \delta^{0} + \frac{d}{c}.
    \end{align*}
\end{lemma}

\begin{proof}
The proof follows similar steps as the proof of Lemma $2$ of \cite{khaled2020better} and we provide it for completeness. Let us fix $w_{-1}>0$ and define $w_t=\frac{w_{t-1}}{b}$. Multiplying \eqref{eq:abcrecursion} by $w_t$ gives
\begin{align*}
    w_t\delta^{t+1}
    + a w_t s^{t+1}
    & \leq b w_t \delta^{t}
    + a w_t s^t
    - c w_t r^t
    + d w_t \\
    &\leq w_{t-1} \delta^{t}
    + a w_{t-1} s^t
    - c w_t r^t
    + d w_t.
\end{align*}
Summing both sides of the inequality for $t=0,\ldots,T-1$, we obtain
\begin{align*}
    w_{T-1} \delta^{T}
    + & a w_{T-1} s^{T}
    \leq w_{-1} \delta^{0}
    + a w_{-1} s^0
    - c \sum_{t=0}^{T-1} w_t r^t
    + d \sum_{t=0}^{T-1} w_t.
\end{align*}
Rearranging and using the assumption that $s^0=0$ and non-negativity of $s^t$ gives
\begin{align*}
    c \sum_{t=0}^{T-1} w_t r^t
    + w_{T-1} \delta^{T}
    & \leq w_{-1} \delta^{0}
    + a w_{-1} s^0
    - a w_{T-1} s^{T}
    + d \sum_{t=0}^{T-1} w_t \\
    &\leq w_{-1} \delta^{0}
    + d \sum_{t=0}^{T-1} w_t.
\end{align*}
Next, using the non-negativity of $\delta^t$ and $w_t$, we have
\begin{align*}
    c \sum_{t=0}^{T-1} w_t r^t
    \leq c \sum_{t=0}^{T-1} w_t r^t
    + w_{T-1} \delta^{T}
    &\leq w_{-1} \delta^{0}
    + d \sum_{t=0}^{T-1} w_t.
\end{align*}
Letting $W_T \eqdef \sum_{t=0}^{T-1} w_t$ and dividing both sides of the inequality by $W_T$, we obtain
\begin{align*}
    c \min_{0\leq t\leq T-1} r^t
    \leq \frac{c}{W_T} \sum_{t=0}^{T-1} w_t r^t
    \leq \frac{w_{-1}}{W_T} \delta^{0}
    + d.
\end{align*}
Using the fact that
\begin{align*}
    W_T = \sum_{t=0}^{T-1} w_t
    \geq \sum_{t=0}^{T-1} \min_{0\leq t\leq T-1} w_t
    = T w_{T-1}
    = \frac{T w_{-1}}{b^T},
\end{align*}
we can finish the proof.
\end{proof}

\begin{lemma}
    \label{lemma:abc}
    Let Assumptions \ref{ass:lipschitz_constant}, \ref{ass:lower_bound}, \ref{ass:unbiased} and \ref{ass:ABC} hold, set $w^0 = x^0,$ and choose
    \begin{align*}
        \gamma &\leq \min\left\{
            \frac{1}{4 A},
            \frac{1}{4 BL},
            \frac{\alpha}{8 L}
            \right\}.
    \end{align*}
    Then
    \begin{align}
        \label{eq:lemma_rate}
        \min_{0\leq t\leq T-1} \Exp{\norm{\nabla f(x^{t})}^2}
        \leq \frac{8\left(1 + 2AL \gamma^2\right)^T}{\gamma T} \Delta_{0}
        + 8CL \gamma.
    \end{align}
\end{lemma}

\begin{proof}
First, $L$-smoothness of $f$ implies that
\begin{align}
\label{eq:fwfx}
    f(w^{t}) &\leq f(x^{t}) + \inp{\nabla f(x^{t})}{w^{t} - x^{t}} + \frac{L}{2}\norm{w^{t} - x^{t}}^2 \nonumber \\
    &\overset{\eqref{eq:fenchel}}{\leq} f(x^{t}) + \frac{1}{2L} \norm{\nabla f(x^{t})}^2 + L \norm{w^{t} - x^{t}}^2
\end{align}
and
\begin{align*}
    f(x^{t+1})
    & \leq f(x^t) 
    + \langle\nabla f(x^t),x^{t+1} - x^{t}\rangle
    + \frac{L}{2} \|x^{t+1} - x^{t}\|^2 \\
    & = f(x^t) 
    - \gamma \langle\nabla f(x^t), g(w^{t})\rangle
    + \frac{L \gamma^2}{2} \|g(w^t)\|^2.
\end{align*}
Using the fact that $g(x)$ is an unbiased estimator of the true gradient, subtracting $f^*$ from both sides of the latter inequality and taking expectation given iterations $\{0, \dots, t\}$, we obtain
\begin{eqnarray*}
    \ExpSub{t+1}{f(x^{t+1})-f^*}
    &\leq& f(x^t) - f^*
    - \gamma \langle\nabla f(x^t), \nabla f(w^{t})\rangle
    + \frac{L \gamma^2}{2} \ExpSub{t+1}{\norm{g(w^t)}^2} \\
    & \overset{\eqref{ass:ABC}, \eqref{eq:inp}}{\leq}& f(x^t) - f^*
    - \frac{\gamma}{2} \|\nabla f(x^t)\|^2
    - \frac{\gamma}{2} \|\nabla f(w^t)\|^2
    + \frac{\gamma}{2} \|\nabla f(x^t) - \nabla f(w^t)\|^2 \\
    && \qquad + \frac{L \gamma^2}{2} \left( 2A(f(w^t)-f^*) + B \norm{\nabla f(w^t)}^2 + C \right) \\
    & \overset{\eqref{ass:lipschitz_constant}}{\leq}& f(x^t) - f^*
    - \frac{\gamma}{2} \|\nabla f(x^t)\|^2
    - \frac{\gamma}{2} \|\nabla f(w^t)\|^2
    + \frac{\gamma L^2}{2} \|x^t-w^t\|^2 \\
    && \qquad + AL \gamma^2(f(w^t)-f^*)
    + \frac{BL \gamma^2}{2}\norm{\nabla f(w^t)}^2
    + \frac{CL \gamma^2}{2} \\
    & = & f(x^t) - f^*
    - \frac{\gamma}{2} \|\nabla f(x^t)\|^2
    - \frac{\gamma}{2} \left( 1 - BL \gamma \right) \norm{\nabla f(w^t)}^2
    + \frac{L^2 \gamma}{2} \|x^t-w^t\|^2 \\
    && \qquad + AL \gamma^2(f(w^t)-f^*)
    + \frac{CL \gamma^2}{2} \\
    & \overset{\eqref{eq:fwfx}}{\leq} & f(x^t) - f^*
    - \frac{\gamma}{2} \|\nabla f(x^t)\|^2
    - \frac{\gamma}{2} \left( 1 - BL\gamma \right) \norm{\nabla f(w^t)}^2
    + \frac{L^2 \gamma}{2} \|x^t-w^t\|^2 \\
    && \qquad +  AL \gamma^2 \left(f(x^{t}) + \frac{1}{2L} \norm{\nabla f(x^{t})}^2 + L \norm{w^{t} - x^{t}}^2 - f^*\right)
    + \frac{CL \gamma^2}{2} \\
    & = &(1+AL \gamma^2) \left(f(x^{t}) - f^*\right)
    - \frac{\gamma}{2} \left( 1 -  A\gamma \right)
    \norm{\nabla f(x^{t})}^2
    - \frac{\gamma}{2} \left( 1 - BL\gamma \right) \norm{\nabla f(w^t)}^2 \\
    && \qquad
    + L^2 \gamma \left( \frac{1}{2} +  A\gamma \right) \norm{w^{t} - x^{t}}^2
    + \frac{CL \gamma^2}{2}.
\end{eqnarray*}
Hence, taking full expectation, for $\gamma \leq \frac{1}{4A}$, we have
\begin{align}
\label{eq:abcffdiff}
    \Exp{f(x^{t+1})-f^*}
    & \leq (1+AL \gamma^2) \Exp{f(x^{t})-f^*}
    - \frac{\gamma}{4} \Exp{\norm{\nabla f(x^{t})}^2} \\
    & \qquad - \frac{\gamma}{2} \left( 1 - BL \gamma \right) \Exp{\norm{\nabla f(w^t)}^2}
    + L^2 \gamma \Exp{\norm{w^{t} - x^{t}}^2}
    + \frac{CL \gamma^2}{2}. \nonumber
\end{align}
Next, variance decomposition and Assumption \ref{ass:ABC} gives
\begin{eqnarray}
\label{eq:abcgfdiff}
    \Exp{\norm{g(w^t) - \nabla f(w^t)}^2}
    &\overset{\eqref{eq:vardecomp}}{=}& \Exp{\norm{g(w^t)}^2} - \norm{\nabla f(w^t)}^2 \nonumber \\
    &\overset{\eqref{ass:ABC}}{\leq}& 2A(f(w^t)-f^*)
    + (B-1) \norm{\nabla f(w^t)}^2
    + C \nonumber \\
    &\overset{\eqref{eq:fwfx}}{\leq}& 2A \left( f(x^{t}) + \frac{1}{2L} \norm{\nabla f(x^{t})}^2 + L \norm{w^{t} - x^{t}}^2 - f^* \right) \nonumber \\
    &&\qquad + (B-1) \norm{\nabla f(w^t)}^2
    + C \nonumber \\
    &=& 2A \left( f(x^{t}) - f^* \right)
    +  \frac{A}{L} \norm{\nabla f(x^{t})}^2 + 2AL \norm{w^{t} - x^{t}}^2 \nonumber \\
    &&\qquad + (B-1) \norm{\nabla f(w^t)}^2
    + C.
\end{eqnarray}
Therefore, using the unbiasedness of $g(x)$, we can bound the expected distance between $w^{t+1}$ and $x^{t+1}$ as
\begin{eqnarray*}
    \Exp{\norm{w^{t+1}-x^{t+1}}^2}
    &=& \Exp{\norm{w^t + \mathcal{C}_p(x^{t+1} - w^t) - x^{t+1}}^2} \\
    &\overset{\eqref{eq:biased_compressor}}{\leq}& (1-\alpha) \Exp{\norm{x^{t+1} - w^t}^2} \\
    &= &\left(1-\alpha\right) \Exp{\norm{x^{t} - \gamma g^t - w^t}^2} \\
    &\overset{\eqref{eq:vardecomp}}{=} &\left(1-\alpha\right) \gamma^2 \Exp{\norm{g^t - \nabla f(w^t)}^2}
    + \left(1-\alpha\right) \Exp{\norm{x^{t} - \gamma \nabla f(w^t) - w^t}^2} \\
    &\overset{\eqref{eq:young}, \eqref{eq:ineq1}, \eqref{eq:ineq2}}{\leq}& \left(1-\alpha\right) \gamma^2 \Exp{\norm{g^t - \nabla f(w^t)}^2}
    + \left(1-\frac{\alpha}{2}\right) \Exp{\norm{x^t - w^t}^2} \\
    &&\qquad + \frac{2 \gamma^2}{\alpha} \Exp{\norm{\nabla f(w^{t})}^2} \\
    &\overset{\eqref{eq:abcgfdiff}}{\leq} &2A\left(1-\alpha\right) \gamma^2 \left( f(x^{t}) - f^* \right)
    + \frac{A \left(1-\alpha\right) \gamma^2}{L} \norm{\nabla f(x^{t})}^2 \\
    &&\qquad + 2A L \left(1-\alpha\right) \gamma^2 \norm{w^{t} - x^{t}}^2
    + (B-1)\left(1-\alpha\right) \gamma^2 \norm{\nabla f(w^t)}^2 \\
    &&\qquad + C \left(1-\alpha\right) \gamma^2
    + \left(1-\frac{\alpha}{2}\right) \Exp{\norm{x^t - w^t}^2}
    + \frac{2 \gamma^2}{\alpha} \Exp{\norm{\nabla f(w^{t})}^2}.
\end{eqnarray*}
Hence, taking expectation, for $\gamma \leq \sqrt{\frac{\alpha}{8AL(1-\alpha)}}$
\begin{align}
\label{eq:abcwxdiff}
    \Exp{\norm{w^{t+1}-x^{t+1}}^2}
    &\leq 2A\left(1-\alpha\right) \gamma^2 \Exp{ f(x^{t}) - f^* }
    + \frac{A \left(1-\alpha\right) \gamma^2}{L} \Exp{\norm{\nabla f(x^{t})}^2} \nonumber \\
    &\qquad + \gamma^2 \left( \frac{2}{\alpha} + (B-1)\left(1-\alpha\right) \right) \Exp{\norm{\nabla f(w^{t})}^2} \nonumber \\
    &\qquad+ \left(1 - \frac{\alpha}{2} + 2A L \left(1-\alpha\right) \gamma^2\right) \Exp{\norm{x^t - w^t}^2}
    + C \left(1-\alpha\right) \gamma^2 \nonumber \\
    &\leq 2A\left(1-\alpha\right) \gamma^2 \Exp{ f(x^{t}) - f^* }
    + \frac{A \left(1-\alpha\right) \gamma^2}{L} \Exp{\norm{\nabla f(x^{t})}^2} \nonumber \\
    &\qquad + \gamma^2 \left( \frac{2}{\alpha} + (B-1)\left(1-\alpha\right) \right) \Exp{\norm{\nabla f(w^{t})}^2} \nonumber \\
    &\qquad+ \left(1 - \frac{\alpha}{4}\right) \Exp{\norm{x^t - w^t}^2}
    + C \left(1-\alpha\right) \gamma^2.
\end{align}
Adding a $\frac{4L^2\gamma}{\alpha}$ multiple of \eqref{eq:abcwxdiff} to \eqref{eq:abcffdiff}, we obtain
\begin{align*}
    \Exp{f(x^{t+1})-f^*}
    + & \frac{4L^2\gamma}{\alpha} \Exp{\norm{w^{t+1}-x^{t+1}}^2} \\
    & \leq \left( 1 + AL \gamma^2 + \frac{8AL^2(1-\alpha) \gamma^3}{\alpha} \right) \Exp{f(x^{t})-f^*} \\
    &\qquad - \frac{\gamma}{4} \left( 1 - \frac{16AL \left(1-\alpha\right) \gamma^2}{\alpha} \right) \Exp{\norm{\nabla f(x^{t})}^2} \\
    &\qquad - \frac{\gamma}{2} \left( 1 - BL \gamma
    - \frac{8 L^2 \gamma^2}{\alpha} \left( \frac{2}{\alpha} + (B-1)\left(1-\alpha\right) \right)
    \right) \Exp{\norm{\nabla f(w^t)}^2} \\
    &\qquad + \frac{4L^2\gamma}{\alpha} \Exp{\norm{w^{t} - x^{t}}^2}
    + \frac{CL \gamma^2}{2}
    + \frac{4C L^2(1-\alpha)\gamma^3}{\alpha}.
\end{align*}
Then, provided that
\begin{align*}
    \gamma &\leq \min\left\{
    \frac{1}{4 BL},
    \frac{\alpha}{8 L},
    \sqrt{\frac{\alpha}{32 (B - 1) (\alpha - 1) L^2}},
    \sqrt{\frac{\alpha}{32AL(1-\alpha)}},
    \right\} \\
    &=\min\left\{
        \frac{1}{4 BL},
        \frac{\alpha}{8 L},
        \sqrt{\frac{\alpha}{32AL(1-\alpha)}},
        \right\},
\end{align*}
(where we used $\min\{a, b\}\leq \sqrt{ab}$ for all $a, b \in \R^+$), this gives
\begin{align*}
    & \Exp{f(x^{t+1})-f^*}
    + \frac{4L^2\gamma}{\alpha} \Exp{\norm{w^{t+1}-x^{t+1}}^2}
    \leq \left( 1 + 2AL \gamma^2 \right) \Exp{f(x^{t})-f^*} \\
    &\qquad - \frac{\gamma}{8} \Exp{\norm{\nabla f(x^{t})}^2}
    + \frac{4L^2\gamma}{\alpha} \Exp{\norm{w^{t} - x^{t}}^2}
    + CL \gamma^2.
\end{align*}
Denoting $a \eqdef \frac{4L^2\gamma}{\alpha}$, $b \eqdef 1 + 2AL \gamma^2$, $c \eqdef \frac{\gamma}{8}$ and $d \eqdef CL \gamma^2$, this is equivalent to
\begin{align}
    \delta^{t+1} + & a s^{t+1}
    \leq b \delta^{t} + a s^t - c r^t + d,
\end{align}
where $\delta^t\eqdef \Exp{f(x^{t}) - f^*}$, $r^t \eqdef \Exp{\norm{\nabla f(x^{t})}^2}$ and $s^t \eqdef \Exp{\norm{w^{t} - x^{t}}^2}$. Hence, using Lemma \ref{abcrecursion}, for any $T \geq 1$
\begin{align*}
    \min_{0\leq t\leq T-1} r^t
    \leq \frac{b^T}{cT} \delta^{0} + \frac{d}{c},
\end{align*}
which proves \eqref{eq:lemma_rate}. In the proof, we have the following constraints on $\gamma$:
\begin{align*}
    \gamma &\leq \min\left\{
        \frac{1}{4 A},
        \frac{1}{4 BL},
        \frac{\alpha}{8 L},
        \sqrt{\frac{\alpha}{32AL(1-\alpha)}}
        \right\}.
\end{align*}
Using the inequality $\min\{a, b\}\leq \sqrt{ab}$ for all $a, b \in \R^+,$ this can be simplified to
\begin{align*}
    \gamma &\leq \min\left\{
        \frac{1}{4 A},
        \frac{1}{4 BL},
        \frac{\alpha}{8 L}
        \right\}.
\end{align*}
\end{proof}

\THEOREMABC*

\begin{proof}
By Lemma~\ref{lemma:abc}, we have
\begin{align*}
    \min_{0\leq t\leq T-1} \Exp{\norm{\nabla f(x^{t})}^2}
    \leq \frac{8\left(1 + 2AL \gamma^2\right)^T}{\gamma T} \Delta_{0}
    + 8CL \gamma
\end{align*}
provided that
$\gamma \leq \min\left\{
    \frac{1}{4 A},
    \frac{1}{4 BL},
    \frac{\alpha}{8 L}
    \right\}.$
Now, using the fact that $1+x\leq e^x$ and the assumption $\gamma \leq \frac{1}{\sqrt{2ALT}}$, we obtain
\begin{align*}
    \left(1 + 2AL \gamma^2\right)^T
    \leq \exp\left(2ALT \gamma^2\right)
    \leq \exp(1)
    < 3.
\end{align*}
Hence
\begin{align*}
    \min_{0\leq t\leq T-1} \Exp{\norm{\nabla f(x^{t})}^2}
    \leq \frac{24}{\gamma T} \Delta_{0}
    + 8CL \gamma.
\end{align*}
In order to obtain  $\frac{24}{\gamma T} \Delta_{0} + 8CL \gamma \leq \varepsilon$, we require that both terms are no larger than $\frac{\varepsilon}{2}$, which is equivalent to
\begin{align}
\label{eq:abct}
    T \geq \frac{48\Delta_{0}}{\gamma\varepsilon}, \\
    \gamma \leq \frac{\varepsilon}{16CL}.
\end{align}
We thus require that:
\begin{align*}
    \gamma \leq \min\left\{
        \frac{1}{4 A},
        \frac{1}{4 BL},
        \frac{\alpha}{8 L},
        \frac{1}{\sqrt{2ALT}},
        \frac{\varepsilon}{16CL}
        \right\}.
\end{align*}
which, combined with \eqref{eq:abct} gives:
\begin{align*}
    T &\geq 
    \frac{48\Delta_{0}}{\varepsilon}
    \max\left\{
    4 A,
    4BL,\frac{8L}{\alpha},
    \frac{96\Delta_{0}AL}{\varepsilon},
    \frac{16CL}{\varepsilon}
    \right\}. 
\end{align*}
It remains to notice that the term $4A$ can be dropped, thus simplifying the constraints to 
\begin{align*}
    \gamma &\leq \min\left\{
        \frac{1}{4 BL},
        \frac{\alpha}{8 L},
        \frac{1}{\sqrt{2ALT}},
        \frac{\varepsilon}{16CL}
        \right\}.
\end{align*}
and 
\begin{align*}
    T \geq 
    \frac{48\Delta_{0}}{\varepsilon}
\max\left\{
4BL,\frac{8L}{\alpha},
\frac{96\Delta_{0}AL}{\varepsilon},
\frac{16CL}{\varepsilon}
\right\}.
\end{align*}
Indeed, if $\norm{\nabla f(x^0)}^2 \leq \varepsilon,$ then \eqref{eq:norm_eps} holds for any $\gamma > 0.$
Let us now assume that $\norm{\nabla f(x^0)}^2 > \varepsilon.$ The above constraints imply that
$\frac{1}{\sqrt{2 A L T}} \leq \frac{\varepsilon}{96 \Delta_0 A L}.$
Moreover, from Lemma~\ref{lemma:lipt_func_nonconvex}, we know that $\varepsilon < \norm{\nabla f(x^0)}^2 \leq 2L \Delta^0.$ Thus
$\frac{1}{\sqrt{2 A L T}} \leq \frac{1}{48 A}.$ Similarly, we see that $\frac{96\Delta_{0}AL}{\varepsilon} \geq 48 A.$
\end{proof}

\subsection{Proof of Proposition \ref{prop:abccases}}

\begin{proof}
    \begin{enumerate}
    \item Using independence of $\cC_1, \ldots, \cC_n$, we have
    \begin{eqnarray*}
        \Exp{ \norm{g(x)}^2 }
        &= &\Exp{ \norm{\frac{1}{n} \sum_{i=1}^n \cC_i \left(\nabla f_i(x)\right)}^2 } \\
        &\stackrel{\eqref{eq:vardecomp}}{=}& \Exp{ \norm{ \frac{1}{n} \sum_{i=1}^n \left(\cC_i \left(\nabla f_i(x)\right) - \nabla f_i(x) \right) }^2}
        + \norm{ \nabla f(x)}^2 \\
        &=& \frac{1}{n^2} \sum_{i=1}^n \Exp{ \norm{\cC_i \left(\nabla f_i(x)\right) - \nabla f_i(x)}^2}
        + \norm{ \nabla f(x)}^2 \\
        &\stackrel{\eqref{eq:compressor}}{\leq}& \frac{1}{n^2} \sum_{i=1}^n \omega \norm{ \nabla f_i(x) }^2
        + \norm{ \nabla f(x)}^2 \\
        &\stackrel{\eqref{lemma:lipt_func_nonconvex}}{\leq}& \frac{\omega}{n^2} \sum_{i=1}^n 2L_i (f_i(x) - f_i^*)
        + \norm{ \nabla f(x)}^2 \\
        &\leq& \frac{2\omega L_{max}}{n^2} \sum_{i=1}^n (f_i(x) - f_i^*)
        + \norm{ \nabla f(x)}^2\\
        &=& 2 A (f(x) - f^*) + \norm{ \nabla f(x)}^2 + 2A\Delta^*,
    \end{eqnarray*}
    where $A \eqdef \frac{\omega L_{max}}{n}$.
    \item Starting as in part \ref{example_full_grad} of the proof, we obtain
    \begin{align*}
        \Exp{ \norm{g(x)}^2 }
        \leq \frac{1}{n^2} \sum_{i=1}^n \omega \norm{ \nabla f_i(x) }^2
        + \norm{ \nabla f(x)}^2 
        \stackrel{\eqref{ass:strong_growth}}{\leq} \left( \frac{D \omega}{n} + 1 \right) \norm{ \nabla f(x) }^2.
    \end{align*}
    \item First let us note that
    \begin{align*}
        \Exp{\norm{g_i(x)}^2}
        \stackrel{\eqref{eq:vardecomp}}{=} \Exp{\norm{g_i(x) - \nabla f_i(x)}^2} + \norm{\nabla f_i(x)}^2 
        \leq \sigma^2 + \norm{\nabla f_i(x)}^2.
    \end{align*}
    Following steps similar to the proof of Proposition $4$ of \cite{khaled2020better}, unbiasedness of the stochastic gradients gives
    \begin{eqnarray*}
        \Exp{ \norm{g(x)}^2 }
        &\overset{\eqref{eq:tower}}{=}& \Exp{ \Exp{ \norm{\frac{1}{n} \sum_{i=1}^n \cC_i \left(g_i(x)\right)}^2 \,|\, g_1(x),\ldots,g_n(x) }} \\
        &\overset{\eqref{eq:vardecomp}}{=}& \Exp{ \Exp{ \norm{\frac{1}{n} \sum_{i=1}^n \left(\cC_i \left(g_i(x)\right) - g_i(x)\right)}^2 \,|\, g_1(x),\ldots,g_n(x) }
        + \norm{\frac{1}{n} \sum_{i=1}^n g_i(x)}^2} \\
        &\overset{\eqref{eq:vardecomp}}{=}& \Exp{ \frac{1}{n^2} \sum_{i=1}^n \Exp{ \norm{\cC_i \left(g_i(x)\right) - g_i(x)}^2 \,|\, g_1(x),\ldots,g_n(x)}} \\
        &&\qquad + \Exp{\norm{\frac{1}{n} \sum_{i=1}^n \left(g_i(x) - \nabla f_i(x)\right)}^2} + \norm{\nabla f(x)}^2 \\
        &\leq& \frac{\omega}{n^2} \sum_{i=1}^n \Exp{\norm{g_i(x)}^2}
        + \Exp{\norm{\frac{1}{n} \sum_{i=1}^n \left(g_i(x) - \nabla f_i(x)\right)}^2} + \norm{\nabla f(x)}^2 \\
        &\leq& \frac{\omega}{n^2} \sum_{i=1}^n \left(\norm{\nabla f_i(x)}^2 + \sigma^2\right)
        + \frac{1}{n^2} \sum_{i=1}^n \Exp{\norm{g_i(x) - \nabla f_i(x)}^2} + \norm{\nabla f(x)}^2 \\
        &\overset{\eqref{lemma:lipt_func_nonconvex}}{\leq}& \frac{\omega}{n^2} \sum_{i=1}^n \left( 2L_i \left(f_i(x)-f_i^*\right) + \sigma^2\right)
        + \frac{\sigma^2}{n} + \norm{\nabla f(x)}^2 \\
        &= &2A\left(f(x)-f^*\right)
        + \norm{\nabla f(x)}^2 + C,
    \end{eqnarray*}
    where $A \eqdef \frac{1}{n}\omega L_{max}$ and $C \eqdef 2A\Delta^* + \frac{\omega+1}{n} \sigma^2$.
    \item Starting as in part \ref{example_bdd_var} and using the assumption $f_i=f$, we have:
    \begin{eqnarray*}
        \Exp{ \norm{g(x)}^2 }
        &\leq& \frac{\omega}{n^2} \sum_{i=1}^n \Exp{\norm{g_i(x)}^2}
        + \Exp{\norm{\frac{1}{n} \sum_{i=1}^n g_i(x) - \nabla f(x)}^2} + \norm{\nabla f(x)}^2 \\
        &\overset{\eqref{eq:vardecomp}}{=}& \frac{\omega}{n^2} \sum_{i=1}^n \left( \Exp{\norm{g_i(x) - \nabla f(x)}^2} + \norm{\nabla f(x)}^2 \right) \\
        &&\qquad + \frac{1}{n^2} \sum_{i=1}^n \Exp{\norm{g_i(x) - \nabla f(x)}^2}
        + \norm{\nabla f(x)}^2 \\
        &\leq& \frac{\omega+1}{n} \sigma^2
        + \left( \frac{\omega}{n} + 1 \right) \norm{\nabla f(x)}^2.
    \end{eqnarray*}
\end{enumerate}
\end{proof}

\clearpage
\section{Proofs for \algname{EF21-P + DIANA} in the Convex Case}
\label{sec:proofs_diana}

First, we prove an auxiliary theorem:
\begin{theorem}
    \label{theorem:general_diana}
    Let us assume that Assumptions~\ref{ass:lipschitz_constant}, \ref{ass:workers_lipschitz_constant} and \ref{ass:convex} hold, $\beta \in \left[0, \frac{1}{\omega + 1}\right],$ and 
    \begin{align}
       &\gamma \leq \min\left\{\frac{n}{160 \omega L_{\max}}, \frac{\alpha}{100 L}, \frac{\beta}{\mu}\right\}. \label{eq:gamma:diana}
    \end{align}
    Then Algorithm~\ref{algorithm:diana_ef21_p} guarantees that
    \begin{align}
       &\frac{1}{2\gamma}\Exp{\norm{x^{t+1} - x^*}^2} + \Exp{f(x^{t+1}) - f(x^*)} \nonumber\\
       &\quad + \kappa \frac{1}{n}\sum_{i=1}^n \Exp{\norm{h^{t+1}_i - \nabla f_i(x^*)}^2} + \nu \Exp{\norm{w^{t+1}-x^{t+1}}^2} \nonumber\\
       &\leq \frac{1}{2\gamma}\left(1 - \frac{\gamma \mu}{2}\right)\Exp{\norm{x^t - x^*}^2} + \frac{1}{2}\Exp{f(x^{t}) - f(x^{*})} \nonumber\\
       &\quad + \kappa \left(1 - \frac{\gamma \mu}{2}\right)\frac{1}{n}\sum_{i=1}^n \Exp{\norm{h^t_i - \nabla f_i(x^*)}^2} + \nu \left(1 - \frac{\gamma \mu}{2}\right)\Exp{\norm{w^t - x^t}^2}, \label{eq:main_diana}
    \end{align}
    where $\kappa \leq \frac{8 \gamma \omega}{n \beta}$ and $\nu \leq \frac{192\gamma\omega \widehat{L}^2}{n \alpha} + \frac{32L}{\alpha}.$
 \end{theorem}
 
 \begin{proof}
    From $L$-smoothness (Assumption~\ref{ass:lipschitz_constant}) of the function $f$, we have
    \begin{eqnarray*}
       f(x^{t+1}) &\leq& f(w^t) + \inp{\nabla f(w^t)}{x^{t+1} - w^t} + \frac{L}{2}\norm{x^{t+1} - w^t}^2 \\
       &\overset{\textnormal{convexity}}{\leq}& f(x^*) + \inp{\nabla f(w^t)}{x^{t+1} - x^*} - \frac{\mu}{2} \norm{w^t - x^*}^2 + \frac{L}{2}\norm{x^{t+1} - w^t}^2 \\
       &=& f(x^*) + \inp{g^t}{x^{t+1} - x^*} + \inp{\nabla f(w^t) - g^t}{x^{t+1} - x^*} \\
       &&\qquad+ \frac{L}{2}\norm{x^{t+1} - w^t}^2 - \frac{\mu}{2} \norm{w^t - x^*}^2.
    \end{eqnarray*}
    We now reprove a well-known equality from the convex world. Noting that $x^{t+1} = x^t - \gamma g^t,$ we obtain
    \begin{align}
    \label{eq:equalconv}
       \norm{x^t - x^*}^2 - &\norm{x^{t+1} - x^*}^2 - \norm{x^{t+1} - x^{t}}^2 \nonumber \\
       &=\inp{x^t - x^{t+1}}{x^t - 2x^* + x^{t+1}} - \inp{x^{t+1} - x^{t}}{x^{t+1} - x^{t}} \nonumber \\
       &=2\inp{x^t - x^{t+1}}{x^{t+1} - x^*} \nonumber \\
       &=2\gamma\inp{g^t}{x^{t+1} - x^*}.
    \end{align}
    Substituting \eqref{eq:equalconv} in the inequality gives
    \begin{align*}
       f(x^{t+1}) &\leq f(x^*) + \inp{\nabla f(w^t) - g^t}{x^{t+1} - x^*} \\
       &\quad + \frac{1}{2\gamma}\norm{x^t - x^*}^2 - \frac{1}{2\gamma}\norm{x^{t+1} - x^*}^2 - \frac{1}{2\gamma}\norm{x^{t+1} - x^t}^2 \\
       &\quad + \frac{L}{2}\norm{x^{t+1} - w^t}^2 - \frac{\mu}{2} \norm{w^t - x^*}^2.
   \end{align*}
   Next, by \eqref{eq:young_2}, we have 
   \begin{align*}
    \frac{L}{2}\norm{x^{t+1} - w^t}^2 \leq L\norm{x^{t+1} - x^t}^2 + L\norm{w^t - x^t}^2
   \end{align*}
   and
   \begin{align*}
    \frac{\mu}{4} \norm{x^t - x^*}^2 \leq \frac{\mu}{2} \norm{w^t - x^*}^2 + \frac{\mu}{2} \norm{w^t - x^t}^2 \leq \frac{\mu}{2} \norm{w^t - x^*}^2 + L \norm{w^t - x^t}^2,
   \end{align*}
   where we used $L \geq \mu.$ Thus
   \begin{align*}
    f(x^{t+1}) &\leq f(x^*) + \inp{\nabla f(w^t) - g^t}{x^{t+1} - x^*} \\
    &\qquad + \frac{1}{2\gamma}\norm{x^t - x^*}^2 - \frac{1}{2\gamma}\norm{x^{t+1} - x^*}^2 - \frac{1}{2\gamma}\norm{x^{t+1} - x^t}^2  \\
    &\qquad+ L\norm{x^{t+1} - x^t}^2 + L\norm{w^t - x^t}^2 - \frac{\mu}{4} \norm{x^t - x^*}^2 + L \norm{w^t - x^t}^2 \\
    &= f(x^*) + \inp{\nabla f(w^t) - g^t}{x^{t+1} - x^*}
    + \frac{1}{2\gamma}\left(1 - \frac{\gamma \mu}{2}\right)\norm{x^t - x^*}^2 \\
    &\qquad - \frac{1}{2\gamma}\norm{x^{t+1} - x^*}^2 - \left(\frac{1}{2\gamma} - L\right)\norm{x^{t+1} - x^t}^2
    + 2 L \norm{w^t - x^t}^2 \\
    &\leq f(x^*) + \inp{\nabla f(w^t) - g^t}{x^{t+1} - x^*}
    + \frac{1}{2\gamma}\left(1 - \frac{\gamma \mu}{2}\right)\norm{x^t - x^*}^2 \\
    &\qquad - \frac{1}{2\gamma}\norm{x^{t+1} - x^*}^2 + 2 L \norm{w^t - x^t}^2,
    \end{align*}
    where we used the fact that $\gamma \leq \frac{1}{2L}.$
    Then, taking expectation conditioned on previous iterations $\{0, \dots, t\}$, we obtain
    \begin{align*}
       \ExpSub{t+1}{f(x^{t+1})} &\leq f(x^*) + \ExpSub{t+1}{\inp{\nabla f(w^t) - g^t}{x^{t+1} - x^*}} \\
       &\qquad + \frac{1}{2\gamma}\left(1 - \frac{\gamma \mu}{2}\right)\norm{x^t - x^*}^2 - \frac{1}{2\gamma}\ExpSub{t+1}{\norm{x^{t+1} - x^*}^2} + 2 L \norm{w^t - x^t}^2.
    \end{align*}
    From the unbiasedness of the compressors $\cC_i^{D}, $ we have
    \begin{align*}
       \ExpSub{t+1}{g^t} = \nabla f(w^t)
    \end{align*}
    and
    \begin{eqnarray*}
       \ExpSub{t+1}{\inp{\nabla f(w^t) - g^t}{x^{t+1} - x^*}}
       &=& \ExpSub{t+1}{\inp{\nabla f(w^t) - g^t}{x^{t} - \gamma g^t- x^*}} \\
       &=& -\gamma\ExpSub{t+1}{\inp{\nabla f(w^t) - g^t}{g^t}} \\
       &= &\gamma\ExpSub{t+1}{\norm{g^t}^2} - \gamma \norm{\nabla f(w^t)}^2 \\
       &\overset{\eqref{eq:vardecomp}}{=}& \gamma\ExpSub{t+1}{\norm{g^t - \nabla f(w^t)}^2}.
    \end{eqnarray*}
    Therefore
    \begin{align}
       \ExpSub{t+1}{f(x^{t+1})} &\leq f(x^*) + \gamma\ExpSub{t+1}{\norm{g^t - \nabla f(w^t)}^2} \nonumber\\
       &\quad + \frac{1}{2\gamma}\left(1 - \frac{\gamma \mu}{2}\right)\norm{x^t - x^*}^2 - \frac{1}{2\gamma}\ExpSub{t+1}{\norm{x^{t+1} - x^*}^2} + 2 L \norm{w^t - x^t}^2 \label{eq:f_x_t}.
    \end{align}
    Now, we separately consider $\ExpSub{t+1}{\norm{g^t - \nabla f(w^t)}^2}.$ 
    From the independence of compressors, we have
    \begin{align*}
       &\ExpSub{t+1}{\norm{g^t - \nabla f(w^t)}^2}\\
       &=\ExpSub{t+1}{\norm{h^t + \frac{1}{n}\sum_{i=1}^n \cC_i^{D}(\nabla f_i(w^t) - h^t_i) - \nabla f(w^t)}^2} \\
       &=\frac{1}{n^2}\sum_{i=1}^n\ExpSub{t+1}{\norm{\cC_i^{D}(\nabla f_i(w^t) - h^t_i) - \left(\nabla f_i(w^t) - h^t_i\right)}^2} \\
       &\leq\frac{\omega}{n^2}\sum_{i=1}^n \norm{\nabla f_i(w^t) - h^t_i}^2 \\
       &\leq \frac{2 \omega}{n^2}\sum_{i=1}^n \norm{h^t_i - \nabla f_i(x^*)}^2 + \frac{2\omega}{n^2}\sum_{i=1}^n \norm{\nabla f_i(w^t) - \nabla f_i(x^*)}^2 \\
       &\leq \frac{2 \omega}{n^2}\sum_{i=1}^n \norm{h^t_i - \nabla f_i(x^*)}^2 + \frac{4\omega}{n^2}\sum_{i=1}^n \norm{\nabla f_i(w^t) - \nabla f_i(x^t)}^2 + \frac{4\omega}{n^2}\sum_{i=1}^n \norm{\nabla f_i(x^t) - \nabla f_i(x^*)}^2,
   \end{align*}
   where in the last three inequalities, we used \eqref{eq:compressor} and \eqref{eq:young_2}. Next, using Assumption~\ref{ass:workers_lipschitz_constant} and Lemma~\ref{lemma:lipt_func}, we obtain
   \begin{align}
    &\ExpSub{t+1}{\norm{g^t - \nabla f(w^t)}^2} \nonumber\\
    &\leq\frac{2 \omega}{n^2}\sum_{i=1}^n \norm{h^t_i - \nabla f_i(x^*)}^2 + \frac{4\omega \widehat{L}^2}{n} \norm{w^t - x^t}^2 + \frac{8\omega L_{\max}}{n} \left(f(x^t) - f(x^*)\right)\label{eq:g_t}.
   \end{align}
   To construct a Lyapunov function, it remains to bound $\frac{1}{n}\sum_{i=1}^n \norm{h^{t+1}_i - \nabla f_i(x^*)}^2$ and $\norm{w^{t+1} - z^{t+1}}^2$:
 \begin{align*}
    &\frac{1}{n}\sum_{i=1}^n \ExpSub{t+1}{\norm{h^{t+1}_i - \nabla f_i(x^*)}^2}\\
    &\qquad=\frac{1}{n}\sum_{i=1}^n \ExpSub{t+1}{\norm{h^{t}_i + \beta \cC_i^{D}(\nabla f_i(w^t) - h^t_i) - \nabla f_i(x^*)}^2}\\
    &\qquad=\frac{1}{n}\sum_{i=1}^n \norm{h^{t}_i - \nabla f_i(x^*)}^2 + \frac{2\beta}{n}\sum_{i=1}^n \inp{h^{t}_i - \nabla f_i(x^*)}{\ExpSub{t+1}{\cC_i^{D}(\nabla f_i(w^t) - h^t_i)}} \\
    &\qquad\qquad + \frac{\beta^2}{n}\sum_{i=1}^n \ExpSub{t+1}{\norm{\cC_i^{D}(\nabla f_i(w^t) - h^t_i)}^2}\\
    &\qquad\overset{\eqref{eq:compressor}}{\leq}\frac{1}{n}\sum_{i=1}^n \norm{h^{t}_i - \nabla f_i(x^*)}^2 + \frac{2\beta}{n}\sum_{i=1}^n \inp{h^{t}_i - \nabla f_i(x^*)}{\nabla f_i(w^t) - h^t_i} \\
    &\qquad\qquad + \frac{\beta^2(\omega + 1)}{n}\sum_{i=1}^n \norm{\nabla f_i(w^t) - h^t_i}^2\\
    &\qquad\overset{\eqref{eq:inp}}{=}\left(1 - \beta\right)\frac{1}{n}\sum_{i=1}^n \norm{h^{t}_i - \nabla f_i(x^*)}^2 + \frac{\beta}{n}\sum_{i=1}^n \norm{\nabla f_i(w^t) - \nabla f_i(x^*)}^2 \\
    &\qquad\qquad + \frac{\beta \left(\beta(\omega + 1) - 1\right)}{n}\sum_{i=1}^n \norm{\nabla f_i(w^t) - h^t_i}^2 \\
    &\qquad\leq \left(1 - \beta\right)\frac{1}{n}\sum_{i=1}^n \norm{h^{t}_i - \nabla f_i(x^*)}^2 + \frac{\beta}{n}\sum_{i=1}^n \norm{\nabla f_i(w^t) - \nabla f_i(x^*)}^2,
 \end{align*}
 where we use that $\beta \in \left[0, \frac{1}{\omega + 1}\right].$
 Thus, using \eqref{eq:young_2}, Assumption~\ref{ass:workers_lipschitz_constant} and Lemma~\ref{lemma:lipt_func}, we have
 \begin{align}
    &\frac{1}{n}\sum_{i=1}^n \ExpSub{t+1}{\norm{h^{t+1}_i - \nabla f_i(x^*)}^2} \nonumber\\
    &\qquad\leq \left(1 - \beta\right)\frac{1}{n}\sum_{i=1}^n \norm{h^{t}_i - \nabla f_i(x^*)}^2 + 2\beta \widehat{L}^2 \norm{w^t - x^t}^2 + 4\beta L_{\max} \left(f(x^t) - f(x^*)\right) \label{eq:h_t}.
 \end{align}
 
 It remains to bound $\ExpSub{t+1}{\norm{w^{t+1} - x^{t+1}}^2}:$
 \begin{eqnarray*}
    \ExpSub{t+1}{\norm{w^{t+1}-x^{t+1}}^2} &=& \ExpSub{t+1}{\norm{w^t + \mathcal{C}^p(x^{t+1} - w^t) - x^{t+1}}^2} \\
    & \overset{\eqref{eq:biased_compressor}}{\leq} &(1-\alpha) \ExpSub{t+1}{\norm{x^{t+1} - w^t}^2} \\
    & = &(1-\alpha) \ExpSub{t+1}{\norm{x^{t} - \gamma g^t - w^t}^2} \\
    & \overset{\eqref{eq:vardecomp}}{=} &(1-\alpha) \gamma^2 \ExpSub{t+1}{\norm{g^t - \nabla f(w^{t})}^2} + (1-\alpha) \norm{x^t - \gamma \nabla f(w^{t}) - w^t}^2 \\
    & \overset{\eqref{eq:young}}{\leq}& \gamma^2 \ExpSub{t+1}{\norm{g^t - \nabla f(w^{t})}^2} + \left(1-\frac{\alpha}{2}\right) \norm{w^t - x^t}^2 + \frac{2 \gamma^2}{\alpha} \norm{\nabla f(w^{t})}^2 \\
    & \overset{\eqref{eq:young_2}}{\leq}& \gamma^2 \ExpSub{t+1}{\norm{g^t - \nabla f(w^{t})}^2}
    + \left(1-\frac{\alpha}{2}\right) \norm{w^t - x^t}^2 \\
    &&\qquad + \frac{4 \gamma^2}{\alpha} \norm{\nabla f(w^{t}) - \nabla f(x^{t})}^2 + \frac{4 \gamma^2}{\alpha} \norm{\nabla f(x^{t}) - \nabla f(x^{*})}^2.
 \end{eqnarray*}
 Using Assumption~\ref{ass:lipschitz_constant} and Lemma~\ref{lemma:lipt_func}, we obtain
 \begin{align*}
    &\ExpSub{t+1}{\norm{w^{t+1}-x^{t+1}}^2} \leq \gamma^2 \ExpSub{t+1}{\norm{g^t - \nabla f(w^{t})}^2} \\
    &\qquad + \left(1-\frac{\alpha}{2} + \frac{4 \gamma^2 L^2}{\alpha}\right) \norm{w^t - x^t}^2 + \frac{8 \gamma^2 L}{\alpha} \left(f(x^{t}) - f(x^{*})\right) \\
    &\overset{\eqref{eq:g_t}}{\leq} \gamma^2 \left(\frac{2 \omega}{n^2}\sum_{i=1}^n \norm{h^t_i - \nabla f_i(x^*)}^2 + \frac{4\omega \widehat{L}^2}{n} \norm{w^t - x^t}^2 + \frac{8\omega L_{\max}}{n} \left(f(x^t) - f(x^*)\right)\right) \\
    &\qquad + \left(1-\frac{\alpha}{2} + \frac{4 \gamma^2 L^2}{\alpha}\right) \norm{w^t - x^t}^2 + \frac{8 \gamma^2 L}{\alpha} \left(f(x^{t}) - f(x^{*})\right) \\
    &= \left(1-\frac{\alpha}{2} + \frac{4 \gamma^2 L^2}{\alpha} + \frac{4\gamma^2\omega \widehat{L}^2}{n}\right) \norm{w^t - x^t}^2
    + \frac{2 \gamma^2 \omega}{n^2}\sum_{i=1}^n \norm{h^t_i - \nabla f_i(x^*)}^2 \\
    &\qquad + \left(\frac{8\gamma^2\omega L_{\max}}{n} + \frac{8 \gamma^2 L}{\alpha}\right) \left(f(x^{t}) - f(x^{*})\right) \\
    &\leq \left(1-\frac{\alpha}{4}\right) \norm{w^t - x^t}^2 + \frac{2 \gamma^2 \omega}{n^2}\sum_{i=1}^n \norm{h^t_i - \nabla f_i(x^*)}^2 \\
    &\qquad+ \left(\frac{8\gamma^2\omega L_{\max}}{n} + \frac{8 \gamma^2 L}{\alpha}\right) \left(f(x^{t}) - f(x^{*})\right),
 \end{align*}
 where we assume that $\gamma \leq \frac{\alpha}{\sqrt{32} L}$ and $\gamma \leq \frac{\sqrt{\alpha n}}{\sqrt{32 \omega}\widehat{L}}.$
 
 Let us fix some constants $\kappa \geq 0$ and $\nu \geq 0.$ We now combine the above inequality with \eqref{eq:f_x_t}, \eqref{eq:g_t} and \eqref{eq:h_t} to obtain
 \begin{align*}
    &\ExpSub{t+1}{f(x^{t+1})} + \kappa \frac{1}{n}\sum_{i=1}^n \ExpSub{t+1}{\norm{h^{t+1}_i - \nabla f_i(x^*)}^2} + \nu \ExpSub{t+1}{\norm{w^{t+1}-x^{t+1}}^2} \\
    &\leq f(x^*) + \gamma \left(\frac{2 \omega}{n^2}\sum_{i=1}^n \norm{h^t_i - \nabla f_i(x^*)}^2 + \frac{4\omega \widehat{L}^2}{n} \norm{w^t - x^t}^2 + \frac{8\omega L_{\max}}{n} \left(f(x^t) - f(x^*)\right)\right)\\
       &\quad + \frac{1}{2\gamma}\left(1 - \frac{\gamma \mu}{2}\right)\norm{x^t - x^*}^2 - \frac{1}{2\gamma}\ExpSub{t+1}{\norm{x^{t+1} - x^*}^2} + 2 L \norm{w^t - x^t}^2 \\
       &\quad + \kappa \left(\left(1 - \beta\right)\frac{1}{n}\sum_{i=1}^n \norm{h^{t}_i - \nabla f_i(x^*)}^2 + 2\beta \widehat{L}^2 \norm{w^t - x^t}^2 + 4\beta L_{\max} \left(f(x^t) - f(x^*)\right)\right) \\
       &\quad + \nu \left(\left(1-\frac{\alpha}{4}\right) \norm{w^t - x^t}^2 + \frac{2 \gamma^2 \omega}{n^2}\sum_{i=1}^n \norm{h^t_i - \nabla f_i(x^*)}^2 + \left(\frac{8\gamma^2\omega L_{\max}}{n} + \frac{8 \gamma^2 L}{\alpha}\right) \left(f(x^{t}) - f(x^{*})\right)\right).
 \end{align*}
 Rearranging the last inequality, one can get
 \begin{align}
    &\frac{1}{2\gamma}\ExpSub{t+1}{\norm{x^{t+1} - x^*}^2}
    + \ExpSub{t+1}{f(x^{t+1}) - f(x^*)}\nonumber \\
    &\qquad+ \kappa \frac{1}{n}\sum_{i=1}^n \ExpSub{t+1}{\norm{h^{t+1}_i - \nabla f_i(x^*)}^2}
    + \nu \ExpSub{t+1}{\norm{w^{t+1}-x^{t+1}}^2}\nonumber \\
    & \leq \frac{1}{2\gamma}\left(1 - \frac{\gamma \mu}{2}\right)\norm{x^t - x^*}^2 \nonumber \\
    &\qquad + \left(\frac{8\gamma\omega L_{\max}}{n} + \kappa 4\beta L_{\max} + \nu \left(\frac{8\gamma^2\omega L_{\max}}{n} + \frac{8 \gamma^2 L}{\alpha}\right)\right)\left(f(x^{t}) - f(x^{*})\right) \nonumber\\
    &\qquad + \left(\frac{2 \gamma \omega}{n} + \nu \frac{2 \gamma^2 \omega}{n} + \kappa \left(1 - \beta\right)\right)\frac{1}{n}\sum_{i=1}^n \norm{h^t_i - \nabla f_i(x^*)}^2 \nonumber\\
    &\qquad + \left(\frac{4\gamma\omega \widehat{L}^2}{n} + 2L + \kappa 2\beta \widehat{L}^2 + \nu \left(1-\frac{\alpha}{4}\right)\right)\norm{w^t - x^t}^2 \label{eq:kappa_nu}.
 \end{align}
 Our final goal is to find $\kappa$ and $\nu$ such that
 $$\frac{2 \gamma \omega}{n} + \nu \frac{2 \gamma^2 \omega}{n} + \kappa \left(1 - \beta\right) = \kappa \left(1 - \frac{\beta}{2}\right)$$
 and 
 $$\frac{4\gamma\omega \widehat{L}^2}{n} + 2L + \kappa 2\beta \widehat{L}^2 + \nu \left(1-\frac{\alpha}{4}\right) \leq \nu \left(1-\frac{\alpha}{8}\right).$$
 The last inequality is equivalent to
 \begin{align}\frac{32\gamma\omega \widehat{L}^2}{n \alpha} + \frac{16L}{\alpha} + \kappa \frac{16\beta \widehat{L}^2}{\alpha} \leq \nu.
 \label{eq:nu}
 \end{align}
 From the first equality we get
 $\kappa = \frac{4 \gamma \omega}{n \beta} + \nu \frac{4 \gamma^2 \omega}{n \beta}.$ Thus
 \begin{align*}
 &\frac{32\gamma\omega \widehat{L}^2}{n \alpha} + \frac{16L}{\alpha} + \kappa \frac{16\beta \widehat{L}^2}{\alpha} = \frac{32\gamma\omega \widehat{L}^2}{n \alpha} + \frac{16L}{\alpha} + \left(\frac{4 \gamma \omega}{n \beta} + \nu \frac{4 \gamma^2 \omega}{n \beta}\right) \frac{16\beta \widehat{L}^2}{\alpha} \\
 &\qquad=\frac{96\gamma\omega \widehat{L}^2}{n \alpha} + \frac{16L}{\alpha} + \nu \frac{64 \gamma^2 \omega \widehat{L}^2}{n \alpha} \leq \frac{96\gamma\omega \widehat{L}^2}{n \alpha} + \frac{16L}{\alpha} + \nu \frac{1}{2},
 \end{align*}
 where we used that $\gamma \leq \frac{\sqrt{n \alpha}}{\sqrt{128 \omega} \widehat{L}}.$ It means that we can take $\nu = \frac{192\gamma\omega \widehat{L}^2}{n \alpha} + \frac{32L}{\alpha}$ to ensure that \eqref{eq:nu} holds. Thus
 $$\kappa = \frac{4 \gamma \omega}{n \beta} + \left(\frac{192\gamma\omega \widehat{L}^2}{n \alpha} + \frac{32L}{\alpha}\right) \frac{4 \gamma^2 \omega}{n \beta} = \frac{4 \gamma \omega}{n \beta} + \frac{768\gamma^3\omega^2 \widehat{L}^2}{n^2 \alpha \beta} + \frac{128 \gamma^2 \omega L}{n \beta \alpha}.$$
 Let us now substitute these values of $\kappa$ and $\nu$ in inequality \eqref{eq:kappa_nu}:
 \begin{align*}
    &\frac{1}{2\gamma}\ExpSub{t+1}{\norm{x^{t+1} - x^*}^2} + \ExpSub{t+1}{f(x^{t+1}) - f(x^*)} \\
    &\qquad+ \kappa \frac{1}{n}\sum_{i=1}^n \ExpSub{t+1}{\norm{h^{t+1}_i - \nabla f_i(x^*)}^2} + \nu \ExpSub{t+1}{\norm{w^{t+1}-x^{t+1}}^2} \nonumber \\
    &\leq \frac{1}{2\gamma}\left(1 - \frac{\gamma \mu}{2}\right)\norm{x^t - x^*}^2 + \kappa \left(1 - \frac{\beta}{2}\right)\frac{1}{n}\sum_{i=1}^n \norm{h^t_i - \nabla f_i(x^*)}^2 + \nu \left(1-\frac{\alpha}{8}\right)\norm{w^t - x^t}^2\\
    &\quad + \left(\frac{8\gamma\omega L_{\max}}{n} + \left(\frac{4 \gamma \omega}{n \beta} + \frac{768\gamma^3\omega^2 \widehat{L}^2}{n^2 \alpha \beta} + \frac{128 \gamma^2 \omega L}{n \beta \alpha}\right) 4\beta L_{\max} \right.\\
    &\quad\quad\left.+\left(\frac{192\gamma\omega \widehat{L}^2}{n \alpha} + \frac{32L}{\alpha}\right) \left(\frac{8\gamma^2\omega L_{\max}}{n} + \frac{8 \gamma^2 L}{\alpha}\right)\right)\left(f(x^{t}) - f(x^{*})\right) \\
    &= \frac{1}{2\gamma}\left(1 - \frac{\gamma \mu}{2}\right)\norm{x^t - x^*}^2 + \kappa \left(1 - \frac{\beta}{2}\right)\frac{1}{n}\sum_{i=1}^n \norm{h^t_i - \nabla f_i(x^*)}^2 + \nu \left(1-\frac{\alpha}{8}\right)\norm{w^t - x^t}^2\\
    &\quad + \left(\frac{24\gamma\omega L_{\max}}{n} + \frac{4608\gamma^3\omega^2 \widehat{L}^2 L_{\max}}{n^2 \alpha} + \frac{768 \gamma^2 \omega L L_{\max}}{n \alpha} + \frac{1536\gamma^3\omega L \widehat{L}^2}{n \alpha^2} + \frac{256 \gamma^2 L^2}{\alpha^2}\right)\left(f(x^{t}) - f(x^{*})\right).
 \end{align*}
 Using the assumptions on $\gamma,$ we have
 \begin{align*}
    &\frac{24\gamma\omega L_{\max}}{n} \leq \frac{1}{10}, \\
    &\frac{4608\gamma^3\omega^2 \widehat{L}^2 L_{\max}}{n^2 \alpha} \leq \frac{20 \gamma^2\omega \widehat{L}^2}{n \alpha} \leq \frac{1}{10}, \\
    &\frac{768 \gamma^2 \omega L L_{\max}}{n \alpha} \leq \frac{4 \gamma L}{\alpha} \leq \frac{1}{10}, \\
    &\frac{1536\gamma^3\omega L \widehat{L}^2}{n \alpha^2} \leq \frac{40\gamma^2\omega \widehat{L}^2}{n \alpha} \leq \frac{1}{10}, \\
    &\frac{256 \gamma^2 L^2}{\alpha^2} \leq \frac{1}{10}.
 \end{align*}
 Finally, considering $\gamma \leq \frac{\beta}{\mu}$ and $\gamma \leq \frac{\alpha}{4\mu}$ gives
 \begin{align*}
    &\frac{1}{2\gamma}\ExpSub{t+1}{\norm{x^{t+1} - x^*}^2} + \ExpSub{t+1}{f(x^{t+1}) - f(x^*)} \\
    &\qquad+ \kappa \frac{1}{n}\sum_{i=1}^n \ExpSub{t+1}{\norm{h^{t+1}_i - \nabla f_i(x^*)}^2} + \nu \ExpSub{t+1}{\norm{w^{t+1}-x^{t+1}}^2} \nonumber \\
    &\leq \frac{1}{2\gamma}\left(1 - \frac{\gamma \mu}{2}\right)\norm{x^t - x^*}^2 + \kappa \left(1 - \frac{\gamma \mu}{2}\right)\frac{1}{n}\sum_{i=1}^n \norm{h^t_i - \nabla f_i(x^*)}^2\\
    &\qquad + \nu \left(1 - \frac{\gamma \mu}{2}\right)\norm{w^t - x^t}^2
    + \frac{1}{2}\left(f(x^{t}) - f(x^{*})\right).
 \end{align*}
 Note that $\kappa = \frac{4 \gamma \omega}{n \beta} + \frac{768\gamma^3\omega^2 \widehat{L}^2}{n^2 \alpha \beta} + \frac{128 \gamma^2 \omega L}{n \beta \alpha} \leq \frac{8 \gamma \omega}{n \beta}.$

 In the proof, we have the requirement that
 \begin{align}
    &\gamma \leq \min\left\{\frac{n}{160 \omega L_{\max}}, \frac{\sqrt{n \alpha}}{20\sqrt{\omega} \widehat{L}}, \frac{\alpha}{100 L}, \frac{\beta}{\mu}\right\}.
    \label{eq:aux_gamma_req}
 \end{align}
 Let us simplify it. Using Lemma~\ref{lemma:lipt_constants}, we have
 \begin{align*}
    \frac{20\sqrt{\omega} \widehat{L}}{\sqrt{n \alpha}} \leq \frac{20\sqrt{\omega} \sqrt{L_{\max} L}}{\sqrt{n \alpha}} \leq \frac{50 \omega L_{\max}}{n} + \frac{2 L}{\alpha} \leq \max\left\{\frac{100 \omega L_{\max}}{n},\frac{4 L}{\alpha}\right\}
 \end{align*}
 Using the last inequality, we can simplify \eqref{eq:aux_gamma_req} to
 \begin{align*}
    &\gamma \leq \min\left\{\frac{n}{160 \omega L_{\max}}, \frac{\alpha}{100 L}, \frac{\beta}{\mu}\right\}.
 \end{align*}
 \end{proof}

We now prove a theorem for the general convex case:
\begin{restatable}{theorem}{THEOREMDIANAGENERALCONVEX}
    \label{theorem:diana_general_convex}
    Let us assume that Assumptions~\ref{ass:lipschitz_constant}, \ref{ass:workers_lipschitz_constant} and \ref{ass:convex} hold, the strong convexity parameter satisfies $\mu = 0,$ $\beta = \frac{1}{\omega + 1}$, $x^0 = w^0$ and
    \begin{align*}
       &\gamma \leq \min\left\{\frac{n}{160 \omega L_{\max}}, \frac{\alpha}{100 L}\right\}.
    \end{align*}
    Then Algorithm~\ref{algorithm:diana_ef21_p} guarantees a convergence rate
    \begin{align}
       f\left(\frac{1}{T} \sum_{t=1}^{T} x^{t}\right) - f(x^*) \leq \frac{1}{\gamma T}\norm{x^0 - x^*}^2 + \frac{f(x^{0}) - \nabla f(x^{*})}{T} + \frac{16 \gamma \omega (\omega + 1)}{T n^2} \sum_{i=1}^n \norm{h^0_i - \nabla f_i(x^*)}^2.
    \end{align}
\end{restatable}
 
 \begin{proof}
    Under our assumptions Theorem~\ref{theorem:general_diana} holds. Let us bound \eqref{eq:main_diana}:
    \begin{align*}
       &\frac{1}{2\gamma}\Exp{\norm{x^{t+1} - x^*}^2} + \Exp{f(x^{t+1}) - f(x^*)} \\
       &\qquad + \kappa \frac{1}{n}\sum_{i=1}^n \Exp{\norm{h^{t+1}_i - \nabla f_i(x^*)}^2} + \nu \Exp{\norm{w^{t+1}-x^{t+1}}^2} \nonumber\\
       &\leq \frac{1}{2\gamma}\left(1 - \frac{\gamma \mu}{2}\right)\Exp{\norm{x^t - x^*}^2} + \frac{1}{2}\Exp{f(x^{t}) - f(x^{*})} \nonumber\\
       &\qquad + \kappa \left(1 - \frac{\gamma \mu}{2}\right)\frac{1}{n}\sum_{i=1}^n \Exp{\norm{h^t_i - \nabla f_i(x^*)}^2} + \nu \left(1 - \frac{\gamma \mu}{2}\right)\Exp{\norm{w^t - x^t}^2} \\
       &\leq \frac{1}{2\gamma}\Exp{\norm{x^t - x^*}^2} + \frac{1}{2}\Exp{f(x^{t}) - f(x^{*})} + \kappa \frac{1}{n}\sum_{i=1}^n \Exp{\norm{h^t_i - \nabla f_i(x^*)}^2} + \nu \Exp{\norm{w^t - x^t}^2}.
    \end{align*}
    We now sum the inequality for $t \in \{0, \dots, T-1\}$ and obtain
    \begin{align*}
       &\frac{1}{2\gamma}\Exp{\norm{x^{T} - x^*}^2} + \frac{1}{2}\Exp{f(x^{T}) - f(x^*)} + \frac{1}{2} \sum_{t=1}^{T}\Exp{f(x^{t}) - f(x^*)} \\
       &\qquad + \kappa \frac{1}{n}\sum_{i=1}^n \Exp{\norm{h^{T}_i - \nabla f_i(x^*)}^2} + \nu \Exp{\norm{w^{T}-x^{T}}^2} \nonumber\\
       &\leq \frac{1}{2\gamma}\norm{x^0 - x^*}^2 + \frac{1}{2}\left(f(x^{0}) - f(x^{*})\right) + \kappa \frac{1}{n}\sum_{i=1}^n \norm{h^0_i - \nabla f_i(x^*)}^2 + \nu \norm{w^0 - x^0}^2 \\
       &\leq \frac{1}{2\gamma}\norm{x^0 - x^*}^2 + \frac{1}{2}\left(f(x^{0}) - f(x^{*})\right) + \frac{8 \gamma \omega}{n^2 \beta} \sum_{i=1}^n \norm{h^0_i - \nabla f_i(x^*)}^2,
    \end{align*}
    where we used the assumption $x^0 = w^0$ and the bound on $\kappa.$
    Using nonnegativity of the terms and convexity, we then have 
    \begin{align*}
       f\left(\frac{1}{T} \sum_{t=1}^{T} x^{t}\right) - f(x^*) \leq \frac{1}{\gamma T}\norm{x^0 - x^*}^2 + \frac{f(x^{0}) - \nabla f(x^{*})}{T} + \frac{16 \gamma \omega}{T n^2 \beta} \sum_{i=1}^n \norm{h^0_i - \nabla f_i(x^*)}^2.
    \end{align*}
 \end{proof}

 We now prove a theorem for the strongly convex case:
 \THEOREMDIANASTRONGLY*
 
 \begin{proof}
    Under our assumptions Theorem~\ref{theorem:general_diana} holds. Using $\gamma \leq \frac{\alpha}{100 L} \leq \frac{1}{\mu},$ let us bound \eqref{eq:main_diana}:
    \begin{align*}
       &\frac{1}{2\gamma}\Exp{\norm{x^{t+1} - x^*}^2} + \Exp{f(x^{t+1}) - f(x^*)} \\
       &\qquad + \kappa \frac{1}{n}\sum_{i=1}^n \Exp{\norm{h^{t+1}_i - \nabla f_i(x^*)}^2} + \nu \Exp{\norm{w^{t+1}-x^{t+1}}^2} \nonumber\\
       &\leq \frac{1}{2\gamma}\left(1 - \frac{\gamma \mu}{2}\right)\Exp{\norm{x^t - x^*}^2} + \frac{1}{2}\Exp{f(x^{t}) - f(x^{*})} \nonumber\\
       &\qquad + \kappa \left(1 - \frac{\gamma \mu}{2}\right)\frac{1}{n}\sum_{i=1}^n \Exp{\norm{h^t_i - \nabla f_i(x^*)}^2} + \nu \left(1 - \frac{\gamma \mu}{2}\right)\Exp{\norm{w^t - x^t}^2} \\
       &\leq \frac{1}{2\gamma}\left(1 - \frac{\gamma \mu}{2}\right)\Exp{\norm{x^t - x^*}^2} + \left(1 - \frac{\gamma \mu}{2}\right)\Exp{f(x^{t}) - f(x^{*})} \nonumber\\
       &\qquad + \kappa \left(1 - \frac{\gamma \mu}{2}\right)\frac{1}{n}\sum_{i=1}^n \Exp{\norm{h^t_i - \nabla f_i(x^*)}^2} + \nu \left(1 - \frac{\gamma \mu}{2}\right)\Exp{\norm{w^t - x^t}^2} \\
       &= \left(1 - \frac{\gamma \mu}{2}\right)\left(\frac{1}{2\gamma}\Exp{\norm{x^t - x^*}^2} + \Exp{f(x^{t}) - f(x^{*})} + \kappa \frac{1}{n}\sum_{i=1}^n \Exp{\norm{h^t_i - \nabla f_i(x^*)}^2} + \nu \Exp{\norm{w^t - x^t}^2}\right).
    \end{align*}
    Recursively applying the last inequality and using $x^0 = w^0,$ one can get that
    \begin{align*}
       &\frac{1}{2\gamma}\Exp{\norm{x^{T} - x^*}^2} + \Exp{f(x^{T}) - f(x^*)} + \kappa \frac{1}{n}\sum_{i=1}^n \Exp{\norm{h^{T}_i - \nabla f_i(x^*)}^2} + \nu \Exp{\norm{w^{T}-x^{T}}^2} \\
       &\qquad\leq \left(1 - \frac{\gamma \mu}{2}\right)^T\left(\frac{1}{2\gamma}\Exp{\norm{x^0 - x^*}^2} + \left(f(x^{0}) - f(x^{*})\right) + \kappa \frac{1}{n}\sum_{i=1}^n \norm{h^0_i - \nabla f_i(x^*)}^2\right).
    \end{align*}
    Using the nonnegativity of the terms and the bound on $\kappa$, we obtain
    \begin{align*}
       &\frac{1}{2\gamma}\Exp{\norm{x^{T} - x^*}^2} + \Exp{f(x^{T}) - f(x^*)} \\
       &\qquad\leq \left(1 - \frac{\gamma \mu}{2}\right)^T\left(\frac{1}{2\gamma}\Exp{\norm{x^0 - x^*}^2} + \left(f(x^{0}) - f(x^{*})\right) + \frac{8 \gamma \omega}{n^2 \beta}\sum_{i=1}^n \norm{h^0_i - \nabla f_i(x^*)}^2\right).
    \end{align*}
 \end{proof}

\subsection{Communication Complexities in the General Convex Case}
\label{sec:comm_diana_general}
We now derive the communication complexities for the general convex case. From Theorem~\ref{theorem:diana_general_convex}, we know that \algname{EF21-P + DIANA} has the following convergence rate:
\begin{align*}
    f\left(\frac{1}{T} \sum_{t=1}^{T} x^{t}\right) - f(x^*) \leq \frac{1}{\gamma T}\norm{x^0 - x^*}^2 + \frac{f(x^{0}) - \nabla f(x^{*})}{T} + \frac{16 \gamma \omega (\omega + 1)}{T n^2} \sum_{i=1}^n \norm{h^0_i - \nabla f_i(x^*)}^2.
\end{align*}
Let us take $h^0_i = \nabla f_i(x^0)$ for all $i \in [n].$ Using Assumptions~\ref{ass:lipschitz_constant} and \ref{ass:workers_lipschitz_constant}, we have
\begin{align*}
    f\left(\frac{1}{T} \sum_{t=1}^{T} x^{t}\right) - f(x^*) &\leq \frac{1}{\gamma T}\norm{x^0 - x^*}^2 + \frac{L \norm{x^0 - x^*}^2}{2 T} \\
    &\qquad + \frac{16 \gamma \omega (\omega + 1)}{T n^2} \sum_{i=1}^n \norm{\nabla f_i(x^0) - \nabla f_i(x^*)}^2 \\
    &\leq \frac{1}{\gamma T}\norm{x^0 - x^*}^2 + \frac{L \norm{x^0 - x^*}^2}{2 T} + \frac{16 \gamma \omega (\omega + 1) \widehat{L}^2 \norm{x^0 - x^*}^2}{T n} \\
    &\leq \frac{1}{\gamma T}\norm{x^0 - x^*}^2 + \frac{L \norm{x^0 - x^*}^2}{2 T} + \frac{16 \gamma \omega (\omega + 1) L_{\max} L \norm{x^0 - x^*}^2}{T n}.
\end{align*}
In the last two inequalities, we use the definition of $\widehat{L}$ and Lemma~\ref{lemma:lipt_constants}.
Using the bound on $\gamma,$ we obtain that \algname{EF21-P + DIANA} returns an $\varepsilon$-solution after 
\begin{align*}
    &\cO\left(\frac{\omega L_{\max}}{n \varepsilon}+ \frac{L}{\alpha \varepsilon} + \frac{L}{\varepsilon} + \frac{\gamma \omega (\omega + 1) L_{\max} L}{n \varepsilon}\right)\\
    &=\cO\left(\frac{\omega L_{\max}}{n \varepsilon}+ \frac{L}{\alpha \varepsilon} + \frac{L}{\varepsilon} + \frac{(\omega + 1) L}{\varepsilon}\right)
\end{align*}
steps.
For simplicity, we assume that the server and the workers use Top$K$ and Rand$K$ compressors, respectively.
Thus the server-to-workers and the workers-to-server communication complexities equal
\begin{align*}
    &\cO\left(K \times \left(\frac{\omega L_{\max}}{n \varepsilon}+ \frac{L}{\alpha \varepsilon} + \frac{L}{\varepsilon} + \frac{(\omega + 1) L}{\varepsilon}\right)\right) \\
    &=\cO\left(\frac{d L_{\max}}{n \varepsilon} + \frac{d L}{\varepsilon} + \frac{K L}{\varepsilon} + \frac{d L}{\varepsilon}\right)\\
    &=\cO\left(\frac{d L_{\max}}{n \varepsilon} + \frac{d L}{\varepsilon}\right).
\end{align*}
Since $L_{\max} \leq n L,$ this complexity is no worse than the \algname{GD}'s complexity $\cO\left(\frac{d L}{\varepsilon}\right)$ for any $K \in [1, d].$

\subsection{Proofs for \algname{EF21-P + DIANA} with Stochastic Gradients}
First, we prove the following auxiliary theorem:
\begin{theorem}
    Let us consider Algorithm~\ref{algorithm:diana_ef21_p} using the stochastic gradients $\widetilde{\nabla} f_i$ instead of the exact gradients $\nabla f_i$ for all $i \in [n]$. Assume that Assumptions~\ref{ass:lipschitz_constant}, \ref{ass:workers_lipschitz_constant}, \ref{ass:convex} and \ref{ass:stochastic_unbiased_and_variance_bounded} hold, $\beta \in \left[0, \frac{1}{\omega + 1}\right],$ and 
    \begin{align*}
       &\gamma \leq \min\left\{\frac{n}{160 \omega L_{\max}}, \frac{\alpha}{100 L}, \frac{\beta}{\mu}\right\}.
    \end{align*}
    Then Algorithm~\ref{algorithm:diana_ef21_p} guarantees that
    \begin{align}
       &\frac{1}{2\gamma}\Exp{\norm{x^{t+1} - x^*}^2} + \Exp{f(x^{t+1}) - f(x^*)} \nonumber\\
       &\quad + \kappa \frac{1}{n}\sum_{i=1}^n \Exp{\norm{h^{t+1}_i - \nabla f_i(x^*)}^2} + \nu \Exp{\norm{w^{t+1}-x^{t+1}}^2} \nonumber\\
       &\leq \frac{1}{2\gamma}\left(1 - \frac{\gamma \mu}{2}\right)\Exp{\norm{x^t - x^*}^2} + \frac{1}{2}\Exp{f(x^{t}) - f(x^{*})} \nonumber\\
       &\quad + \kappa \left(1 - \frac{\gamma \mu}{2}\right)\frac{1}{n}\sum_{i=1}^n \Exp{\norm{h^t_i - \nabla f_i(x^*)}^2} + \nu \left(1 - \frac{\gamma \mu}{2}\right)\Exp{\norm{w^t - x^t}^2} + \frac{12 \gamma(\omega + 1) \sigma^2}{n}, \label{eq:main_diana_stochastic}
    \end{align}
    where $\kappa \leq \frac{8 \gamma \omega}{n \beta}$ and $\nu \leq \frac{192\gamma\omega \widehat{L}^2}{n \alpha} + \frac{32L}{\alpha}.$
\end{theorem}

\begin{proof}
    First, we bound $\ExpSub{t+1}{\norm{g^t - \nabla f(w^t)}^2},$ $\frac{1}{n}\sum_{i=1}^n \ExpSub{t+1}{\norm{h^{t+1}_i - \nabla f_i(x^*)}^2}$ and $\ExpSub{t+1}{\norm{w^{t+1}-x^{t+1}}^2}.$ 
   
    Using  independence of the compressors, we have
    \begin{eqnarray*}
       &&\ExpSub{t+1}{\norm{g^t - \nabla f(w^t)}^2}\\
       &=&\ExpSub{t+1}{\norm{h^t + \frac{1}{n}\sum_{i=1}^n \cC_i^{D}(\widetilde{\nabla} f_i(w^t) - h^t_i) - \nabla f(w^t)}^2} \\
       &=&\frac{1}{n^2}\sum_{i=1}^n\ExpSub{t+1}{\norm{\cC_i^{D}(\widetilde{\nabla} f_i(w^t) - h^t_i) - \left(\nabla f_i(w^t) - h^t_i\right)}^2} \\
       &\overset{\eqref{eq:vardecomp}}{=}&\frac{1}{n^2}\sum_{i=1}^n \left( \ExpSub{t+1}{\norm{\cC_i^{D}(\widetilde{\nabla} f_i(w^t) - h^t_i) - \left(\widetilde{\nabla} f_i(w^t) - h^t_i\right)}^2} + \ExpSub{t+1}{\norm{\widetilde{\nabla} f_i(w^t) - \nabla f_i(w^t)}^2} \right)\\
       &\leq&\frac{\omega}{n^2}\sum_{i=1}^n \ExpSub{t+1}{\norm{\widetilde{\nabla} f_i(w^t) - h^t_i}^2} + \frac{1}{n^2}\sum_{i=1}^n \ExpSub{t+1}{\norm{\widetilde{\nabla} f_i(w^t) - \nabla f_i(w^t)}^2} \\
       &\overset{\eqref{eq:vardecomp}}{=}&\frac{\omega}{n^2}\sum_{i=1}^n \norm{\nabla f_i(w^t) - h^t_i}^2 + \frac{\omega + 1}{n^2}\sum_{i=1}^n \ExpSub{t+1}{\norm{\widetilde{\nabla} f_i(w^t) - \nabla f_i(w^t)}^2} \\
       &\leq&\frac{\omega}{n^2}\sum_{i=1}^n \norm{\nabla f_i(w^t) - h^t_i}^2 + \frac{(\omega + 1) \sigma^2}{n} \\
       &\leq&\frac{2 \omega}{n^2}\sum_{i=1}^n \norm{h^t_i - \nabla f_i(x^*)}^2 + \frac{2\omega}{n^2}\sum_{i=1}^n \norm{\nabla f_i(w^t) - \nabla f_i(x^*)}^2 + \frac{(\omega + 1) \sigma^2}{n}\\
       &\leq&\frac{2 \omega}{n^2}\sum_{i=1}^n \norm{h^t_i - \nabla f_i(x^*)}^2 + \frac{4\omega}{n^2}\sum_{i=1}^n \norm{\nabla f_i(w^t) - \nabla f_i(x^t)}^2 \\
       &&\qquad+ \frac{4\omega}{n^2}\sum_{i=1}^n \norm{\nabla f_i(x^t) - \nabla f_i(x^*)}^2 + \frac{(\omega + 1) \sigma^2}{n},
   \end{eqnarray*}
   where in the last three inequalities, we used \eqref{eq:compressor} and \eqref{eq:young_2}. Using Assumption~\ref{ass:workers_lipschitz_constant} and Lemma~\ref{lemma:lipt_func}, we obtain
   \begin{align*}
    &\ExpSub{t+1}{\norm{g^t - \nabla f(w^t)}^2} \nonumber\\
    &\leq\frac{2 \omega}{n^2}\sum_{i=1}^n \norm{h^t_i - \nabla f_i(x^*)}^2 + \frac{4\omega \widehat{L}^2}{n} \norm{w^t - x^t}^2 + \frac{8\omega L_{\max}}{n} \left(f(x^t) - f(x^*)\right) + \frac{(\omega + 1) \sigma^2}{n}.
   \end{align*}
   Next, we bound $\frac{1}{n}\sum_{i=1}^n \norm{h^{t+1}_i - \nabla f_i(x^*)}^2$ to construct a Lyapunov function:
 \begin{eqnarray*}
    &&\frac{1}{n}\sum_{i=1}^n \ExpSub{t+1}{\norm{h^{t+1}_i - \nabla f_i(x^*)}^2}\\
    &=&\frac{1}{n}\sum_{i=1}^n \ExpSub{t+1}{\norm{h^{t}_i + \beta \cC_i^{D}(\widetilde{\nabla} f_i(w^t) - h^t_i) - \nabla f_i(x^*)}^2}\\
    &=&\frac{1}{n}\sum_{i=1}^n \norm{h^{t}_i - \nabla f_i(x^*)}^2 + \frac{2\beta}{n}\sum_{i=1}^n \inp{h^{t}_i - \nabla f_i(x^*)}{\ExpSub{t+1}{\cC_i^{D}(\widetilde{\nabla} f_i(w^t) - h^t_i)}} \\
    &&\qquad+ \frac{\beta^2}{n}\sum_{i=1}^n \ExpSub{t+1}{\norm{\cC_i^{D}(\widetilde{\nabla} f_i(w^t) - h^t_i)}^2}\\
    &\overset{\eqref{eq:compressor}}{\leq}&\frac{1}{n}\sum_{i=1}^n \norm{h^{t}_i - \nabla f_i(x^*)}^2 + \frac{2\beta}{n}\sum_{i=1}^n \inp{h^{t}_i - \nabla f_i(x^*)}{\nabla f_i(w^t) - h^t_i} \\
    &&\qquad+ \frac{\beta^2(\omega + 1)}{n}\sum_{i=1}^n \ExpSub{t+1}{\norm{\widetilde{\nabla} f_i(w^t) - h^t_i}^2}\\
    &\overset{\eqref{eq:vardecomp}}{=}&\frac{1}{n}\sum_{i=1}^n \norm{h^{t}_i - \nabla f_i(x^*)}^2 + \frac{2\beta}{n}\sum_{i=1}^n \inp{h^{t}_i - \nabla f_i(x^*)}{\nabla f_i(w^t) - h^t_i} \\
    &&\qquad + \frac{\beta^2(\omega + 1)}{n}\sum_{i=1}^n \ExpSub{t+1}{\norm{\nabla f_i(w^t) - h^t_i}^2} + \frac{\beta^2(\omega + 1)}{n}\sum_{i=1}^n \ExpSub{t+1}{\norm{\widetilde{\nabla} f_i(w^t) - \nabla f_i(w^t)}^2}\\
    &\leq&\frac{1}{n}\sum_{i=1}^n \norm{h^{t}_i - \nabla f_i(x^*)}^2 + \frac{2\beta}{n}\sum_{i=1}^n \inp{h^{t}_i - \nabla f_i(x^*)}{\nabla f_i(w^t) - h^t_i} \\
    &&\qquad + \frac{\beta^2(\omega + 1)}{n}\sum_{i=1}^n \ExpSub{t+1}{\norm{\nabla f_i(w^t) - h^t_i}^2} + \beta^2(\omega + 1) \sigma^2 \\
    &\overset{\eqref{eq:inp}}{=}&\left(1 - \beta\right)\frac{1}{n}\sum_{i=1}^n \norm{h^{t}_i - \nabla f_i(x^*)}^2 + \frac{\beta}{n}\sum_{i=1}^n \norm{\nabla f_i(w^t) - \nabla f_i(x^*)}^2 \\
    &&\qquad + \frac{\beta \left(\beta(\omega + 1) - 1\right)}{n}\sum_{i=1}^n \norm{\nabla f_i(w^t) - h^t_i}^2 + \beta^2(\omega + 1) \sigma^2\\
    &\leq& \left(1 - \beta\right)\frac{1}{n}\sum_{i=1}^n \norm{h^{t}_i - \nabla f_i(x^*)}^2 + \frac{\beta}{n}\sum_{i=1}^n \norm{\nabla f_i(w^t) - \nabla f_i(x^*)}^2 + \beta^2(\omega + 1) \sigma^2,
 \end{eqnarray*}
 where we use the assumption $\beta \in \left[0, \frac{1}{\omega + 1}\right].$
 Using \eqref{eq:young_2}, Assumption~\ref{ass:workers_lipschitz_constant} and Lemma~\ref{lemma:lipt_func}, we have
 \begin{align*}
    \frac{1}{n}\sum_{i=1}^n &\ExpSub{t+1}{\norm{h^{t+1}_i - \nabla f_i(x^*)}^2}
    \leq \left(1 - \beta\right)\frac{1}{n}\sum_{i=1}^n \norm{h^{t}_i - \nabla f_i(x^*)}^2 \\
    &\qquad+ 2\beta \widehat{L}^2 \norm{w^t - x^t}^2 + 4\beta L_{\max} \left(f(x^t) - f(x^*)\right) + \beta^2(\omega + 1) \sigma^2.
 \end{align*}
 It remains to bound $\ExpSub{t+1}{\norm{w^{t+1} - x^{t+1}}^2}:$
 \begin{eqnarray*}
    \ExpSub{t+1}{\norm{w^{t+1}-x^{t+1}}^2} &=& \ExpSub{t+1}{\norm{w^t + \mathcal{C}^p(x^{t+1} - w^t) - x^{t+1}}^2} \\
    & \overset{\eqref{eq:biased_compressor}}{\leq} &(1-\alpha) \ExpSub{t+1}{\norm{x^{t+1} - w^t}^2} \\
    & =& (1-\alpha) \ExpSub{t+1}{\norm{x^{t} - \gamma g^t - w^t}^2} \\
    & \overset{\eqref{eq:vardecomp}}{=} &(1-\alpha) \gamma^2 \ExpSub{t+1}{\norm{g^t - \nabla f(w^{t})}^2} + (1-\alpha) \norm{x^t - \gamma \nabla f(w^{t}) - w^t}^2 \\
    & \overset{\eqref{eq:young}}{\leq}& \gamma^2 \ExpSub{t+1}{\norm{g^t - \nabla f(w^{t})}^2} + \left(1-\frac{\alpha}{2}\right) \norm{w^t - x^t}^2 + \frac{2 \gamma^2}{\alpha} \norm{\nabla f(w^{t})}^2 \\
    & \overset{\eqref{eq:young_2}}{\leq}& \gamma^2 \ExpSub{t+1}{\norm{g^t - \nabla f(w^{t})}^2}
    + \left(1-\frac{\alpha}{2}\right) \norm{w^t - x^t}^2 \\
    &&\qquad + \frac{4 \gamma^2}{\alpha} \norm{\nabla f(w^{t}) - \nabla f(x^{t})}^2 + \frac{4 \gamma^2}{\alpha} \norm{\nabla f(x^{t}) - \nabla f(x^{*})}^2.
 \end{eqnarray*}
 Using Assumption~\ref{ass:lipschitz_constant} and Lemma~\ref{lemma:lipt_func}, we obtain
 \begin{align*}
    &\ExpSub{t+1}{\norm{w^{t+1}-x^{t+1}}^2} \leq \gamma^2 \ExpSub{t+1}{\norm{g^t - \nabla f(w^{t})}^2} \\
    &\qquad + \left(1-\frac{\alpha}{2} + \frac{4 \gamma^2 L^2}{\alpha}\right) \norm{w^t - x^t}^2 + \frac{8 \gamma^2 L}{\alpha} \left(f(x^{t}) - f(x^{*})\right) \\
    &\leq \gamma^2 \left(\frac{2 \omega}{n^2}\sum_{i=1}^n \norm{h^t_i - \nabla f_i(x^*)}^2 + \frac{4\omega \widehat{L}^2}{n} \norm{w^t - x^t}^2 + \frac{8\omega L_{\max}}{n} \left(f(x^t) - f(x^*)\right) + \frac{(\omega + 1) \sigma^2}{n}\right) \\
    &\qquad + \left(1-\frac{\alpha}{2} + \frac{4 \gamma^2 L^2}{\alpha}\right) \norm{w^t - x^t}^2 + \frac{8 \gamma^2 L}{\alpha} \left(f(x^{t}) - f(x^{*})\right) \\
    &= \frac{2 \gamma^2 \omega}{n^2}\sum_{i=1}^n \norm{h^t_i - \nabla f_i(x^*)}^2
    + \left(1-\frac{\alpha}{2} + \frac{4 \gamma^2 L^2}{\alpha} + \frac{4\gamma^2\omega \widehat{L}^2}{n}\right) \norm{w^t - x^t}^2 \\
    &\qquad + \left(\frac{8\gamma^2\omega L_{\max}}{n} + \frac{8 \gamma^2 L}{\alpha}\right) \left(f(x^{t}) - f(x^{*})\right)
    + \frac{\gamma^2 (\omega + 1) \sigma^2}{n} \\
    &\leq \left(1-\frac{\alpha}{4}\right) \norm{w^t - x^t}^2 + \frac{2 \gamma^2 \omega}{n^2}\sum_{i=1}^n \norm{h^t_i - \nabla f_i(x^*)}^2 \\
    &\qquad+ \left(\frac{8\gamma^2\omega L_{\max}}{n} + \frac{8 \gamma^2 L}{\alpha}\right) \left(f(x^{t}) - f(x^{*})\right) + \frac{\gamma^2 (\omega + 1) \sigma^2}{n},
 \end{align*}
 where we assume that $\gamma \leq \frac{\alpha}{\sqrt{32} L}$ and $\gamma \leq \frac{\sqrt{\alpha n}}{\sqrt{32 \omega}\widehat{L}}.$ Let us fix some constants $\kappa \geq 0$ and $\nu \geq 0.$ 
 In the proof of \eqref{eq:f_x_t} in Theorem~\ref{theorem:general_diana}, we do not use the structure of $g^t$. Hence we can reuse \eqref{eq:f_x_t} here and combine it with the above inequalities to obtain
 {\footnotesize
 \begin{align*}
    &\ExpSub{t+1}{f(x^{t+1})} + \kappa \frac{1}{n}\sum_{i=1}^n \ExpSub{t+1}{\norm{h^{t+1}_i - \nabla f_i(x^*)}^2} + \nu \ExpSub{t+1}{\norm{w^{t+1}-x^{t+1}}^2} \\
    &\leq f(x^*) + \gamma \left(\frac{2 \omega}{n^2}\sum_{i=1}^n \norm{h^t_i - \nabla f_i(x^*)}^2 + \frac{4\omega \widehat{L}^2}{n} \norm{w^t - x^t}^2 + \frac{8\omega L_{\max}}{n} \left(f(x^t) - f(x^*)\right) + \frac{(\omega + 1) \sigma^2}{n}\right)\\
       &\quad + \frac{1}{2\gamma}\left(1 - \frac{\gamma \mu}{2}\right)\norm{x^t - x^*}^2 - \frac{1}{2\gamma}\ExpSub{t+1}{\norm{x^{t+1} - x^*}^2} + 2 L \norm{w^t - x^t}^2 \\
       &\quad + \kappa \left(\left(1 - \beta\right)\frac{1}{n}\sum_{i=1}^n \norm{h^{t}_i - \nabla f_i(x^*)}^2 + 2\beta \widehat{L}^2 \norm{w^t - x^t}^2 + 4\beta L_{\max} \left(f(x^t) - f(x^*)\right) + \beta^2(\omega + 1) \sigma^2\right) \\
       &\quad + \nu \left(\left(1-\frac{\alpha}{4}\right) \norm{w^t - x^t}^2 + \frac{2 \gamma^2 \omega}{n^2}\sum_{i=1}^n \norm{h^t_i - \nabla f_i(x^*)}^2 + \left(\frac{8\gamma^2\omega L_{\max}}{n} + \frac{8 \gamma^2 L}{\alpha}\right) \left(f(x^{t}) - f(x^{*})\right) + \frac{\gamma^2 (\omega + 1) \sigma^2}{n}\right).
 \end{align*}
 }
 Rearranging the last inequality, one can get
 \begin{align*}
    &\frac{1}{2\gamma}\ExpSub{t+1}{\norm{x^{t+1} - x^*}^2} + \ExpSub{t+1}{f(x^{t+1}) - f(x^*)} \\
    &\qquad+ \kappa \frac{1}{n}\sum_{i=1}^n \ExpSub{t+1}{\norm{h^{t+1}_i - \nabla f_i(x^*)}^2} + \nu \ExpSub{t+1}{\norm{w^{t+1}-x^{t+1}}^2} \nonumber \\
    &\leq \frac{1}{2\gamma}\left(1 - \frac{\gamma \mu}{2}\right)\norm{x^t - x^*}^2 \nonumber \\
    &\qquad+ \left(\frac{8\gamma\omega L_{\max}}{n} + \kappa 4\beta L_{\max} + \nu \left(\frac{8\gamma^2\omega L_{\max}}{n} + \frac{8 \gamma^2 L}{\alpha}\right)\right)\left(f(x^{t}) - f(x^{*})\right) \nonumber\\
    &\qquad + \left(\frac{2 \gamma \omega}{n} + \nu \frac{2 \gamma^2 \omega}{n} + \kappa \left(1 - \beta\right)\right)\frac{1}{n}\sum_{i=1}^n \norm{h^t_i - \nabla f_i(x^*)}^2 \nonumber\\
    &\qquad + \left(\frac{4\gamma\omega \widehat{L}^2}{n} + 2L + \kappa 2\beta \widehat{L}^2 + \nu \left(1-\frac{\alpha}{4}\right)\right)\norm{w^t - x^t}^2 \\
    &\qquad + \frac{\gamma(\omega + 1) \sigma^2}{n} + \kappa \beta^2(\omega + 1) \sigma^2 + \nu \frac{\gamma^2 (\omega + 1) \sigma^2}{n}.
 \end{align*}
 Using the same reasoning as in the proof of Theorem~\ref{theorem:general_diana}, we have
 \begin{align*}
    &\frac{1}{2\gamma}\ExpSub{t+1}{\norm{x^{t+1} - x^*}^2} + \ExpSub{t+1}{f(x^{t+1}) - f(x^*)}\\
    &\qquad + \kappa \frac{1}{n}\sum_{i=1}^n \ExpSub{t+1}{\norm{h^{t+1}_i - \nabla f_i(x^*)}^2} + \nu \ExpSub{t+1}{\norm{w^{t+1}-x^{t+1}}^2} \nonumber \\
    &\leq \frac{1}{2\gamma}\left(1 - \frac{\gamma \mu}{2}\right)\norm{x^t - x^*}^2 + \kappa \left(1 - \frac{\gamma \mu}{2}\right)\frac{1}{n}\sum_{i=1}^n \norm{h^t_i - \nabla f_i(x^*)}^2 + \nu \left(1 - \frac{\gamma \mu}{2}\right)\norm{w^t - x^t}^2\\
    &\qquad + \frac{1}{2}\left(f(x^{t}) - f(x^{*})\right)
    + \frac{\gamma(\omega + 1) \sigma^2}{n} + \kappa \beta^2(\omega + 1) \sigma^2 + \nu \frac{\gamma^2 (\omega + 1) \sigma^2}{n}
 \end{align*}
 for some $\kappa \leq \frac{8 \gamma \omega}{n \beta}$ and $\nu \leq \frac{192\gamma\omega \widehat{L}^2}{n \alpha} + \frac{32L}{\alpha}.$ Thus
 \begin{align*}
    &\frac{1}{2\gamma}\ExpSub{t+1}{\norm{x^{t+1} - x^*}^2} + \ExpSub{t+1}{f(x^{t+1}) - f(x^*)}\\
    &\qquad + \kappa \frac{1}{n}\sum_{i=1}^n \ExpSub{t+1}{\norm{h^{t+1}_i - \nabla f_i(x^*)}^2} + \nu \ExpSub{t+1}{\norm{w^{t+1}-x^{t+1}}^2} \nonumber \\
    &\leq \frac{1}{2\gamma}\left(1 - \frac{\gamma \mu}{2}\right)\norm{x^t - x^*}^2 + \kappa \left(1 - \frac{\gamma \mu}{2}\right)\frac{1}{n}\sum_{i=1}^n \norm{h^t_i - \nabla f_i(x^*)}^2 + \nu \left(1 - \frac{\gamma \mu}{2}\right)\norm{w^t - x^t}^2\\
    &\qquad + \frac{1}{2}\left(f(x^{t}) - f(x^{*})\right)
    + \frac{\gamma(\omega + 1) \sigma^2}{n} \\
    &\qquad + \frac{8 \gamma \beta \omega (\omega + 1) \sigma^2}{n} + \frac{192 \gamma^3\omega (\omega + 1) \widehat{L}^2 \sigma^2}{n^2 \alpha} + \frac{32 \gamma^2 (\omega + 1) L \sigma^2}{n \alpha} \\
    &\leq \frac{1}{2\gamma}\left(1 - \frac{\gamma \mu}{2}\right)\norm{x^t - x^*}^2 + \kappa \left(1 - \frac{\gamma \mu}{2}\right)\frac{1}{n}\sum_{i=1}^n \norm{h^t_i - \nabla f_i(x^*)}^2 + \nu \left(1 - \frac{\gamma \mu}{2}\right)\norm{w^t - x^t}^2\\
    &\qquad + \frac{1}{2}\left(f(x^{t}) - f(x^{*})\right)
    + \frac{12 \gamma(\omega + 1) \sigma^2}{n},
 \end{align*}
 where used the bounds on $\gamma$ and $\beta.$

 In the proof, we have the requirement that
 \begin{align*}
    &\gamma \leq \min\left\{\frac{n}{160 \omega L_{\max}}, \frac{\sqrt{n \alpha}}{20\sqrt{\omega} \widehat{L}}, \frac{\alpha}{100 L}, \frac{\beta}{\mu}\right\}.
 \end{align*}
 As in the proof of Theorem~\ref{theorem:general_diana}, using Lemma~\ref{lemma:lipt_constants}, we can simplify it to
 \begin{align*}
    &\gamma \leq \min\left\{\frac{n}{160 \omega L_{\max}}, \frac{\alpha}{100 L}, \frac{\beta}{\mu}\right\}.
 \end{align*}
 
\end{proof}

 \THEOREMDIANASTRONGLYSTOCHASTIC*
 \begin{proof}
    Using $\gamma \leq \frac{\alpha}{100 L} \leq \frac{1}{\mu},$ we can bound \eqref{eq:main_diana_stochastic} as follows:
    \begin{align*}
       &\frac{1}{2\gamma}\Exp{\norm{x^{t+1} - x^*}^2} + \Exp{f(x^{t+1}) - f(x^*)}\\
       &\qquad + \kappa \frac{1}{n}\sum_{i=1}^n \Exp{\norm{h^{t+1}_i - \nabla f_i(x^*)}^2} + \nu \Exp{\norm{w^{t+1}-x^{t+1}}^2} \nonumber\\
       &\leq \frac{1}{2\gamma}\left(1 - \frac{\gamma \mu}{2}\right)\Exp{\norm{x^t - x^*}^2} + \frac{1}{2}\Exp{f(x^{t}) - f(x^{*})} \nonumber\\
       &\qquad + \kappa \left(1 - \frac{\gamma \mu}{2}\right)\frac{1}{n}\sum_{i=1}^n \Exp{\norm{h^t_i - \nabla f_i(x^*)}^2} + \nu \left(1 - \frac{\gamma \mu}{2}\right)\Exp{\norm{w^t - x^t}^2} + \frac{12 \gamma(\omega + 1) \sigma^2}{n}\\
       &\leq \frac{1}{2\gamma}\left(1 - \frac{\gamma \mu}{2}\right)\Exp{\norm{x^t - x^*}^2} + \left(1 - \frac{\gamma \mu}{2}\right)\Exp{f(x^{t}) - f(x^{*})} \nonumber\\
       &\qquad + \kappa \left(1 - \frac{\gamma \mu}{2}\right)\frac{1}{n}\sum_{i=1}^n \Exp{\norm{h^t_i - \nabla f_i(x^*)}^2} + \nu \left(1 - \frac{\gamma \mu}{2}\right)\Exp{\norm{w^t - x^t}^2} + \frac{12 \gamma(\omega + 1) \sigma^2}{n}\\
       &= \left(1 - \frac{\gamma \mu}{2}\right)\left(\frac{1}{2\gamma}\Exp{\norm{x^t - x^*}^2} + \Exp{f(x^{t}) - f(x^{*})} + \kappa \frac{1}{n}\sum_{i=1}^n \Exp{\norm{h^t_i - \nabla f_i(x^*)}^2} + \nu \Exp{\norm{w^t - x^t}^2}\right)\\
       &\qquad + \frac{12 \gamma(\omega + 1) \sigma^2}{n}.
    \end{align*}
    Recursively applying the last inequality and using the assumption $x^0 = w^0,$ one can get that
    \begin{align*}
       \frac{1}{2\gamma}&\Exp{\norm{x^{T} - x^*}^2} + \Exp{f(x^{T}) - f(x^*)} + \kappa \frac{1}{n}\sum_{i=1}^n \Exp{\norm{h^{T}_i - \nabla f_i(x^*)}^2} + \nu \Exp{\norm{w^{T}-x^{T}}^2} \\
       &\leq \left(1 - \frac{\gamma \mu}{2}\right)^T\left(\frac{1}{2\gamma}\Exp{\norm{x^0 - x^*}^2} + \left(f(x^{0}) - f(x^{*})\right) + \kappa \frac{1}{n}\sum_{i=1}^n \norm{h^0_i - \nabla f_i(x^*)}^2\right) \\
       &\qquad + \sum_{i=0}^{T-1}\left(1 - \frac{\gamma \mu}{2}\right)^i \frac{12 \gamma(\omega + 1) \sigma^2}{n} \\
       &\leq \left(1 - \frac{\gamma \mu}{2}\right)^T\left(\frac{1}{2\gamma}\Exp{\norm{x^0 - x^*}^2} + \left(f(x^{0}) - f(x^{*})\right) + \kappa \frac{1}{n}\sum_{i=1}^n \norm{h^0_i - \nabla f_i(x^*)}^2\right) \\
       &\qquad + \frac{24 (\omega + 1) \sigma^2}{n \mu}
    \end{align*}
    Using the nonnegativity of the terms and the bound on $\kappa$, we obtain
    \begin{align*}
       &\frac{1}{2\gamma}\Exp{\norm{x^{T} - x^*}^2} + \Exp{f(x^{T}) - f(x^*)} \\
       &\leq \left(1 - \frac{\gamma \mu}{2}\right)^T\left(\frac{1}{2\gamma}\Exp{\norm{x^0 - x^*}^2} + \left(f(x^{0}) - f(x^{*})\right) + \frac{8 \gamma \omega}{n^2 \beta}\sum_{i=1}^n \norm{h^0_i - \nabla f_i(x^*)}^2\right) \\
       &\qquad + \frac{24 (\omega + 1) \sigma^2}{n \mu}.
    \end{align*}
\end{proof}

\begin{restatable}{theorem}{THEOREMDIANAGENERALCONVEXSTOCHASTIC}
        \label{theorem:diana_general_convex_stochastic}
        Let us consider Algorithm~\ref{algorithm:diana_ef21_p} using stochastic gradients $\widetilde{\nabla} f_i$ instead of the exact gradients $\nabla f_i$ for all $i \in [n]$.
        Let us assume that Assumptions~\ref{ass:lipschitz_constant}, \ref{ass:workers_lipschitz_constant}, \ref{ass:convex} and \ref{ass:stochastic_unbiased_and_variance_bounded} hold, the strong convexity parameter satisfies $\mu = 0,$ $\beta = \frac{1}{\omega + 1}$, $x^0 = w^0,$ and
        \begin{align*}
           &\gamma \leq \min\left\{\frac{n}{160 \omega L_{\max}}, \frac{\alpha}{100 L}\right\}.
        \end{align*}
        Then Algorithm~\ref{algorithm:diana_ef21_p} guarantees the following convergence rate:
        \begin{align*}
           f\left(\frac{1}{T} \sum_{t=1}^{T} x^{t}\right) - f(x^*) 
           &\leq \frac{1}{\gamma T}\norm{x^0 - x^*}^2 + \frac{f(x^{0}) - \nabla f(x^{*})}{T} \\
           &\qquad+ \frac{16 \gamma \omega (\omega + 1)}{T n^2} \sum_{i=1}^n \norm{h^0_i - \nabla f_i(x^*)}^2 + \frac{24\gamma(\omega + 1) \sigma^2}{n}.
        \end{align*}
\end{restatable}

\begin{proof}
    Let us bound \eqref{eq:main_diana_stochastic}:
    \begin{align*}
       &\frac{1}{2\gamma}\Exp{\norm{x^{t+1} - x^*}^2} + \Exp{f(x^{t+1}) - f(x^*)} \\
       &\qquad + \kappa \frac{1}{n}\sum_{i=1}^n \Exp{\norm{h^{t+1}_i - \nabla f_i(x^*)}^2} + \nu \Exp{\norm{w^{t+1}-x^{t+1}}^2} \nonumber\\
       &\leq \frac{1}{2\gamma}\left(1 - \frac{\gamma \mu}{2}\right)\Exp{\norm{x^t - x^*}^2} + \frac{1}{2}\Exp{f(x^{t}) - f(x^{*})} \nonumber\\
       &\quad + \kappa \left(1 - \frac{\gamma \mu}{2}\right)\frac{1}{n}\sum_{i=1}^n \Exp{\norm{h^t_i - \nabla f_i(x^*)}^2} + \nu \left(1 - \frac{\gamma \mu}{2}\right)\Exp{\norm{w^t - x^t}^2} + \frac{12 \gamma(\omega + 1) \sigma^2}{n}\\
       &\leq \frac{1}{2\gamma}\Exp{\norm{x^t - x^*}^2} + \frac{1}{2}\Exp{f(x^{t}) - f(x^{*})} + \kappa \frac{1}{n}\sum_{i=1}^n \Exp{\norm{h^t_i - \nabla f_i(x^*)}^2} \\
       &\qquad + \nu \Exp{\norm{w^t - x^t}^2} + \frac{12 \gamma(\omega + 1) \sigma^2}{n}.
    \end{align*}
    Summing the inequality for $t \in \{0, \dots, T-1\}$ gives
    \begin{align*}
       &\frac{1}{2\gamma}\Exp{\norm{x^{T} - x^*}^2} + \frac{1}{2}\Exp{f(x^{T}) - f(x^*)} + \frac{1}{2} \sum_{t=1}^{T}\Exp{f(x^{t}) - f(x^*)} \\
       &\quad + \kappa \frac{1}{n}\sum_{i=1}^n \Exp{\norm{h^{T}_i - \nabla f_i(x^*)}^2} + \nu \Exp{\norm{w^{T}-x^{T}}^2} \nonumber\\
       &\leq \frac{1}{2\gamma}\norm{x^0 - x^*}^2 + \frac{1}{2}\left(f(x^{0}) - f(x^{*})\right) + \kappa \frac{1}{n}\sum_{i=1}^n \norm{h^0_i - \nabla f_i(x^*)}^2 \\
       &\qquad + \nu \norm{w^0 - x^0}^2 + \frac{12 T \gamma(\omega + 1) \sigma^2}{n}\\
       &\leq \frac{1}{2\gamma}\norm{x^0 - x^*}^2 + \frac{1}{2}\left(f(x^{0}) - f(x^{*})\right) + \frac{8 \gamma \omega}{n^2 \beta} \sum_{i=1}^n \norm{h^0_i - \nabla f_i(x^*)}^2 + \frac{12 T \gamma(\omega + 1) \sigma^2}{n},
    \end{align*}
    where we used the fact that $x^0 = w^0$ and the bound on $\kappa.$
    Using nonnegativity of the terms and convexity, we have 
    \begin{align*}
       f\left(\frac{1}{T} \sum_{t=1}^{T} x^{t}\right) - f(x^*) 
       &\leq \frac{1}{\gamma T}\norm{x^0 - x^*}^2 + \frac{f(x^{0}) - \nabla f(x^{*})}{T} \\
       &\qquad + \frac{16 \gamma \omega}{T n^2 \beta} \sum_{i=1}^n \norm{h^0_i - \nabla f_i(x^*)}^2 + \frac{24\gamma(\omega + 1) \sigma^2}{n}.
    \end{align*}
 \end{proof}

\clearpage
\section{Proofs for \algname{EF21-P + DCGD} in the Convex Case}

As mentioned before,  \algname{EF21-P + DCGD} arises a special case of  \algname{EF21-P + DIANA} if we do not attempt to learn any local gradient shifts $h_i^t$ and instead set them to $0$ throughout. This can be achieved by setting $\beta=0$.

        \begin{algorithm}[H]
        \footnotesize
            \caption{\algname{EF21-P + DCGD}}
            \begin{algorithmic}[1]
            \label{algorithm:dcgd_ef21_p}
            \STATE \textbf{Parameters:} learning rate $\gamma > 0$; initial iterate $x^0 \in \R^d$ {\color{gray}(stored on the server and the workers);} initial iterate shift $w^0 = x^0 \in \R^d$  {\color{gray}(stored on the server and the workers)} 
            \FOR{$t = 0, 1, \dots, T - 1$}
            \FOR{$i = 1, \dots, n$ {\bf in parallel}} 
            \STATE $g_i^t = \cC_i^{D}(\nabla f_i(w^t))$ \hfill {\scriptsize \color{gray} Compress gradient via $\cC_i^{D} \in \mathbb{U}(\omega)$}

            \STATE {\color{blue}Send message $g_i^t$ to the server}
            \ENDFOR
            \STATE $g^t = \frac{1}{n}\sum_{i = 1}^n g_i^t$ \hfill {\scriptsize \color{gray} Compute gradient estimator}
            \STATE $x^{t+1} = x^t - \gamma g^t$ \hfill {\scriptsize \color{gray} Take  gradient-type step}
            \STATE $p^{t+1} = \cC^{P}\left(x^{t+1} - w^t\right)$  \hfill {\scriptsize \color{gray} Compress shifted model on the server via $\cC^{P} \in \mathbb{B}\left(\alpha\right)$}
            \STATE $w^{t+1} = w^t + p^{t+1}$ \hfill {\scriptsize \color{gray} Update model shift}
            \STATE {\color{blue}Broadcast $p^{t+1}$ to all workers}
            \FOR{$i = 1, \dots, n$ {\bf in parallel}}
            \STATE $w^{t+1} = w^{t} + p^{t+1}$ \hfill  {\scriptsize \color{gray} Update model shift}
            \ENDFOR
            \ENDFOR
            \end{algorithmic}
         \end{algorithm}

The proofs in this section almost repeat the proofs from Section~\ref{sec:proofs_diana}.
\begin{theorem}
    \label{theorem:general_dcgd}
    Let us assume that Assumptions~\ref{ass:lipschitz_constant}, \ref{ass:workers_lipschitz_constant} and \ref{ass:convex} hold and choose
    \begin{align*}
        &\gamma \leq \min\left\{\frac{n}{160 \omega L_{\max}}, \frac{\alpha}{100 L}\right\}.
    \end{align*}
     Then Algorithm~\ref{algorithm:dcgd_ef21_p} guarantees that
     \begin{align}
        \frac{1}{2\gamma}&\Exp{\norm{x^{t+1} - x^*}^2} + \Exp{f(x^{t+1}) - f(x^*)} + \nu \Exp{\norm{w^{t+1}-x^{t+1}}^2} \nonumber \\
        &\leq \frac{1}{2\gamma}\left(1 - \frac{\gamma \mu}{2}\right)\Exp{\norm{x^t - x^*}^2} + \frac{1}{2}\Exp{f(x^{t}) - f(x^{*})} \nonumber\\
        &\qquad + \nu \left(1-\frac{\gamma \mu}{2}\right) \Exp{\norm{w^t - x^t}^2}
        + \frac{4 \gamma \omega}{n}\left(\frac{1}{n}\sum_{i=1}^n \norm{\nabla f_i(x^*)}^2\right),
        \label{eq:main_dcgd}
     \end{align}
     where $\nu \leq \frac{32\gamma\omega \widehat{L}^2}{n \alpha} + \frac{16L}{\alpha}.$
 \end{theorem}

 \begin{proof}
    Note that \algname{EF21-P + DCGD} is \algname{EF21-P + DIANA} with $\beta = 0$ and $h^{t}_i = 0$ for all $i \in [n]$ and $t \geq 0.$ Up to \eqref{eq:kappa_nu}, we can reuse the proof of Theorem~\ref{theorem:general_diana} and obtain
    \begin{align*}
        \frac{1}{2\gamma}&\ExpSub{t+1}{\norm{x^{t+1} - x^*}^2} + \ExpSub{t+1}{f(x^{t+1}) - f(x^*)} \\
        &\qquad + \kappa \frac{1}{n}\sum_{i=1}^n \ExpSub{t+1}{\norm{h^{t+1}_i - \nabla f_i(x^*)}^2} + \nu \ExpSub{t+1}{\norm{w^{t+1}-x^{t+1}}^2} \nonumber \\
        &\leq \frac{1}{2\gamma}\left(1 - \frac{\gamma \mu}{2}\right)\norm{x^t - x^*}^2 \nonumber \\
        &\quad+ \left(\frac{8\gamma\omega L_{\max}}{n} + \kappa 4\beta L_{\max} + \nu \left(\frac{8\gamma^2\omega L_{\max}}{n} + \frac{8 \gamma^2 L}{\alpha}\right)\right)\left(f(x^{t}) - f(x^{*})\right) \nonumber\\
        &\quad + \left(\frac{2 \gamma \omega}{n} + \nu \frac{2 \gamma^2 \omega}{n} + \kappa \left(1 - \beta\right)\right)\frac{1}{n}\sum_{i=1}^n \norm{h^t_i - \nabla f_i(x^*)}^2 \nonumber\\
        &\quad + \left(\frac{4\gamma\omega \widehat{L}^2}{n} + 2L + \kappa 2\beta \widehat{L}^2 + \nu \left(1-\frac{\alpha}{4}\right)\right)\norm{w^t - x^t}^2.
     \end{align*}
     Due to $\beta = 0,$ we have
     \begin{align*}
        &\frac{1}{2\gamma}\ExpSub{t+1}{\norm{x^{t+1} - x^*}^2} + \ExpSub{t+1}{f(x^{t+1}) - f(x^*)} \\
        &\qquad + \kappa \frac{1}{n}\sum_{i=1}^n \ExpSub{t+1}{\norm{h^{t+1}_i - \nabla f_i(x^*)}^2} + \nu \ExpSub{t+1}{\norm{w^{t+1}-x^{t+1}}^2} \nonumber \\
        &\leq \frac{1}{2\gamma}\left(1 - \frac{\gamma \mu}{2}\right)\norm{x^t - x^*}^2 \nonumber \\
        &\quad+ \left(\frac{8\gamma\omega L_{\max}}{n} + \nu \left(\frac{8\gamma^2\omega L_{\max}}{n} + \frac{8 \gamma^2 L}{\alpha}\right)\right)\left(f(x^{t}) - f(x^{*})\right) \nonumber\\
        &\quad + \left(\frac{2 \gamma \omega}{n} + \nu \frac{2 \gamma^2 \omega}{n} + \kappa \right)\frac{1}{n}\sum_{i=1}^n \norm{h^t_i - \nabla f_i(x^*)}^2 \nonumber\\
        &\quad + \left(\frac{4\gamma\omega \widehat{L}^2}{n} + 2L + \nu \left(1-\frac{\alpha}{4}\right)\right)\norm{w^t - x^t}^2.
     \end{align*}
     Taking $\kappa = 0$ and $\nu = \frac{32\gamma\omega \widehat{L}^2}{\alpha n} + \frac{16L}{\alpha}$, we obtain
     \begin{align*}
        \frac{1}{2\gamma}&\ExpSub{t+1}{\norm{x^{t+1} - x^*}^2} + \ExpSub{t+1}{f(x^{t+1}) - f(x^*)} + \nu \ExpSub{t+1}{\norm{w^{t+1}-x^{t+1}}^2} \nonumber \\
        &\leq \frac{1}{2\gamma}\left(1 - \frac{\gamma \mu}{2}\right)\norm{x^t - x^*}^2 + \nu \left(1-\frac{\alpha}{8}\right) \norm{w^t - x^t}^2 \\
        &\qquad+ \left(\frac{8\gamma\omega L_{\max}}{n} + \left(\frac{32\gamma\omega \widehat{L}^2}{\alpha n} + \frac{16L}{\alpha}\right) \left(\frac{8\gamma^2\omega L_{\max}}{n} + \frac{8 \gamma^2 L}{\alpha}\right)\right)\left(f(x^{t}) - f(x^{*})\right) \nonumber\\
        &\qquad + \left(\frac{2 \gamma \omega}{n} + \nu \frac{2 \gamma^2 \omega}{n} \right)\frac{1}{n}\sum_{i=1}^n \norm{h^t_i - \nabla f_i(x^*)}^2 \\
        &= \frac{1}{2\gamma}\left(1 - \frac{\gamma \mu}{2}\right)\norm{x^t - x^*}^2 + \nu \left(1-\frac{\alpha}{8}\right) \norm{w^t - x^t}^2 \\
        &\qquad+ \left(\frac{8\gamma\omega L_{\max}}{n} +\frac{256\gamma^3\omega^2 \widehat{L}^2 L_{\max}}{n^2 \alpha} + \frac{256\gamma^3\omega L \widehat{L}^2}{n \alpha^2} + \frac{128 \gamma^2 \omega L L_{\max}}{n \alpha} + \frac{128 \gamma^2 L^2}{\alpha^2}\right)\left(f(x^{t}) - f(x^{*})\right) \nonumber\\
        &\qquad + \left(\frac{2 \gamma \omega}{n} + \nu \frac{2 \gamma^2 \omega}{n} \right)\frac{1}{n}\sum_{i=1}^n \norm{h^t_i - \nabla f_i(x^*)}^2.
     \end{align*}
     Using the assumptions on $\gamma,$ we have
    \begin{align*}
        &\frac{8\gamma\omega L_{\max}}{n} \leq \frac{1}{10}, \\
        &\frac{256\gamma^3\omega^2 \widehat{L}^2 L_{\max}}{n^2 \alpha} \leq \frac{20 \gamma^2\omega \widehat{L}^2}{n \alpha} \leq \frac{1}{10}, \\
        &\frac{128 \gamma^2 \omega L L_{\max}}{n \alpha} \leq \frac{4 \gamma L}{\alpha} \leq \frac{1}{10}, \\
        &\frac{256\gamma^3\omega L \widehat{L}^2}{n \alpha^2} \leq \frac{40\gamma^2\omega \widehat{L}^2}{n \alpha} \leq \frac{1}{10}, \\
        &\frac{128 \gamma^2 L^2}{\alpha^2} \leq \frac{1}{10}.
    \end{align*}
    Considering $\gamma \leq \frac{\alpha}{4\mu},$ we obtain
 \begin{align*}
    \frac{1}{2\gamma}&\ExpSub{t+1}{\norm{x^{t+1} - x^*}^2} + \ExpSub{t+1}{f(x^{t+1}) - f(x^*)} + \nu \ExpSub{t+1}{\norm{w^{t+1}-x^{t+1}}^2} \nonumber \\
    &\leq \frac{1}{2\gamma}\left(1 - \frac{\gamma \mu}{2}\right)\norm{x^t - x^*}^2 + \frac{1}{2}\left(f(x^{t}) - f(x^{*})\right) + \nu \left(1-\frac{\gamma \mu}{2}\right) \norm{w^t - x^t}^2 \\
        &\qquad + \left(\frac{2 \gamma \omega}{n} + \nu \frac{2 \gamma^2 \omega}{n} \right)\frac{1}{n}\sum_{i=1}^n \norm{h^t_i - \nabla f_i(x^*)}^2.
 \end{align*}
 From the assumptions on $\gamma,$ we have $$\frac{2 \gamma \omega}{n} + \nu \frac{2 \gamma^2 \omega}{n} \leq \frac{2 \gamma \omega}{n} + \left(\frac{32\gamma\omega \widehat{L}^2}{\alpha n} + \frac{16L}{\alpha}\right) \frac{2 \gamma^2 \omega}{n} \leq \frac{4 \gamma \omega}{n}$$
 and hence
 \begin{align*}
    \frac{1}{2\gamma}&\ExpSub{t+1}{\norm{x^{t+1} - x^*}^2} + \ExpSub{t+1}{f(x^{t+1}) - f(x^*)} + \nu \ExpSub{t+1}{\norm{w^{t+1}-x^{t+1}}^2} \nonumber \\
    &\leq \frac{1}{2\gamma}\left(1 - \frac{\gamma \mu}{2}\right)\norm{x^t - x^*}^2 + \frac{1}{2}\left(f(x^{t}) - f(x^{*})\right) + \nu \left(1-\frac{\gamma \mu}{2}\right) \norm{w^t - x^t}^2 \\
        &\qquad + \frac{4 \gamma \omega}{n}\frac{1}{n}\sum_{i=1}^n \norm{h^t_i - \nabla f_i(x^*)}^2.
 \end{align*}
 Taking the full expectation, we obtain
 \begin{align*}
    \frac{1}{2\gamma}&\Exp{\norm{x^{t+1} - x^*}^2} + \Exp{f(x^{t+1}) - f(x^*)} + \nu \Exp{\norm{w^{t+1}-x^{t+1}}^2} \nonumber \\
    &\leq \frac{1}{2\gamma}\left(1 - \frac{\gamma \mu}{2}\right)\Exp{\norm{x^t - x^*}^2} + \frac{1}{2}\Exp{f(x^{t}) - f(x^{*})} + \nu \left(1-\frac{\gamma \mu}{2}\right) \Exp{\norm{w^t - x^t}^2} \\
        &\qquad + \frac{4 \gamma \omega}{n}\frac{1}{n}\sum_{i=1}^n \Exp{\norm{h^t_i - \nabla f_i(x^*)}^2}.
 \end{align*}
 It remains to use \eqref{eq:h_t} with $\beta = 0$ to finish the proof of the theorem.

 In the proof, we have the requirement that
 \begin{align*}
    &\gamma \leq \min\left\{\frac{n}{160 \omega L_{\max}}, \frac{\sqrt{n \alpha}}{20\sqrt{\omega} \widehat{L}}, \frac{\alpha}{100 L}\right\}.
\end{align*}
 As in the proof of Theorem~\ref{theorem:general_diana}, using Lemma~\ref{lemma:lipt_constants}, we can simplify it to
 \begin{align*}
    &\gamma \leq \min\left\{\frac{n}{160 \omega L_{\max}}, \frac{\alpha}{100 L}\right\}.
\end{align*}
 \end{proof}

 \begin{theorem}
    \label{theorem:dcgd_general_convex}
    Let us assume that Assumptions~\ref{ass:lipschitz_constant}, \ref{ass:workers_lipschitz_constant} and \ref{ass:convex} hold, the strong convexity parameter satisfies $\mu = 0,$ $x^0 = w^0$ and
    \begin{align*}
       &\gamma \leq \min\left\{\frac{n}{160 \omega L_{\max}}, \frac{\alpha}{100 L}\right\}.
    \end{align*}
    Then Algorithm~\ref{algorithm:dcgd_ef21_p} guarantees that
    \begin{align*}
        f\left(\frac{1}{T} \sum_{t=1}^{T} x^{t}\right) - f(x^*) \leq \frac{1}{\gamma T}\norm{x^0 - x^*}^2 + \frac{f(x^{0}) - \nabla f(x^{*})}{T} + \frac{8 \gamma \omega}{n}\left(\frac{1}{n}\sum_{i=1}^n \norm{\nabla f_i(x^*)}^2\right).
    \end{align*}
 \end{theorem}
 
 \begin{proof}
    Let us bound \eqref{eq:main_dcgd}:
    \begin{align*}
        &\frac{1}{2\gamma}\Exp{\norm{x^{t+1} - x^*}^2} + \Exp{f(x^{t+1}) - f(x^*)} + \nu \Exp{\norm{w^{t+1}-x^{t+1}}^2} \nonumber \\
        &\leq \frac{1}{2\gamma}\left(1 - \frac{\gamma \mu}{2}\right)\Exp{\norm{x^t - x^*}^2} + \frac{1}{2}\Exp{f(x^{t}) - f(x^{*})} + \nu \left(1-\frac{\gamma \mu}{2}\right) \Exp{\norm{w^t - x^t}^2} \\
            &\quad + \frac{4 \gamma \omega}{n}\left(\frac{1}{n}\sum_{i=1}^n \norm{\nabla f_i(x^*)}^2\right) \\
       &\leq \frac{1}{2\gamma}\Exp{\norm{x^t - x^*}^2} + \frac{1}{2}\Exp{f(x^{t}) - f(x^{*})} + \nu \Exp{\norm{w^t - x^t}^2} + \frac{4 \gamma \omega}{n}\left(\frac{1}{n}\sum_{i=1}^n \norm{\nabla f_i(x^*)}^2\right).
    \end{align*}
    We now sum the inequality for $t \in \{0, \dots, T-1\}$ and obtain
    \begin{align*}
       &\frac{1}{2\gamma}\Exp{\norm{x^{T} - x^*}^2} + \frac{1}{2}\Exp{f(x^{T}) - f(x^*)} + \frac{1}{2} \sum_{t=1}^{T}\Exp{f(x^{t}) - f(x^*)} + \nu \Exp{\norm{w^{T}-x^{T}}^2} \\
       &\leq \frac{1}{2\gamma}\norm{x^0 - x^*}^2 + \frac{1}{2}\left(f(x^{0}) - f(x^{*})\right) + \nu \norm{w^0 - x^0}^2 + T \frac{4 \gamma \omega}{n}\left(\frac{1}{n}\sum_{i=1}^n \norm{\nabla f_i(x^*)}^2\right) \\
       &= \frac{1}{2\gamma}\norm{x^0 - x^*}^2 + \frac{1}{2}\left(f(x^{0}) - f(x^{*})\right) + T \frac{4 \gamma \omega}{n}\left(\frac{1}{n}\sum_{i=1}^n \norm{\nabla f_i(x^*)}^2\right).
    \end{align*}
    where we used the assumption $x^0 = w^0.$ Non-negativity of the terms and convexity gives
    \begin{align*}
       f\left(\frac{1}{T} \sum_{t=1}^{T} x^{t}\right) - f(x^*) \leq \frac{1}{\gamma T}\norm{x^0 - x^*}^2 + \frac{f(x^{0}) - \nabla f(x^{*})}{T} + \frac{8 \gamma \omega}{n}\left(\frac{1}{n}\sum_{i=1}^n \norm{\nabla f_i(x^*)}^2\right).
    \end{align*}
 \end{proof}
 
 \begin{theorem}
    \label{theorem:dcgd_strong}
    Let us assume that Assumptions~\ref{ass:lipschitz_constant}, \ref{ass:workers_lipschitz_constant} and \ref{ass:convex} hold, $x^0 = w^0,$ and
    \begin{align*}
       &\gamma \leq \min\left\{\frac{n}{160 \omega L_{\max}}, \frac{\alpha}{100 L}\right\}.
    \end{align*}
    Then Algorithm~\ref{algorithm:dcgd_ef21_p} guarantees that
    \begin{align*}
        &\frac{1}{2\gamma}\Exp{\norm{x^{T} - x^*}^2} + \Exp{f(x^{T}) - f(x^*)} \\
        &\qquad\leq \left(1 - \frac{\gamma \mu}{2}\right)^T\left(\frac{1}{2\gamma}\Exp{\norm{x^0 - x^*}^2} + \left(f(x^{0}) - f(x^{*})\right)\right) + \frac{8 \omega}{n \mu}\left(\frac{1}{n}\sum_{i=1}^n \norm{\nabla f_i(x^*)}^2\right).
     \end{align*}
 \end{theorem}
 
 \begin{proof}
    Using $\gamma \leq \frac{\alpha}{100 L} \leq \frac{1}{\mu},$ let us bound \eqref{eq:main_dcgd}:
    \begin{align*}
       &\frac{1}{2\gamma}\Exp{\norm{x^{t+1} - x^*}^2} + \Exp{f(x^{t+1}) - f(x^*)} + \nu \Exp{\norm{w^{t+1}-x^{t+1}}^2} \nonumber \\
       &\leq \frac{1}{2\gamma}\left(1 - \frac{\gamma \mu}{2}\right)\Exp{\norm{x^t - x^*}^2} + \frac{1}{2}\Exp{f(x^{t}) - f(x^{*})} + \nu \left(1-\frac{\gamma \mu}{2}\right) \Exp{\norm{w^t - x^t}^2} \\
           &\quad + \frac{4 \gamma \omega}{n}\left(\frac{1}{n}\sum_{i=1}^n \norm{\nabla f_i(x^*)}^2\right) \\
           &\leq \frac{1}{2\gamma}\left(1 - \frac{\gamma \mu}{2}\right)\Exp{\norm{x^t - x^*}^2} + \left(1 - \frac{\gamma \mu}{2}\right)\Exp{f(x^{t}) - f(x^{*})} + \nu \left(1-\frac{\gamma \mu}{2}\right) \Exp{\norm{w^t - x^t}^2} \\
           &\quad + \frac{4 \gamma \omega}{n}\left(\frac{1}{n}\sum_{i=1}^n \norm{\nabla f_i(x^*)}^2\right) \\
       &= \left(1 - \frac{\gamma \mu}{2}\right)\left(\frac{1}{2\gamma}\Exp{\norm{x^t - x^*}^2} + \Exp{f(x^{t}) - f(x^{*})} + \nu \Exp{\norm{w^t - x^t}^2}\right) \\
       &\quad+\frac{4 \gamma \omega}{n}\left(\frac{1}{n}\sum_{i=1}^n \norm{\nabla f_i(x^*)}^2\right).
    \end{align*}
    Recursively applying the last inequality and using $x^0 = w^0,$ one obtains
    \begin{align*}
       \frac{1}{2\gamma}&\Exp{\norm{x^{T} - x^*}^2} + \Exp{f(x^{T}) - f(x^*)} + \nu \Exp{\norm{w^{T}-x^{T}}^2} \\
       &\leq \left(1 - \frac{\gamma \mu}{2}\right)^T\left(\frac{1}{2\gamma}\Exp{\norm{x^0 - x^*}^2} + \left(f(x^{0}) - f(x^{*})\right)\right) \\
       &\qquad+ \sum_{i=0}^{T-1}\left(1 - \frac{\gamma \mu}{2}\right)^i\frac{4 \gamma \omega}{n}\left(\frac{1}{n}\sum_{i=1}^n \norm{\nabla f_i(x^*)}^2\right) \\
       &\leq \left(1 - \frac{\gamma \mu}{2}\right)^T\left(\frac{1}{2\gamma}\Exp{\norm{x^0 - x^*}^2} + \left(f(x^{0}) - f(x^{*})\right)\right) \\
       &\qquad+ \sum_{i=0}^{\infty}\left(1 - \frac{\gamma \mu}{2}\right)^i\frac{4 \gamma \omega}{n}\left(\frac{1}{n}\sum_{i=1}^n \norm{\nabla f_i(x^*)}^2\right) \\
       &= \left(1 - \frac{\gamma \mu}{2}\right)^T\left(\frac{1}{2\gamma}\Exp{\norm{x^0 - x^*}^2} + \left(f(x^{0}) - f(x^{*})\right)\right) + \frac{8 \omega}{n \mu}\left(\frac{1}{n}\sum_{i=1}^n \norm{\nabla f_i(x^*)}^2\right) \\
    \end{align*}
    Non-negativity of $\Exp{\norm{w^{T}-x^{T}}^2}$ gives
    \begin{align*}
       &\frac{1}{2\gamma}\Exp{\norm{x^{T} - x^*}^2} + \Exp{f(x^{T}) - f(x^*)} \\
       &\leq \left(1 - \frac{\gamma \mu}{2}\right)^T\left(\frac{1}{2\gamma}\Exp{\norm{x^0 - x^*}^2} + \left(f(x^{0}) - f(x^{*})\right)\right) + \frac{8 \omega}{n \mu}\left(\frac{1}{n}\sum_{i=1}^n \norm{\nabla f_i(x^*)}^2\right).
    \end{align*}
 \end{proof}

\end{document}

%% file: macros.tex
\usepackage{graphicx}
\usepackage{apptools}
\usepackage[flushleft]{threeparttable}
\usepackage{array,booktabs,makecell}
\usepackage{multirow}
\usepackage{nicefrac}
\usepackage{amsthm}
\usepackage{amsmath,amsfonts,bm}
\usepackage{wrapfig}
\usepackage{caption}
\usepackage{subcaption}
\usepackage{siunitx}
\usepackage{thm-restate}
\usepackage{nccmath}

\usepackage{algorithmic}
\usepackage{algorithm}

\usepackage{color}
\usepackage{colortbl}
\definecolor{bgcolor}{rgb}{0.76,0.88,0.50}
\definecolor{bgcolor0}{rgb}{0.93,0.99,1}
\definecolor{bgcolor1}{rgb}{0.8,1,1}
\definecolor{bgcolor2}{rgb}{0.8,1,0.8}
\definecolor{bgcolor3}{rgb}{0.50,0.90,0.50}
\usepackage{tcolorbox}
\usepackage{pifont}
\definecolor{mydarkgreen}{rgb}{39,130,67}
\definecolor{mydarkred}{rgb}{192,25,25}

\newcommand{\norm}[1]{\left\| #1 \right\|}

\newcommand{\inp}[2]{\left\langle#1,#2\right\rangle} 

\newcommand{\R}{\mathbb{R}} 

\newcommand{\Exp}[1]{{\rm E}\left[#1\right]}
\newcommand{\ExpSub}[2]{{\rm E}_{#1}\left[#2\right]}

\newcommand{\cC}{\mathcal{C}}

\newcommand{\cO}{\mathcal{O}}

\theoremstyle{plain}
\newtheorem{theorem}{Theorem}[section]
\newtheorem{proposition}[theorem]{Proposition}
\newtheorem{lemma}[theorem]{Lemma}

\theoremstyle{definition}
\newtheorem{definition}[theorem]{Definition}
\newtheorem{assumption}[theorem]{Assumption}
\theoremstyle{remark}

\newcommand{\eqdef}{:=} 

\makeatletter
\newcommand{\vast}{\bBigg@{4}}